\documentclass[12pt,a4paper,oneside]{article}
\usepackage{a4wide}
\usepackage[T1]{fontenc}
\usepackage[utf8]{inputenc}
\usepackage{amsmath}
\usepackage{times}
\usepackage{graphicx}
\usepackage[final]{pdfpages}
\usepackage{longtable}
\usepackage{adjustbox}
\usepackage{float}
\usepackage{amssymb}
\usepackage{url}
\usepackage{xcolor}
\usepackage{mathtools}
\usepackage{tcolorbox}
\definecolor{pigpink}{HTML}{FDD7E4}
\definecolor{lcyan}{HTML}{E0FFFF}
\definecolor{mint}{HTML}{98FF98}

\usepackage{marginnote}
\usepackage{soul}

\usepackage{amsthm}
\newtheorem{theorem}{Theorem}[section]
\newtheorem{lemma}[theorem]{Lemma}

\newtheorem{definition}[theorem]{Definition}
\newtheorem{remark}[theorem]{Remark}

\newtheorem{algorithm}[theorem]{Algorithm}

\newcommand{\Kt}{\tilde{K}}
\newcommand{\Ft}{\tilde{F}}

\newcommand{\ocal}{\mathcal{O}}
\newcommand{\M}{\mathcal{M}}

\usepackage{indentfirst}

\title{A geometric approach to Phase Response Curves and its numerical computation through the parameterization method}

\author{Alberto P\'erez-Cervera$^{1}$, Tere M-Seara$^{1}$ and Gemma Huguet\\
	\parbox{12.5cm}{
		\small
		\begin{itemize}
			\item[$^1$]
			Departament de Matem\`atiques, Universitat Polit\`ecnica de
			Catalunya, Avda. Diagonal 647, 08028 Barcelona. BGSMATH \\
		\end{itemize}
	}}

\begin{document}
\date{}
\maketitle

\noindent \textbf{Corresponding author:} Alberto P\'erez-Cervera,
\texttt{alberto.perez@upc.edu} \\



\section*{Abstract}

The Phase Response Curve (PRC) is a tool used in neuroscience that measures the phase shift experienced
by an oscillator due to a perturbation applied at different phases of the limit cycle. In this paper we present a new approach to PRCs based on the parameterization method. 
The underlying idea relies on the construction of a periodic system
whose corresponding stroboscopic map has an invariant curve. 
We demonstrate the relationship between the internal dynamics of this
invariant curve and the PRC, which yields a method to numerically compute
the PRCs.  
Moreover, we link the existence properties of this invariant curve as the amplitude of the perturbation is increased with changes in the PRC waveform and with the geometry of isochrons.
The invariant curve and its dynamics will be computed by means of the parameterization method consisting of solving an invariance equation.
We show that the method to compute the PRC can be extended beyond the breakdown of the curve by means of introducing a modified invariance equation.
The method also computes the amplitude response functions (ARCs) which provide information on the 
displacement away from the oscillator due to the effects of the perturbation. 
Finally, we apply the method to several classical models in  neuroscience to illustrate how the results herein 
extend the framework of computation and interpretation of the PRC and ARC for perturbations of large amplitude and not necessarily pulsatile.

\section{Introduction}

Oscillations are ubiquitous in the brain \cite{buzsaki2006rhythms}. From cellular to population level, there exist numerous recordings 
showing periodic activity.
Mathematically, oscillations correspond to attracting limit cycles in the phase space whose dynamics can be described by a phase variable. 
Under generic conditions, the phase can be extended to a neighbourhood of the limit cycle via the concepts of asymptotic phase and isochrons 
\cite{Guckenheimer74, winfree1974patterns}. 
Isochrons are the sets of points in the basin of attraction of a limit cycle whose orbits approach asymptotically the orbit of a given point on the limit cycle. 
We associate to these points the same phase as the base point on the limit cycle. 
The regions outside the basin of attraction are called phaseless sets \cite{Guckenheimer74}.

When the oscillator is perturbed with a transient external stimulus, the trajectory of each point is displaced away from the limit cycle 
and set to the isochron of a different point, thus causing a change in the phase of the oscillation. 
Phase displacements due to perturbations of the oscillator that act at different phases of the limit cycle are described by the so-called phase response curves (PRC) 
\cite{ErmentroutTerman2010, schultheiss2011phase}.
PRCs constitute a useful tool to reduce the dynamics of the oscillator --which can be of high dimension-- to a single equation for the phase.
This approach, based on the phase reduction, has been extensively used to study weakly perturbed
nonlinear oscillators and predict synchronization properties in neuronal networks \cite{canavier2010pulse, hoppensteadt2012weakly}.

PRCs can be measured for arbitrary stimuli, both experimentally and numerically, in individual neurons and in neuronal populations, assuming that there is enough time to allow the perturbed trajectory to relax back to the limit cycle.
For perturbations that are infinitesimally small in duration (pulsatile) and amplitude, one obtains the so called infinitesimal PRC (iPRC).
The iPRC corresponds to the first order approximation of the PRC with respect to the amplitude and it can be easily computed solving the Adjoint equation \cite{ermentrout1991}. 
Perturbations of small amplitude but longer duration are assumed to sum linearly, thus the phase change is obtained by convolving the input waveform with the iPRC. Of course, this approximation fails when the perturbation is strong.

Recently, there has been a large effort to compute isochrons and iPRCs accurately up to high order
\cite{Guillamon2009,huguet2013,mauroy2012, osinga2010continuation}. 
Moreover, the isochrons allow for the control of the phase 
for trajectories away from the limit cycle. One can extend the phase coordinate system to a 
neighbourhood of the limit cycle by incorporating an amplitude variable. This variable is 
transverse to the periodic orbit and controls the ``distance'' to the limit cycle \cite{castejon2013phase, wedgwood2013phase}. 
This coordinate is also known as isostable \cite{wilsonermentrout18,moehliswilsonpre2016}.
It is therefore natural to compute the amplitude response curve (ARC) (or isostable response curve IRC) 
which, analogously to the PRC, provides the shift in amplitude due to a perturbation.

In this paper we present a methodology to compute the PRC for perturbations of large amplitude and not necessarily pulsatile using the
parameterization method. 
The underlying idea of the method is to construct a particular periodic 
perturbation consisting of the repetition of the transient stimulus followed by a resting period when no perturbation acts.
For this periodic system we consider its corresponding stroboscopic map and we prove that, under certain conditions, the map has an invariant curve. 
The core mathematical result of this paper is Theorem~\ref{thm:theorem1}, which gives the existence of the invariant curve and
provides the relationship between the PRC and the internal dynamics of the curve.
To prove the Theorem we use the coordinate system given by the phase and the amplitude variables. In these variables, 
the map is contractive in the amplitude direction and one can apply the results about the existence of invariant curves in \cite{NippS92, NippS2013}. 
Working in the original variables, one can also use theorems on the persistence of normally hyperbolic invariant manifolds with \textit{a posteriori} format  
\cite{BatesLuZeng2008, haroDeLaLlave}. 
That is, one formulates a functional equation for the parameterization of the invariant curve and its internal dynamics.
Then, if there exists an approximate solution of this invariance equation, which satisfies some explicit nondegeneracy conditions, there
is a true solution nearby.
Moreover, these \textit{a posteriori} theorems provide a numerical algorithm to compute the invariant curve and its internal dynamics based on a ``quasi-Newton'' method. 
We will implement this algorithm and compute the PRC using the result in Theorem~\ref{thm:theorem1}. 
We also present an extension of the algorithm to compute the PRC after the breakdown of the invariant curve
(possibly because it loses its normal hyperbolic properties).
In this case, it is possible to write an invariance equation which can be solved approximately using similar algorithms obtaining the PRC and the ARC.

We apply our methodology to some representative examples in the literature, namely the Morris-Lecar model and the Wilson-Cowan equations, 
with a sinusoidal type of stimulus. 
As the amplitude is increased, we detect the breakdown of the curve, which we can relate with the geometry of the isochrons. 
Moreover, we use the modified version of the algorithms to compute the PRC beyond the breakdown of the invariant curve. 
We compare the PRC computed using our methodology  with the 
one computed using the standard approach, showing a good agreement. 
This accuracy is maintained for all the  amplitudes, including the transition from type 0 to type 1 PRC \cite{ glass1988clocks, glass1984discontinuities}, 
which occurs when the perturbation sends points of the limit cycle to the phaseless sets.

The paper is organized as follows: in Section~\ref{sec:math_formalism} we set the mathematical formalism. 
In Section~\ref{sec:method} we state the main result: Theorem~\ref{thm:theorem1}.
In Section~\ref{sec:num_method} we describe the numerical algorithms based on Theorem~\ref{thm:theorem1} and present the extension for the case when the invariant curve does not exist but the PRC can still be computed.
In Section~\ref{sec:num_ex} we present numerical results for some representative examples.
We finish with a discussion in Section~\ref{sec:discussion}. The Appendix contains the algorithms to compute the PRC based on the parameterization method described along the manuscript.

\section{Mathematical Formalism}\label{sec:math_formalism}

Let us consider a smooth  system of ODEs given by
\begin{equation}\label{eq:equacion0}
    \dot{x} = X(x) + Ap(t; A), \quad \quad \quad \quad x \in \mathbb{R}^n,
\end{equation}
where $p(t; A)$ is a function with compact support satisfying $p(t; A)=0$ everywhere except for $0 \leq t \leq T_{pert}$ and
$\max\limits_{t \in \mathbb{R}} |p(t; A)|=1$. Therefore, $A$ determines the amplitude of the perturbation.

We assume that for $A = 0$ (i.e the unperturbed case) system \eqref{eq:equacion0} has a hyperbolic attracting limit cycle $\Gamma_0$ of period $T$
$$
\Gamma_0:=\{\gamma_0(t), t \in [0,T)\},
$$
being $\gamma_0$  a $T$-periodic solution of \eqref{eq:equacion0}.

We will denote by $\psi_A(t;t_0, x)$ the general solution of system \eqref{eq:equacion0}.
As system \eqref{eq:equacion0} is autonomous for $A = 0$, we know that  $\psi_0(t; t_0, x)=\phi_0(t-t_0;x)$, where $\phi_0(t;x)$ is the flow of the unperturbed system.
Moreover, abusing of notation, we will denote by $\phi_A(t;x)=\psi_A(t;0, x)$.

For the unperturbed case, we can define a parameterization $K_0$ for $\Gamma_0$ by means of the phase variable $\theta = \frac{t}{T}$, that is,
\begin{equation}\label{eq:paramK0}
K_0 : \mathbb{T}:=[0,1) \rightarrow \mathbb{R}^n,
\end{equation}
such that $K_0(\theta) = \gamma_0(\theta T)$.
Thus, the dynamics for $\theta$ satisfies
\begin{equation} \label{eq:theta}
\dot{\theta} = 1/T, \quad \quad \quad \text{with solution} \quad \Psi_0(t; \theta_0) = \theta_0 + \frac{t}{T}.
\end{equation}
Observe that $\phi_0(t; K_0(\theta)) = K_0(\theta + \frac{t}{T}) = K_0(\Psi_0(t; \theta))$.

Consider a point $x$ in the basin of attraction $\mathcal{M}$ (stable manifold) of the limit cycle $\Gamma_0$.
Since $\Gamma_0$ is a Normally Hyperbolic Invariant Manifold (NHIM), by NHIM theory (see \cite{Fenichel71, Guckenheimer74, HirschPS77}), 
there exists a unique point on the limit cycle, $K_0(\theta) \in \Gamma_0$,  such that
\begin{equation}\label{eq:nhimIsochron}
d\big(\phi_0(t; x), \phi_0(t; K_0(\theta))\big) \leq Ce^{-\lambda t}, \quad \text{for} \quad t \geq 0,
\end{equation}
where $-\lambda < 0$ is the maximal Lyapunov exponent of $\Gamma_0$.
This property allows us to assign a phase $\theta$ to any point $x \in \mathcal{M}$.
Indeed, the phase function is defined as (see \cite{Guckenheimer74}):
\begin{equation}
\begin{aligned}
\Theta:\mathcal{M} \subset \mathbb{R}^{n} &\to \mathbb{T},\\
x &\mapsto \Theta(x) = \theta,
\end{aligned}\label{eq:thetaFunction}
\end{equation}
such that equation \eqref{eq:nhimIsochron} is satisfied. 
The sets of points with the same asymptotic phase are called isochrons 
\cite{winfree1974patterns}. 
The sets of points where the asymptotic phase is not defined are called phaseless sets \cite{Guckenheimer74}.
Clearly, for an attracting Normally Hyperbolic Invariant Manifold the phaseless sets are contained in  $\mathbb{R}^n\setminus \mathcal{M}$.

In this context, the PRC (see \cite{ErmentroutTerman2010}) for the perturbation $Ap(t; A)$ in \eqref{eq:equacion0} is defined as
\begin{equation}\label{eq:prcDefinition}
PRC(\theta, A) = \Theta\big(\phi_A(T_{pert};K_0(\theta))\big) - \Theta\big(\phi_0(T_{pert};K_0(\theta))\big),
\end{equation}
if $\phi_A(T_{pert};K_0(\theta))\in \mathcal{M}$. For the rest of the manuscript, abusing notation, we will denote by $\mathcal{M}$ a bounded neighbourhood of the periodic orbit $\Gamma_0$ 
such that $\bar{\mathcal{M}}$ is contained in the basin of attraction of $\Gamma_0$.

If we denote by
\begin{equation}\label{eq:xpertthetapert}
x_{pert}:=\phi_A(T_{pert};K_0(\theta)), \quad \theta_{pert}:= \Theta (x_{pert}),
\end{equation}
from \eqref{eq:theta} and \eqref{eq:prcDefinition} we have that
\begin{equation}\label{eq:prc1}
PRC (\theta,A) = \theta_{pert} - \left ( \theta + \frac{T_{pert}}{T} \right ).
\end{equation}
Moreover, since $\phi_A(T_{pert}+t; K_0(\theta))=\phi_0(t; x_{pert})$ and using the definition of the phase function given in \eqref{eq:thetaFunction}, we have that
\begin{equation}\label{eq:prc2}
PRC(\theta, A) = \Theta\big(\phi_A(T_{pert}+t; K_0(\theta))\big) - \Theta\big(\phi_0(T_{pert}+t; K_0(\theta))\big),
\end{equation}
for all $t \geq 0$.

The usual way to compute the PRC either experimentally or numerically is the following. First, one looks for the time  $t_1 \gg T_{pert}$ at which some
$x_i$-coordinate of the perturbed trajectory $\phi_A(t;K_0(\theta))$ reaches its maximum value after the perturbation is turned off. Then,
one compares time $t_1$ with the time $t_0$ which is closest to $t_1$ at which the unperturbed trajectory $\phi_0(t;K_0(\theta))$ reaches its maximum
(see Figure~\ref{fig:pertConsidered2}).
Finally, the PRC is given approximately by
\begin{equation}\label{eq:prc3}
\Delta \theta = \frac{t_1-t_0}{T}.
\end{equation}

This approach (that we will refer to as the \textit{standard method}) provides a good approximation of the PRC if the time to relax back to the oscillator $\Gamma_0$ is short either because there is a strong contraction
(the maximal Lyapunov exponent $-\lambda$ in \eqref{eq:nhimIsochron} is sufficiently negative) or the perturbation is weak ($A \ll 1$ in \eqref{eq:equacion0}).
Otherwise, one should wait several periods ($kT$, $k \in \mathbb{N}$ sufficiently large) before computing the phase difference.

In the next Sections we present theoretical and numerical results based on the parameterization method that yield novel algorithms to compute the PRC.

\begin{figure}[H]
\centering
{\includegraphics[width=105mm]{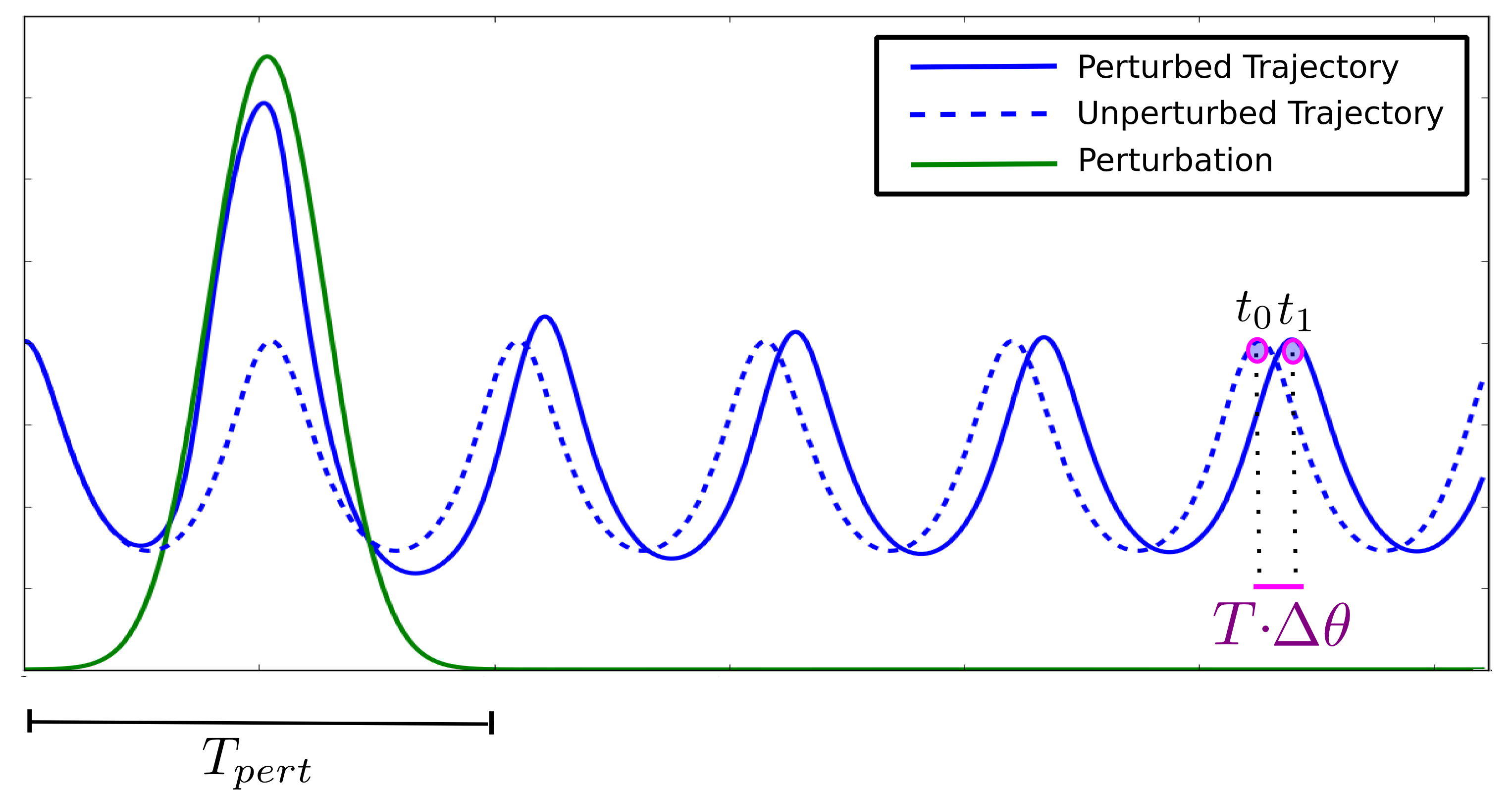}}
\caption{After the perturbation is turned off ($t > T_{pert}$), the trajectories relax back to the limit cycle
and a phase shift $\Delta \theta$ is experienced.} \label{fig:pertConsidered2}
\end{figure}

\section{Stroboscopic approach to compute the PRC by means of the parameterization method. Theoretical results.}\label{sec:method}

The perturbation $p(t;A)$ in \eqref{eq:equacion0} is not periodic.
However, we will introduce a periodic perturbation $\bar{p}(t;A)$ of period
$T':=T_{pert}+T_{rel}$, with $T_{rel} \gg T_{pert}$
which coincides with $p(t;A)$ for $0 \leq t \leq T'$.
Then, we consider the $T'$-periodic system
\begin{equation}\label{eq:equacion1}
    \dot{x} = X(x) + A\bar{p}(t;A), \quad \quad \quad \quad x \in \mathbb{R}^n,
\end{equation}
whose solutions coincide with the solutions of \eqref{eq:equacion0} for $0 \leq t \leq T'$. 
Since $\bar{p}(t;A)$ is periodic, we can define the stroboscopic map given by the flow of \eqref{eq:equacion1} at time $T'$ starting at $t=0$, i.e.
\begin{eqnarray}\label{eq:poincareMap}
F_{A}:\mathbb{R}^{n} &\to& \mathbb{R}^{n}, \notag \\
 x &\to& F_{A}(x) = \phi_A(T';x)=\phi_A(T_{pert} + T_{rel};x).
\end{eqnarray}
Using this approach, the formula for the PRC given in \eqref{eq:prc2} for $t=T_{rel}$ writes as
\begin{equation}\label{eq:prcStrDefinition}
PRC(\theta, A) = \Theta\big(F_A(K_0(\theta))\big) - \Theta\big(F_0(K_0(\theta))\big).
\end{equation}
Note that for $A=0$, one has
\begin{equation}\label{eq:invariancek0}
F_0(K_0(\theta))=\phi_0(T';K_0(\theta))=K_0(\theta + T'/T),
\end{equation}
and therefore
\begin{equation}\label{eq:gamma0}
\Gamma_0=\{ K_0(\theta), \, \theta \in [0,1)\}
\end{equation}
is an invariant curve of the map $F_0$. Moreover, by \eqref{eq:nhimIsochron} we have that for any $x \in \mathcal{M}$
\[
|F_0(x)-F_0(K_0(\theta))| \leq C e^{-\lambda T'},
\]
where $\theta=\Theta(x)$. Therefore, $\Gamma_0$ is a normally hyperbolic attracting invariant curve of the map $F_0$.

Let us recall here, following \cite{Fenichel71,Fenichel74,HirschPS77},
the definition of normally hyperbolic attracting invariant curve adapted to our problem.
\begin{definition}\label{def:nhim}
Let $F:M\to M$ a $C^r$ diffeomorphism on a $C^r$-differentiable manifold $M$.
Assume that there exists a manifold $\Gamma\subseteq M$ that
is  invariant  for $F$. We say that   $\Gamma \subset M$ is a hyperbolic attracting manifold if there
exists a splitting of the tangent bundle $TM$ into $DF$-invariant
sub-bundles, i.e,
\[TM= E^s\oplus T\Gamma,\]
and constants $C>0$ and 
\begin{equation}\label{rates0}
0<\lambda_+ < \eta^ {-1} \le 1,
\end{equation}
such that for all  $x\in\Gamma$ we have
\begin{equation}\label{characterization}
\begin{split}
v\in E^s_x  &\Leftrightarrow \|DF^k(x) \, v\|\leq C\lambda_+^k\|v\|,\,  \textrm{ for all } k\geq 0,\\
v\in T_x\Gamma &\Leftrightarrow \|DF^k (x) \, v \|\leq C\eta^{|k|} \|v\|,\, \textrm{ for all } k \in \mathbb{Z}.\\
\end{split}
\end{equation}
\end{definition}

For our problem, following \cite{Guillamon2009, huguet2013},
we can differentiate  the invariance equation \eqref{eq:invariancek0} of $\Gamma_0$  obtaining
\begin{equation}\label{eq:nhim1}
DF_0(K_0(\theta))DK_0(\theta)=DK_0 \left (\theta+\frac{T'}{T} \right ),
\end{equation}
and, for $n=2$, in  \cite{Guillamon2009, huguet2013} it is shown that  there exists $N(\theta)$ such that
\begin{equation}\label{eq:nhim2}
DF_0(K_0(\theta))N(\theta)=e^{-\lambda T'}N \left(\theta+\frac{T'}{T}\right).
\end{equation}
Therefore, for any $x=K_0(\theta) \in \Gamma_0$, there is a splitting $\mathbb{R}^2=<DK_0(\theta)>\oplus<N(\theta)>$ and rates 
$\lambda_+=e^{-\lambda T'}$ and $\eta=1$, satisfying \eqref{characterization} thus showing that $\Gamma_0$ is a normally hyperbolic attracting manifold. This result can be generalized to $n>2$ using the information provided by the variational equations along the periodic orbit $\Gamma_0$ (see Remark \ref{rem:gen_n}).

Next, we present the main result of this paper which provides the existence of an invariant curve $\Gamma_A$ of the stroboscopic map $F_A$ \eqref{eq:poincareMap} which is $\mathcal{O} (Ae^{-\lambda T_{rel}})$-close to $\Gamma_0$ and relates its
internal dynamics with the PRC of $\Gamma_0$ in system \eqref{eq:equacion0} (see Fig. \ref{fig:methodSketch}).
The proof of this Theorem is given in Section~\ref{sec:proofs}.

\begin{theorem}\label{thm:theorem1}
Consider the stroboscopic map of the $T'$-periodic system \eqref{eq:equacion1} defined in \eqref{eq:poincareMap} with $T'=T_{pert}+T_{rel}$ 
and let $\Gamma_0$ be the normally hyperbolic attracting invariant curve of the map $F_0$, parameterized by $K_0$, such that
\[F_0 \circ K_0 = K_0 \circ f_0,\]
where $f_0(\theta)=\theta+T'/T$. 

Consider $A>0$. Assume that $A$ is small or $A=\ocal(1)$ and the following hypothesis are satisfied:
\begin{itemize}
 \item[\textbf{H1}] 
 $\phi_A(T_{pert};x) \in \mathcal{M}$ for any $x \in \Gamma_0$, 
 \item[\textbf{H2}] 
 The function $PRC(\theta,A)+\theta$ is a monotone function,
 \item[\textbf{H3}] 
 $T_{rel}$ is sufficiently large,
\end{itemize}
then, there exists an invariant curve $\Gamma_A$ of the map $F_A$.
Moreover, there exist a parameterization $K_A$ of $\Gamma_A$ and a periodic function $f_A$ such that
\begin{itemize}
 \item[]
\begin{equation}\label{eq:invariance_teo}
F_A(K_A(\theta))=K_A(f_A(\theta)),
\end{equation}
\item[]
\[
 K_A(\theta)=K_0(\theta)+\mathcal{O} (Ae^{-\lambda T_{rel}}),
\]
\item[]
\[
PRC(\theta,A)= f_A(\theta) - f_0(\theta) + \mathcal{O} (Ae^{-\lambda T_{rel}}).
\]
\end{itemize}
\end{theorem}

\begin{figure}[H]
	\centering
	{\includegraphics[width=105mm]{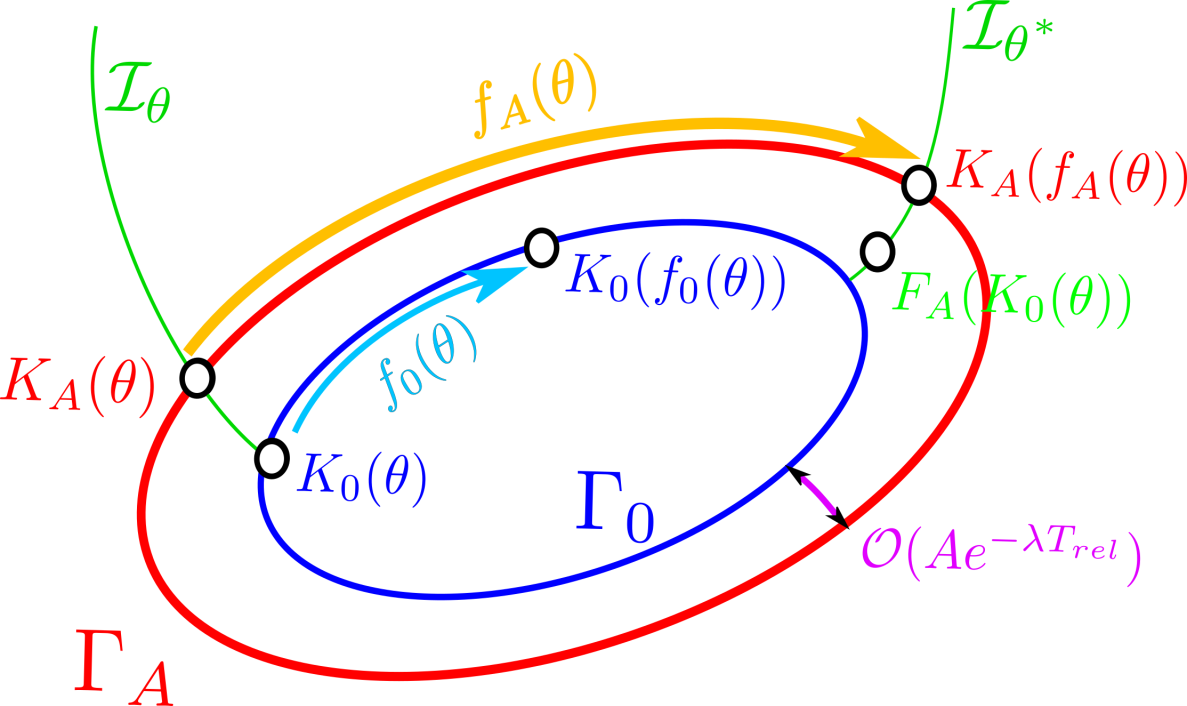}}
	\caption{Sketch of the results of Theorem \ref{thm:theorem1}. The perturbation acting on a point $K_0(\theta) \in \Gamma_0$ for a time $T' = T_{pert} + T_{rel}$ displaces it to a point $F_A(K_0(\theta))$. In Theorem \ref{thm:theorem1} we show that the phase difference between the perturbed and unperturbed trajectories is given up to an error $\mathcal{O}(A e^{-\lambda T_{rel}})$ by the difference between the internal dynamics $f_A(\theta)$ on $\Gamma_A$ and $f_0(\theta)$ on $\Gamma_0$. For the sake of clarity we have located the points $K_0(\theta)$ and $K_A(\theta)$ (resp. $F_A(K_0(\theta))$ and $K_A(f_A(\theta))$) on the same isochron $\mathcal{I}_{\theta}$ (resp. $\mathcal{I}_{\theta^*}$, where  $\theta^*:=f_A(\theta)$) although they are $\mathcal{O}(A e^{-\lambda T_{rel}})$-close.} \label{fig:methodSketch}
\end{figure}

\subsection{Proof of Theorem~\ref{thm:theorem1}}\label{sec:proofs}
\subsubsection{The case $A$ small}

In this Section we prove Theorem~\ref{thm:theorem1} for $A$ small. We first prove the
following lemma which shows that the map $F_A$ has an invariant curve $\Gamma_A$ which is
$\mathcal{O}(Ae^{-\lambda T_{rel}})$-close to $\Gamma_0$.

\begin{lemma}\label{thm:lemma2}
Consider the stroboscopic map of the $T'$-periodic system \eqref{eq:equacion1} defined in \eqref{eq:poincareMap}
and let $\Gamma_0$ be the normally hyperbolic invariant curve of the map $F_0$,
parameterized by $K_0$ (see \eqref{eq:paramK0}).
Then, for $A$ small enough there exists an invariant curve $\Gamma_A$
of the map $F_A$. 
Moreover, there exist a parameterization $K_A:\mathbb{T} \rightarrow \mathbb{R}^n$
and a periodic function $f_A: \mathbb{T} \rightarrow \mathbb{T}$ satisfying the invariance equation
\[
F_A(K_A(\theta))=K_A(f_A(\theta)),
\]
such that $K_A(\theta)$ satisfies
\begin{equation}\label{kamenysk0}
|K_A(\theta) - K_0(\theta)| = \mathcal{O}(A e^{-\lambda T_{rel}}),
\end{equation}
where
$-\lambda<0$ is the maximal Lyapunov exponent of $\Gamma_0$.
\end{lemma}

\begin{proof}
When $n=2$, since, by \eqref{eq:nhim1} and \eqref{eq:nhim2}, $\Gamma_0$ is a normally hyperbolic attracting invariant manifold  of $F_0$,
the existence of the invariant curve for $A$ small enough  follows from Fenichel's Theorem \cite{Fenichel71, Fenichel74}.
We will perform the rest of the proof for $n=2$ but it can be easily generalized to arbitrary $n$ (see Remark~\ref{rem:gen_n}). Using the results in \cite{Guillamon2009} (see \cite{Cabre2005,castelli2015} for higher dimensions) we can describe a point $(x,y) \in \mathcal{M}$ in terms of the so called
phase-amplitude variables. More precisely, consider the change of coordinates 

\begin{equation}
\begin{aligned}
K: \Omega \subset \mathbb{T} \times \mathbb{R} & \rightarrow \mathcal{M} \subset \mathbb{R}^2\\
(\theta, \sigma) & \rightarrow K(\theta, \sigma)=(x,y),
\end{aligned}\label{eq:paramK}
\end{equation}
where $\Omega:=\mathbb{T} \times U$ and $U \subset \mathbb{R}$, such that system~\eqref{eq:equacion1} for $A=0$,
expressed in the variables $(\theta,\sigma)$, has the following form
\begin{equation}
\begin{aligned}
\dot{\theta} & = \frac{1}{T},\\
\dot{\sigma} & = -\lambda\sigma.
\end{aligned}\label{eq:variables_prc}
\end{equation}

Moreover, system~\eqref{eq:equacion1} for $A \neq 0$ small enough, expressed in the variables 
$(\theta,\sigma)$, writes as the $T'$-periodic system
\begin{equation}
\begin{aligned}
\dot{\theta} & = \frac{1}{T} + \mathcal{O}(A),\\
\dot{\sigma} & = - \lambda \sigma + \mathcal{O}(A),
\end{aligned}\label{eq:variables_prc_A}
\end{equation}
and we will denote by $\Psi_A(t;t_0,\theta,\sigma)$ the general solution of \eqref{eq:variables_prc_A}.

Consider now the stroboscopic map $F_{A}$ in the variables $(\theta,\sigma)$, i.e.
$\Ft_A: \Omega \rightarrow \Omega$, such that $\Ft_A=K^{-1} \circ F_A \circ K$.
We have

\begin{equation}\label{eq:Fat}
\Ft_A(\theta,\sigma)=\Psi_A(T'; 0, \theta, \sigma)= \left ( \theta_{pert} + \frac{T_{rel}}{T}, \sigma_{pert}e^{- \lambda T_{rel}} \right ),
\end{equation}
where $K(\theta_{pert},\sigma_{pert})=\phi_A(T_{pert};K(\theta,\sigma))\in \M$. For $A$ small enough we have
\[
(\theta_{pert},\sigma_{pert})=\Psi_A(T_{pert}; 0, \theta, \sigma)
= \left ( \theta + \frac{T_{pert}}{T} + \mathcal{O}(A), \sigma e^{- \lambda T_{pert}} + \mathcal{O}(A) \right ).
\]
In conclusion,
\begin{equation}\label{eq:strobts}
\Ft_A(\theta,\sigma)= \left ( \theta + \frac{T'}{T} + \mathcal{O}(A), \sigma e^{- \lambda T'} + \mathcal{O}(Ae^{-\lambda T_{rel}}) \right)
=\Ft_0 + \mathcal{O}(A).
\end{equation}

The unperturbed invariant curve $\Gamma_0$ in the variables $(\theta, \sigma)$ is given by
\[
\tilde{\Gamma}_0=\{(\theta,\sigma) \, | \, \theta \in \mathbb{T}, \sigma=0\}.
\]
Therefore, by Fenichel's Theorem, for $A \neq 0$ small enough, there exists a function
\begin{equation}
\begin{aligned}
S_A: \mathbb{T} & \rightarrow \mathbb{R}\\
\theta & \rightarrow S_A(\theta),
\end{aligned}\label{eq:functionS_A}
\end{equation}
such that $S_A(\theta)=\mathcal{O}(A)$ and the perturbed invariant curve for $\Ft_A$ is given by
\begin{equation}\label{eq:gammatilde_A}
\tilde{\Gamma}_A = \{(\theta, \sigma) \, | \, \theta \in \mathbb{T}, \sigma=S_A(\theta) \}.
\end{equation}
Analogously, $\tilde{K}_A(\theta)=(\theta,S_A(\theta))$ is a parameterization of the invariant curve $\tilde{\Gamma}_A$.
Hence, using the invariance property, we have
\[
\begin{array}{rcl}
\Ft_A(\Kt_{A}(\theta)) & = &\Ft_A(\theta,S_{A}(\theta)) \\
&=& \left( \Ft_A^1(\theta,S_{A}(\theta)), \Ft_A^2(\theta,S_{A}(\theta)) \right) \\
&=& \left(\Ft_A^1(\theta,S_{A}(\theta)), S_{A}\left(\Ft_A^1(\theta,S_{A}(\theta))\right) \right) \\
&=&\Kt_{A}(\Ft_A^1(\theta,S_{A}(\theta))) \\
&:=&\Kt_{A}(f_{A}(\theta)),
\end{array}
\]
where the internal dynamics is given by
\[
f_A(\theta)=\Ft_A^1(\theta,S_{A}(\theta)),
\]
and $\Ft^1_A$ and $\Ft^2_A$ correspond to  the $\theta$ and $\sigma$ component of $\Ft_A$, respectively.

Using the invariance property of $\tilde {\Gamma}_A$ and  the expression~\eqref{eq:strobts} for the stroboscopic map $\Ft_A$, we obtain:
\[
S_A \left ( \theta + \frac{T'}{T} + \mathcal{O}(A) \right ) = S_A(\theta)e^{- \lambda T'} + \mathcal{O}(A e^{- \lambda T_{rel}}), \quad \forall \theta
\in \mathbb{T}.
\]
Therefore, since $S_A(\theta)=\mathcal{O}(A)$ and $T'=T_{pert} + T_{rel}$ we get an improved bound for $S_A(\theta)$
\[
S_A(\theta)=\mathcal{O}(Ae^{- \lambda T_{rel}}).
\]
Using the change of variables $(x,y) = K(\theta, \sigma)$ given 
in \eqref{eq:paramK} we can return to the original variables. Then, assuming $(x,y) \in \Gamma_{A}\subset \mathcal{M}$, if $A$ is small enough, one has
\begin{equation}\label{eq:paramoriginal}
(x,y)=K(\theta,\sigma)=K(\theta,S_{A}(\theta))=K \circ \Kt_{A}(\theta)=:K_{A}(\theta).
\end{equation}
Thus, the invariant curve can be parameterized by $K_{A}$ and
\[
F_A (K_{A} (\theta))= F_A \circ K \circ \Kt_{A} (\theta) = K \circ \Ft_A \circ \Kt_{A} (\theta) = K \circ \Kt_{A} (f_{A}(\theta)) = K_{A} (f_{A}(\theta)),
\]
that is, the internal dynamics over the invariant curve $\Gamma_A$ is the same for both parameterizations.
Therefore,
\begin{equation}\label{eq:kamenysk0}
\begin{split}
|K_A (\theta) - K_0 (\theta)| &=  |K \circ \tilde{K}_A(\theta) - K \circ \tilde{K}_0(\theta)| \\
&\leq \sup_{(\theta,\sigma) \in \bar{\Omega}}|DK(\theta,\sigma)| |\tilde{K}_A(\theta) - \tilde{K}_0(\theta)| \\
& \leq \bar C |S_A(\theta)| \leq C Ae^{-\lambda T_{rel}},
\end{split}
\end{equation}
where $C$ is a constant independent of $T_{rel}$ and $A$.
\end{proof}

\begin{remark}\label{rem:gen_n}
Notice that the proof can be generalized to any $n>2$, just considering $\sigma=(\sigma_1,\ldots,\sigma_{n-1}) \in \mathbb{R}^{n-1}$ and
\[
\dot{\sigma}= M \sigma,
\]
where $M$ is the real canonical form of the projection onto the stable subspace of the monodromy matrix of the first variational equation along the periodic orbit:
\[
 \dot x=DX(\gamma_0(t))x.
\]

The proof can be  derived analogously using that $\sigma(t)=\sigma(0)e^{Mt}$ and
$|\sigma_0 e^{Mt}| < |\sigma_0| e^ {-\lambda t}$, where $-\lambda <0$ is the maximal Lyapunov exponent of $\Gamma_0$.
\end{remark}

\noindent \textbf{End of the proof of Theorem~\ref{thm:theorem1} for $A$ small}\\
To finish the proof of  Theorem~\ref{thm:theorem1} we need to show that the internal dynamics
$f_A$ in $\Gamma_A$ is close to the PRC of $\Gamma_0$ of system \eqref{eq:equacion0}.

Consider the parameterization $K_A$ of the invariant curve $\Gamma_A$ given in Lemma~\ref{thm:lemma2}, we have
\begin{equation*}
K_A(f_A(\theta)) = F_A(K_A (\theta)) = F_A(K_0 (\theta)) + F_A (K_A (\theta)) - F_A(K_0 (\theta)). \\
\end{equation*}

Assuming that $\sup\limits_{x \in \bar{\mathcal{M}}}$ $|DF_A| \leq C$ and using that 
$|K_A(\theta) - K_0(\theta)| = \mathcal{O}(A e^{-\lambda T_{rel}})$ (see Lemma~\ref{thm:lemma2}),
we have
\begin{equation}
F_A(K_0(\theta)) = K_A(f_A(\theta)) + \mathcal{O} (Ae^{-\lambda T_{rel}}).
\end{equation}
Moreover, using the formula for the PRC given in \eqref{eq:prcStrDefinition}, we have
\begin{equation}
\begin{array}{rcl}
PRC(\theta, A) &=& \Theta[F_A(K_0(\theta))] - \Theta[F_0(K_0(\theta))] \\
&=& \Theta[K_A(f_A(\theta)) + \mathcal{O} (Ae^{-\lambda T_{rel}})] - \Theta[K_0(f_0(\theta))] \\
&=& \Theta[K_A(f_A(\theta))] - \Theta[K_0(f_0(\theta))] \\
&& + \enskip \Theta [K_A(f_A(\theta)) + \mathcal{O} (Ae^{-\lambda T_{rel}})] - \Theta[K_A(f_A(\theta))] \\
&=& f_A(\theta) - f_0(\theta) + \Theta[K_A(f_A(\theta)) + \mathcal{O} (Ae^{-\lambda T_{rel}})] - \Theta[K_A(f_A(\theta))].
\end{array}
\end{equation}
Now using that $\sup\limits_{x \in \bar{\mathcal{M}}} |\nabla \Theta| \leq C$ 
we have
\[PRC(\theta, A)=f_A(\theta) - f_0(\theta) + \mathcal{O}(Ae^{-\lambda T_{rel}}).\]

\subsubsection{The case $A=\mathcal{O}(1)$}\label{sec:proofsA}

To prove Theorem~\ref{thm:theorem1} for $A=\mathcal{O}(1)$ one can use the results in \cite{BatesLuZeng2008}, which state that
if a map has an approximately invariant manifold which is approximately normally hyperbolic, then the map has a true
invariant manifold nearby.

Due to the strong attracting properties of the invariant curve $\Gamma_0$, it is straightforward to see that  $\Gamma_0$ 
is approximately invariant for the map $F_A$, even if $A=\mathcal{O}(1)$.

Consider the intermediate map 
\begin{equation}\label{eq:Fpert}
F_{pert}(x)=\phi_A(T_{pert};x),
\end{equation}
we will use the hypothesis \textbf{H1} that states that $F_{pert}$ maps the curve  $\Gamma_0$ into its basin of 
attraction $\mathcal{M}$. \\
Then, given a point $x=K_0(\theta)\in \Gamma_0$, if $x_{pert} =F_{pert}(x)=\phi_A(T_{pert};x) \in \mathcal{M}$, (see \eqref{eq:xpertthetapert}),
by equation \eqref{eq:nhimIsochron}, there exists a point $K_0 (\theta_{pert}) \in \Gamma_0$ such that, for $t\ge 0$ 
\begin{equation}\label{eq:equacion2}
\left | F_A(K_0(\theta))-K_0 \left ( \theta_{pert}+\frac{T_{rel}}{T} \right ) \right|= \left |\phi_0(T_{rel}; x_{pert}) - \phi_0 \left(T_{rel}; K_0(\theta_{pert}) \right) \right|
\leq C e^{-\lambda T_{rel}}.
\end{equation}
Using the formula for the PRC given in \eqref{eq:prc1} we have that
\begin{equation}\label{algo}
\theta_{pert} + \frac{T_{rel}}{T} 
=PRC(\theta,A) + \theta +\frac{T_{pert}}{T} + \frac{T_{rel}}{T}
=PRC(\theta,A) + f_0(\theta) ,
\end{equation}
where $f_0(\theta)=\theta+T'/T$.
Hence, defining
\begin{equation}\label{eq:fbara}
\bar{f}_A(\theta):=PRC(\theta,A) + f_0(\theta),
\end{equation}
expression \eqref{eq:equacion2} reads as
\begin{equation}\label{eq:asimptotics}
|F_A(K_0(\theta)) - K_0(\bar{f}_A(\theta))| \leq C e^{-\lambda T_{rel}}.
\end{equation}
In other words, the curve $\Gamma_0$ with inner dynamics $\bar {f}_A$ is approximately invariant for the map $F_A$
with an error $\mathcal{O}(e^{-\lambda T_{rel}})$ that can be made as small as we want taking $T_{rel}$ large enough. To apply the results in \cite{BatesLuZeng2008} one needs to show that $\Gamma_0$ is approximately normally hyperbolic for $F_A$. That is, for each point $x \in \Gamma_0$ there exists a decomposition $\Gamma_{0,x}=\Gamma_{0,x}^c \oplus \Gamma_{0,x}^s$,
with $\Gamma_{0,x}^c$ being an approximation of the tangent space to $\Gamma_0$ at $x$, such that
\begin{itemize} 
\item   
The splitting is approximately invariant under the linearised map $DF_A$, 
\item 
$DF_A(x)|_{\Gamma_0^s}$ expands and does so to a greater rate
than $DF_A(x)|_{\Gamma_0^c}$ does. 
\end{itemize}
Again, we will consider the $2$-dimensional case, but results can be generalized to arbitrary dimension (see Remark~\ref{rem:gen_n}).
Using the change of variables $K$ introduced in \eqref{eq:paramK} the map $F_A$ satisfies (see equation~\eqref{eq:Fat}): 
\begin{equation}\label{eq:kathetasigma}
F_A(K(\theta,\sigma))=K \left ( \theta_{pert} + \frac{T_{rel}}{T}, \sigma_{pert}e^{- \lambda T_{rel}} \right ),
\end{equation}
where $K(\theta_{pert}, \sigma_{pert})=F_{pert}(K(\theta,\sigma))$ (see \eqref{eq:Fpert}). 
Notice that  $\theta_{pert}$ and $\sigma_{pert}$ are correctly defined as long as 
$F_{pert}(K(\theta,\sigma)) \in \mathcal{M}$, which is satisfied for points $(\theta,0)$ on the invariant curve $\Gamma_0$
by hypothesis \textbf{H1}, and therefore in a small neighbourhood of $\Gamma_0$. 
Taking derivatives with respect to $\theta$ and $\sigma$ in expression~\eqref{eq:kathetasigma}, we have 
\begin{equation}\label{eq:derivadeskthetasigma}
\begin{array}{rcl}
DF_A(K(\theta,\sigma))D_{\theta}K(\theta,\sigma) &= &
D_{\theta}K(\theta_{pert} + \frac{T_{rel}}{T}, \sigma_{pert}e^{- \lambda T_{rel}}) \frac{\partial \theta_{pert}}{\partial \theta}  \\
&  & + D_{\sigma}K(\theta_{pert} + \frac{T_{rel}}{T}, \sigma_{pert}e^{- \lambda T_{rel}}) e^{- \lambda T_{rel}} \frac{\partial \sigma_{pert}}{\partial \theta},  \\
DF_A(K(\theta,\sigma))D_{\sigma}K(\theta,\sigma) &= &
D_{\theta}K(\theta_{pert} + \frac{T_{rel}}{T}, \sigma_{pert}e^{- \lambda T_{rel}}) \frac{\partial \theta_{pert}}{\partial \sigma}  \\
 & &+ D_{\sigma}K(\theta_{pert} + \frac{T_{rel}}{T}, \sigma_{pert}e^{- \lambda T_{rel}}) e^{- \lambda T_{rel}} \frac{\partial \sigma_{pert}}{\partial \sigma}.  \\
\end{array}
\end{equation}
Evaluating the above expression on the points $(\theta,0)$, we have 
\begin{equation}
\begin{array}{rcl}
DF_A(K_0(\theta))D K_0(\theta) &= &
D_{\theta}K(\bar f_A(\theta), \sigma_{pert}(\theta,0)e^{- \lambda T_{rel}}) \frac{\partial \theta_{pert}}{\partial \theta} (\theta,0)\\
&  & + D_{\sigma}K(\bar f_A(\theta), \sigma_{pert}(\theta,0)e^{- \lambda T_{rel}}) e^{- \lambda T_{rel}} \frac{\partial \sigma_{pert}}{\partial \theta}  (\theta,0),\\
DF_A(K_0(\theta))K_1(\theta) &= & D_{\theta}K(\bar f_A(\theta), \sigma_{pert}(\theta,0)e^{- \lambda T_{rel}}) \frac{\partial \theta_{pert}}{\partial \sigma} (\theta,0) \\
& &+ D_{\sigma}K(\bar f_A(\theta), \sigma_{pert}(\theta,0)e^{- \lambda T_{rel}}) e^{- \lambda T_{rel}} \frac{\partial \sigma_{pert}}{\partial \sigma} (\theta,0), \\
\end{array}
\end{equation}
where $\bar{f}_A$ is defined in \eqref{eq:fbara} (see also \eqref{algo}) and
\begin{equation}\label{eq:k1}
K_1(\theta):=D_{\sigma} K(\theta,\sigma)|_{\sigma=0}. 
\end{equation}
Now, we Taylor expand the function $K(\theta,\sigma)$ around $\sigma=0$, and obtain
\begin{equation}
\begin{array}{rcl}
DF_A(K_0(\theta))DK_0(\theta) &= & \left[ DK_0(\bar{f}_A(\theta))+ e^{-\lambda T_{rel}}\sigma_{pert}(\theta,0)DK_1(\bar{f}_A(\theta))\right]
\frac{\partial \theta_{pert}}{\partial \theta}   (\theta,0)\\
&&+ K_1(\bar{f}_A(\theta)) e^{-\lambda T_{rel}}  \frac{\partial \sigma_{pert}}{\partial \theta} (\theta,0) +\mathcal{O}(e^{-2\lambda T_{rel}}), \\
DF_A(K_0(\theta))K_1(\theta) &= & \left[ DK_0(\bar{f}_A(\theta))+ e^{-\lambda T_{rel}} \sigma_{pert}(\theta,0)DK_1(\bar{f}_A(\theta))\right]
\frac{\partial \theta_{pert}}{\partial \sigma} (\theta,0) \\
&&+ K_1(\bar{f}_A(\theta)) e^{- \lambda T_{rel}}  \frac{\partial \sigma_{pert}}{\partial \sigma}(\theta,0)  +\mathcal{O}(e^{-2\lambda T_{rel}}).\\
\end{array}
\end{equation}
Moreover, as the functions $\theta_{pert}(\theta, \sigma)$ and $\sigma_{pert}(\theta, \sigma)$ are smooth functions
at the points $(\theta,0)$ we can ensure that the error terms are uniform with respect to $\theta \in \mathbb{T}$.
Let us now define
\[
Z(\theta)= \frac{\partial \theta_{pert}}{\partial \sigma} (\theta,0) DK_0(\theta) -\frac{\partial \theta_{pert}}{\partial \theta} (\theta,0) K_1(\theta),
\]
straightforward computations give
\[
DF_A(K_0(\theta))Z(\theta)= \mathcal{O}(e^{-{\lambda T_{rel}}}).
\]
Therefore, 
calling $\varepsilon=e^{-{\lambda T_{rel}}}$, we have
\begin{equation}\label{eq:equation43}
\begin{array}{rcl}
DF_A(K_0(\theta))DK_0(\theta)&=&\Lambda_T(\theta)DK_0(\bar {f}_A(\theta))+\mathcal{O}(\varepsilon),\\
DF_A(K_0(\theta))Z(\theta)&=&
\mathcal{O}(\varepsilon),
\end{array}
\end{equation}
with
\begin{equation}\label{eq:lambda_tn}
\begin{array}{rcl}
\Lambda_T(\theta)&=&\frac{\partial \theta_{pert}}{\partial \theta} (\theta,0),\\
\end{array}
\end{equation}
and as long as
\begin{equation}\label{eq:nhcond}
\frac{\partial \theta_{pert}}{\partial \theta} (\theta,0) \ne 0, \ \theta \in \mathbb{T},
\end{equation}
which is guaranteed by hypothesis \textbf{H2}, one can produce an iteration procedure to construct an approximate splitting by which 
$\Gamma_0$ becomes approximately Normally Hyperbolic (see Definition~\ref{def:nhim}). Then, we apply the results in \cite{BatesLuZeng2008}, which yield that $F_A$  will have an invariant curve $\Gamma_A$ near $\Gamma_0$.

A more direct argument consists in considering the map $F_A$ in the variables $(\theta,\sigma)$ in 
\eqref{eq:paramK}, denoted by $\tilde{F}_A$ in \eqref{eq:Fat}, and apply the results in \cite{NippS92} (see also \cite{NippS2013}) to this map.
Thanks to hypothesis \textbf{H1}, one can consider a neighbourhood of 
$\Gamma_0$ where the change of variables $(x,y)=K(\theta,\sigma)$ is  defined, and therefore the map $\tilde F_A$ is
a smooth diffeomorphism
$$
\tilde F_A : \mathcal{D}_\rho : =\mathbb{T}\times \mathcal{I}_\rho \to \mathbb{T}\times \mathbb{R},
$$
where $\mathcal{I}_\rho=\{\sigma \in \mathbb{R}, \ |\sigma|\le \rho \}$, with $\rho>0$ small, and has the form
$$
\tilde F_A(\theta,\sigma)=\left( \begin{array}{c}f_0(\theta)+\hat f(\theta,\sigma)\\g(\theta,\sigma)\end{array}\right),
$$                                   
where 
$$
f_0(\theta)=\theta_{pert}(\theta,0)+\frac{T_{rel}}{T},\
\hat f(\theta,\sigma)=\theta_{pert}(\theta,\sigma)-\theta_{pert}(\theta,0), \
g(\theta, \sigma)=\sigma_{pert}(\theta,\sigma)e^{-\lambda T_{rel}}.
$$
Hypothesis \textbf{H2} ensures that $f_0$ is a smooth diffeomorphism (and therefore invertible). Taking $T_{rel}$ large enough, the map $\tilde F_A$ strongly contracts in the $\sigma$
direction. Moreover, for $(\theta,\sigma)\in \mathcal{D}_\rho$, we have
$$
\biggl|\frac{\partial \hat f}{\partial \theta} \biggr| \le L_{11}, \
\biggl|\frac{\partial \hat f}{\partial \sigma} \biggr| \le L_{12},  \
\biggl| \frac{\partial g}{\partial \theta} \biggr| \le L_{21}, \
\biggl| \frac{\partial g}{\partial \sigma} \biggr| \le L_{22}, \
$$
where $L_{11}, L_{12} = \mathcal{O}(\rho)$, $L_{21},L_{22}= \mathcal{O}(e^{-\lambda T_{rel}})$ can be made small by taking $\rho$ small and $T_{rel}$ large.
One can then apply Theorem 3 in \cite{NippS92}, which give, for $T_{rel}$ large enough (hypothesis \textbf{H3}), the existence
of the invariant curve in the form \eqref{eq:gammatilde_A}, where the function $S_A$ must satisfy
$$
\tilde{F}_A^2(\theta,S_A(\theta))=S_A(\tilde{F}_A^1(\theta,S_A(\theta))),
$$
and $S_A=\mathcal{O}(e^{-\lambda T_{rel}})$. Again, $\Ft^1_A$ and $\Ft^2_A$ correspond to  the $\theta$ and $\sigma$ components of $\Ft_A$, respectively.

Returning to the original variables $x=K(\theta,\sigma)$ defined in \eqref{eq:paramK} and using that 
$\Gamma_{A}\subset \mathcal{M}$, we obtain the parameterization $K_A$ of $\Gamma_A$ as in \eqref{eq:paramoriginal}.
Moreover, once we have bounded the size of $S_A$, an analogous reasoning to \eqref{eq:kamenysk0},  gives
\begin{equation}\label{kamenysk0agran}
|K_A(\theta)-K_0(\theta)|\le C e^{- T_{rel}}.
\end{equation}

\noindent \textbf{End of the Proof of Theorem~\ref{thm:theorem1} for $A=\mathcal{O}(1)$}\\
To finish the proof of  Theorem~\ref{thm:theorem1} we need to show that the internal dynamics
$f_A$ in $\Gamma_A$ is close to the PRC of $\Gamma_0$ for system \eqref{eq:equacion0}.

This can be done analogously to the case $A$ small using \eqref{kamenysk0agran} instead of \eqref{kamenysk0} arriving to
\begin{equation}\label{eq:ffbar}
PRC(\theta, A)=f_A(\theta) - f_0(\theta) + \mathcal{O}(e^{-\lambda T_{rel}}).
\end{equation}
This step concludes the proof.

\section{Computation of the PRC by means of the parameterization method}\label{sec:num_method}

Theorem~\ref{thm:theorem1} establishes that the PRC can be obtained from the dynamics $f_A$ of the stroboscopic map
$F_A$ on the invariant curve $\Gamma_A$ . 
This allows us to take advantage of the existing algorithms based on the parameterization method \cite{Cabre2005, Haro2016}
to compute the parameterization of the invariant curve $K_A$ as well as its internal dynamics $f_A$. The algorithms are based on a Newton-like method to solve the invariance
equation \eqref{eq:invariance_teo} for the unknowns $K_A$ and $f_A$. Indeed, given an approximation of the parameterization of the invariant curve $K_A$ and its internal dynamics $f_A$, 
the method provides improved solutions that solve the invariance equation up to an error which is quadratic with respect to the initial one at each step. Moreover, 
the method requires to compute alongside the invariant normal bundle of the invariant curve $N(\theta)$ and its linearised dynamics $\Lambda(\theta)$. 
In order to make the paper self-contained, the algorithms are reviewed in detail in Appendix~\ref{sec:num_alg}.

\subsection{Computation of the PRC beyond the existence of the invariant curve}\label{sec:beyond}

The results of Theorem \ref{thm:theorem1}
rely on the computation of an invariant curve for the
stroboscopic map $F_A$ of a system with an `artificially' constructed periodic perturbation (see Eq. \eqref{eq:equacion1}).
In some cases, as we will see in the numerical examples presented in Section~\ref{sec:num_ex}, the invariant curve $\Gamma_A$ 
does not exist. 
This situation can happen if $F_{pert}(\Gamma_0)$ is not in the basin of attraction 
$\mathcal{M}$ of the limit cycle $\Gamma_0$ (breaking hypothesis \textbf{H1}), or if
$\theta_{pert}(\theta,0)$ has a critical point   $\theta^*$ and therefore $d \theta_{pert}/d\theta(\theta^*,0)=0$ (breaking hypothesis \textbf{H2}). 
However, when the hypothesis \textbf{H2} 
fails, it is possible to design an algorithm based on the parameterization method \cite{canadell2016newton}, to compute the PRC 
with enough accuracy
by means of solving an approximate invariance equation.

Using \eqref{eq:kathetasigma} with $\sigma=0$ we have
\[
F_A(K_0(\theta)) = K \left (\theta_{pert}(\theta,0)+ \frac{T_{rel}}{T},\sigma_{pert}(\theta,0)e^{-\lambda T_{rel}} \right )
=K(\bar{f}_A(\theta), \bar{C}_A(\theta)),
\]
where $\bar{f}_A(\theta)$ is given in \eqref{eq:fbara} and 
\begin{equation}\label{eq:Ca}
\bar{C}_A(\theta):=\sigma_{pert}(\theta,0)e^{-\lambda T_{rel}}.
\end{equation}
Taylor expanding $K(\theta,\sigma)$ with respect to $\sigma$ we obtain
\begin{equation}\label{eq:expandingK}
F_A(K_0(\theta)) = K_0 (\bar{f}_A(\theta)) +
\ocal (e^ {- \lambda T_{rel}}).
\end{equation}

Of course, expression \eqref{eq:expandingK} is only valid if $F_{pert}(\Gamma_0)\in \mathcal{M}$ (hypothesis \textbf{H1}), but we do not impose that 
$\Gamma_0$ is approximately normally hyperbolic. 
Nevertheless, we will use the ideas in the algorithms reviewed in Appendix~\ref{sec:num_alg} 
and we will design a quasi-Newton method to compute a function $g_A$ that satisfies 
\begin{equation}\label{eq:invariance_mod0}
F_A(K_0 (\theta)) - K_0 (g_A (\theta))=E(\theta),
\end{equation}
where the error $E$ can not be smaller than the terms $\ocal(e^{-\lambda T_{rel}})$ that we have dropped.

Assuming that $g_A$ satisfies equation \eqref{eq:invariance_mod0}, we look for an improved solution
$\hat{g}_A(\theta) = g_A(\theta) + \Delta g_A(\theta)$ such that $\hat{g}_A$ solves the approximate invariance
equation up to an error which is
quadratic in $E$. Thus, if we linearise about $g_A$ we have
\begin{equation}\label{eq:linearize0}
\begin{split}
F_A(K_0(\theta)) - K_0(\hat{g}_A(\theta)) & = F_A(K_0(\theta)) - K_0(g_A(\theta)) - DK_0(g_A(\theta))\Delta g_A(\theta) +  \ocal (|\Delta g_A|^2) \\
&= E(\theta) - DK_0(g_A(\theta)) \Delta g_A(\theta) +  \ocal (|\Delta g_A|^2). \\
\end{split}
\end{equation}
Therefore, we look for $\Delta g_A$ satisfying the equation
\[E(\theta) = DK_0(g_A(\theta)) \Delta g_A (\theta), \]
which provides
\begin{equation}\label{eq:deltaf0}
\Delta g_A (\theta) = \frac{<DK_0 (g_A(\theta)),E(\theta)>}{<DK_0 (g_A(\theta)), DK_0(g_A(\theta))>},
\end{equation}
where $<\cdot,\cdot>$ denotes the dot product.

\begin{remark}
Notice that expression~\eqref{eq:deltaf0} corresponds to the projection of $E$ onto the tangent direction $DK_0$,
thus obtaining $\Delta g_A$ in the same way as $\Delta f_A$ in Algorithm~\ref{alg:algorithm2} (see Appendix \ref{sec:num_alg}).
\end{remark}

The algorithm to compute the PRC is then:
\begin{algorithm}\label{alg:algorithm_new0}
\textbf{Computation of the PRC}.
Given a parameterization of the limit cycle  $K_0(\theta)$ and an approximate solution of equation 
\eqref{eq:invariance_mod0} $g_A(\theta)$, perform the following operations:
\begin{enumerate}
\item Compute $E(\theta) = F_A(K_0(\theta)) - K_0(g_A(\theta))$.
\item Compute $DK_0(g_A(\theta))$.
\item Compute $\Delta g_A = \frac{<DK_0 (g_A(\theta)),E(\theta)>}{<DK_0 (g_A(\theta)), DK_0(g_A(\theta))>}$.
\item Set $g_A(\theta) \leftarrow g_A(\theta)+\Delta g_A(\theta)$.
\item Repeat steps 1-4 until the error $E$ is smaller than the established tolerance. Then $PRC (\theta, A) = g_A(\theta) - \left(\theta + T'/T \right) $.
\end{enumerate}
\end{algorithm}

In Section~\ref{sec:num_ex}, we apply  Algorithm~\ref{alg:algorithm_new0} to several examples, illustrating
the convergence of the method and the good agreement of the results with the standard method.

\subsection{Computation of the PRC and ARC}\label{sec:sec42}

In the previous Section we have used that $K_0$ satisfies equation~\eqref{eq:expandingK}. Notice that we can be more precise 
and include the exact expression of the terms of $\ocal(e^{-\lambda T_{rel}})$, that is,
\[
F_A(K_0(\theta)) = K_0 (\bar{f}_A(\theta)) + K_1 (\bar{f}_A(\theta)) \bar{C}_A(\theta) + \ocal (e^ {-2 \lambda T_{rel}}),
\]
with $K_1(\theta)$ as in \eqref{eq:k1}.

We already know that the function $\bar{f}_A(\theta)$ provides the PRC through the relation \eqref{eq:fbara}. 
We would like to emphasize here the role of $\bar{C}_A(\theta)$ defined in \eqref{eq:Ca}. 
The analogous curve to the PRC for the amplitude $\sigma$ is known as the Amplitude 
Response Curve \cite{castejon2013phase, wilson2015extending} (ARC), and is given by $ARC (\theta)=\sigma_{pert}(\theta,0)$, that is, the value of 
$\sigma$ at $x_{pert}$ (see \eqref{eq:xpertthetapert}). 
Therefore, since $\sigma_{pert}(\theta,0)=\bar{C}_A(\theta)e^{\lambda T_{rel}}$, the function $\bar{C}_A(\theta)$ provides the ARC through the expression
$ARC(\theta, A)= \bar{C}_A(\theta)e^{\lambda T_{rel}}$.

As in the previous Section, it is possible to design a quasi-Newton method to compute the functions $g_A$ and $C_A$ 
that satisfy 
\begin{equation}\label{eq:invariance_mod}
F_A(K_0(\theta)) - K_0 (g_A(\theta)) - C_A(\theta) K_1 (g_A(\theta))=E(\theta),
\end{equation}
where the error $E$ will not be smaller than the terms of order $\ocal(e^{-2\lambda T_{rel}})$ that
we have dropped.

Proceeding as in the previous Section, we assume that $g_A$ and $C_A$ satisfy equation \eqref{eq:invariance_mod}
and we look for improved solutions,
$\hat{g}_A(\theta) = g_A(\theta) + \Delta g_A(\theta)$ and $\hat{C}_A(\theta) = C_A(\theta) +
\Delta C_A(\theta)$, such that $\hat{g}_A$ and
$\hat{C}_A$ solve the approximate invariance equation up to an error which is
quadratic in $E$. 
Thus, if we linearise about $g_A$ and $C_A$ we have
\begin{equation}
\begin{split}
F_A(K_0(\theta)) & - K_0(\hat{g}_A(\theta)) - K_1(\hat{g}_A(\theta)) \hat{C}_A(\theta) \\
= \, & F_A(K_0(\theta)) - K_0(g_A(\theta)) - DK_0(g_A(\theta))\Delta g_A(\theta) -  K_1(g_A (\theta)) C_A(\theta) \\
& - DK_1(g_A(\theta)) \Delta g_A (\theta) C_A(\theta) - K_1(g_A(\theta)) \Delta C_A(\theta) +
\ocal (\Delta^2, e^ {-2\lambda T_{rel}}) \\
= \, & E(\theta) - DK_0(g_A(\theta)) \Delta g_A(\theta) - DK_1(g_A(\theta)) \Delta g_A (\theta) C_A(\theta) \\
& - \enskip K_1(g_A(\theta)) \Delta C_A(\theta) + \ocal (\Delta^2). \\
\end{split}
\end{equation}
Hence, we are left with the following equation for $\Delta g_A$ and $\Delta C_A$,
\begin{equation}\label{eq:newton_full}
E(\theta) = [DK_0(g_A(\theta)) + DK_1 (g_A (\theta)) C_A(\theta)] \Delta g_A(\theta)  + K_1(g_A(\theta)) \Delta C_A(\theta).
\end{equation}

Therefore, the unknown $\Delta g_A$ corresponds to the projection of the error $E$ onto the direction 
$R:=DK_0 \circ g_A + C_A \cdot DK_1 \circ g_A$ and
$\Delta C_A(\theta)$ corresponds to the projection of $E$ onto the $K_1$ direction.
Of course, $DK_0$ and $K_1$ are transversal since
$K_1$ is tangent to the isochrons of the unperturbed limit cycle, which are always transversal to the limit cycle.
Since $C_A=\ocal(e^ {-\lambda T_{rel}})$, assuming
that $T_{rel}$ is large enough, we can always guarantee that $R$ and $K_1$ are transversal. Therefore,
multiplying \eqref{eq:newton_full} by $K_1^{\bot} (g_A(\theta))$, we have
\begin{equation}\label{eq:deltaf}
\Delta g_A(\theta) = \frac{ < K_1^{\bot}(g_A(\theta)), E(\theta) >}{<K_1^{\bot}(g_A(\theta)),R(\theta)>},
\end{equation}
whereas multiplying by $R(\theta)^{\bot}$, we obtain
\begin{equation}
\Delta C_A(\theta) = \frac{<R^{\bot}(\theta),E(\theta)>}{<R^{\bot}(\theta),K_1(g_A(\theta))>},
\end{equation}
where $<\cdot,\cdot>$ denotes the dot product.

\begin{remark}
Notice that $C_A(\theta)=\ocal(e^{-\lambda T_{rel}})$, and if we disregard 
the terms $\ocal(e^{-\lambda T_{rel}})$ in expression \eqref{eq:deltaf}, we obtain
\begin{equation}\label{eq:solve_deltafa}
\Delta g_A(\theta) = \frac{ < K_1^{\bot}(g_A(\theta)), E(\theta) >}{<K_1^{\bot}(g_A(\theta)),DK_0(g_A(\theta))>},
\end{equation}
which is equivalent to the expression obtained in \eqref{eq:deltaf0}. Indeed, 
in this case the vectors $E(\theta)$ and $DK_0(\theta)$ have the same direction, and
expression~\eqref{eq:deltaf0} can be replaced by 
\[
\Delta g_A (\theta) = \frac{<v,E(\theta)>}{<v, DK_0(g_A(\theta))>},
\]
where $v$ can be any vector as long as it is not perpendicular to $DK_0(g_A(\theta))$. 
\end{remark}

Thus, the algorithm to compute the PRC and the ARC is:
\begin{algorithm}\label{alg:algorithm_new}
\textbf{Computation of the PRC and the ARC}.
Given a parameterization of the limit cycle $K_0(\theta)$, the tangent vector to the isochrons of the 
limit cycle $K_1(\theta)$, and approximate solutions of equation 
\eqref{eq:invariance_mod} $g_A(\theta)$ and $C_A(\theta)$, perform the following operations:
\begin{enumerate}
\item Compute $E(\theta) = F_A(K_0(\theta)) - K_0(g_A(\theta)) - C_A(\theta) K_1 (g_A(\theta))$.
\item Compute $R(\theta)=DK_0(g_A(\theta)) + DK_1 (g_A (\theta)) C_A(\theta)$.
\item Compute $\Delta g_A(\theta) = \frac{ < K_1^{\bot}(g_A(\theta)), E(\theta) >}{<K_1^{\bot}(g_A(\theta)),R(\theta)>}$.
\item Compute $\Delta C_A(\theta) = \frac{<R^{\bot}(\theta),E(\theta)>}{<R^{\bot}(\theta),K_1(g_A(\theta))>}$.
\item Set $g_A(\theta) \leftarrow g_A(\theta)+\Delta g_A(\theta)$.
\item Set $C_A(\theta) \leftarrow C_A(\theta)+\Delta C_A(\theta)$.
\item Repeat steps 1-6 until the error $E$ is smaller than the established tolerance. Then, 
\[PRC(\theta,A)=g_A(\theta)-(\theta + T'/T)\],
and
\[ARC(\theta,A)=C_A(\theta)e^{\lambda T_{rel}}.\]
\end{enumerate}
\end{algorithm}

\section{Numerical Examples}\label{sec:num_ex}

In this Section we apply the algorithms based on the parameterization method introduced in 
Section~\ref{sec:num_method} to compute the PRC to some relevant models in neuroscience, namely the Wilson-Cowan model \cite{wilson1972} and the Morris-Lecar model \cite{morris1981}.
We will use the same perturbation for both models:
\[
p(t) = \sin ^6 \left ( \frac{\pi t}{T_{pert}} \right ),
\]
for $0 \leq t \leq T_{pert}$ and $T_{pert}=10$.
The value of $T_{rel}$ is different for each model and its value is indicated with the 
other parameters of the model in Tables \ref{table:WCprms} and \ref{tab:MLprms}, respectively.
 
In order to validate the algorithms we will compare the results obtained using the parameterization method with the \emph{standard method}
(see formula \eqref{eq:prc3}).

\textbf{The Wilson-Cowan model.}
The Wilson-Cowan model describes the behaviour of a coupled network of excitatory and inhibitory neurons.
The perturbed model has the form (see  \cite{wilson1972}):
\begin{equation}\label{eq:WCsys}
\begin{split}
 \dot{E} &= -E + S_e(aE - bI + P + Ap(t)),\\
 \dot{I} &= -I + S_i(cE - dI + Q),
\end{split}
\end{equation}
where the variables $E$ and $I$ are the firing rate activity of the excitatory and inhibitory populations, respectively, and
\begin{equation}
S_{k}(x) = \frac{1}{1+ e^{-a_k(x-\theta_k)}}, \quad \textrm{for } k=e,i,
\end{equation}
is the response function.

We consider two sets of parameters, for which the system displays a limit cycle.
For the first set of parameters the limit cycle is
born from a Hopf bifurcation, and for the second one from a saddle-node on invariant circle (SNIC) bifurcation \cite{borisyuk1992}.
We refer to them as WC-Hopf and WC-SNIC, respectively.
Some parameters of the model are common to both cases, namely $a=13$, $b=12$, $c=6$, $d=3$, $a_e=1.3$, $a_i=2$, $\theta_e = 4$, $\theta_i = 1.5$.
Parameters $(P,Q)$ for each set are given in Table~\ref{table:WCprms}, together with the period $T$, 
the characteristic exponent $-\lambda$ of each periodic orbit, and the relaxation time $T_{rel}$ of the perturbation.

We compute the PRC for the limit cycle of the Wilson-Cowan model for different values of the amplitude $A$.
In Figs. \ref{fig:prcsHopfCase0} and \ref{fig:prcsSnicCase} we show the comparison between the PRCs computed using 
the standard method and the parameterization method for the WC-Hopf and the WC-SNIC, respectively. We remark the good agreement between both methods.

We also show the invariant curve $\Gamma_A$, the 
internal dynamics $f_A$ and the derivative of $f_A$ in Fig.
\ref{fig:morePrcsHopfCase} for the WC-Hopf and in Fig.~\ref{fig:morePrcsSnicCase} for the WC-SNIC.
Notice that as $A$ increases, the shape of the PRC shows a sudden increase for certain phase values
(see panel A in Figs. \ref{fig:morePrcsHopfCase} and \ref{fig:morePrcsSnicCase}).
A more detailed discussion about this phenomenon will be given in Section~\ref{sec:largeA}.

\begin{table}
	\begin{adjustbox}{minipage=.35\textwidth,valign=t}
\begin{center}{\begin{tabular}{ccc}
		\hline
		Parameter         & Hopf & SNIC \\
		\hline
		P & 2.5 & 1.45 \\
		Q & 0 & -0.75 \\
		\hline
		T & 5.26 & 13.62 \\
		$T_{rel}$ & 15T & 6T \\
		$-\lambda$ & -0.157 & -0.66 \\
		\hline
	\end{tabular}}
\end{center}
	\end{adjustbox}\hfill
	\begin{adjustbox}{minipage=[t]{.5\textwidth},valign=t}
		\setlength{\abovecaptionskip}{0pt}%
{\caption{$(P,Q)$ parameter values for the Wilson-Cowan model close to the corresponding type
		of bifurcation. For the indicated parameter values and $A=0$, system \eqref{eq:WCsys} has a stable limit cycle of period $T$ and characteristic exponent $-\lambda$.}
	\label{table:WCprms}}%
	\end{adjustbox}	
\end{table}

\begin{figure}[H]
\centering
{\includegraphics[width=80mm]{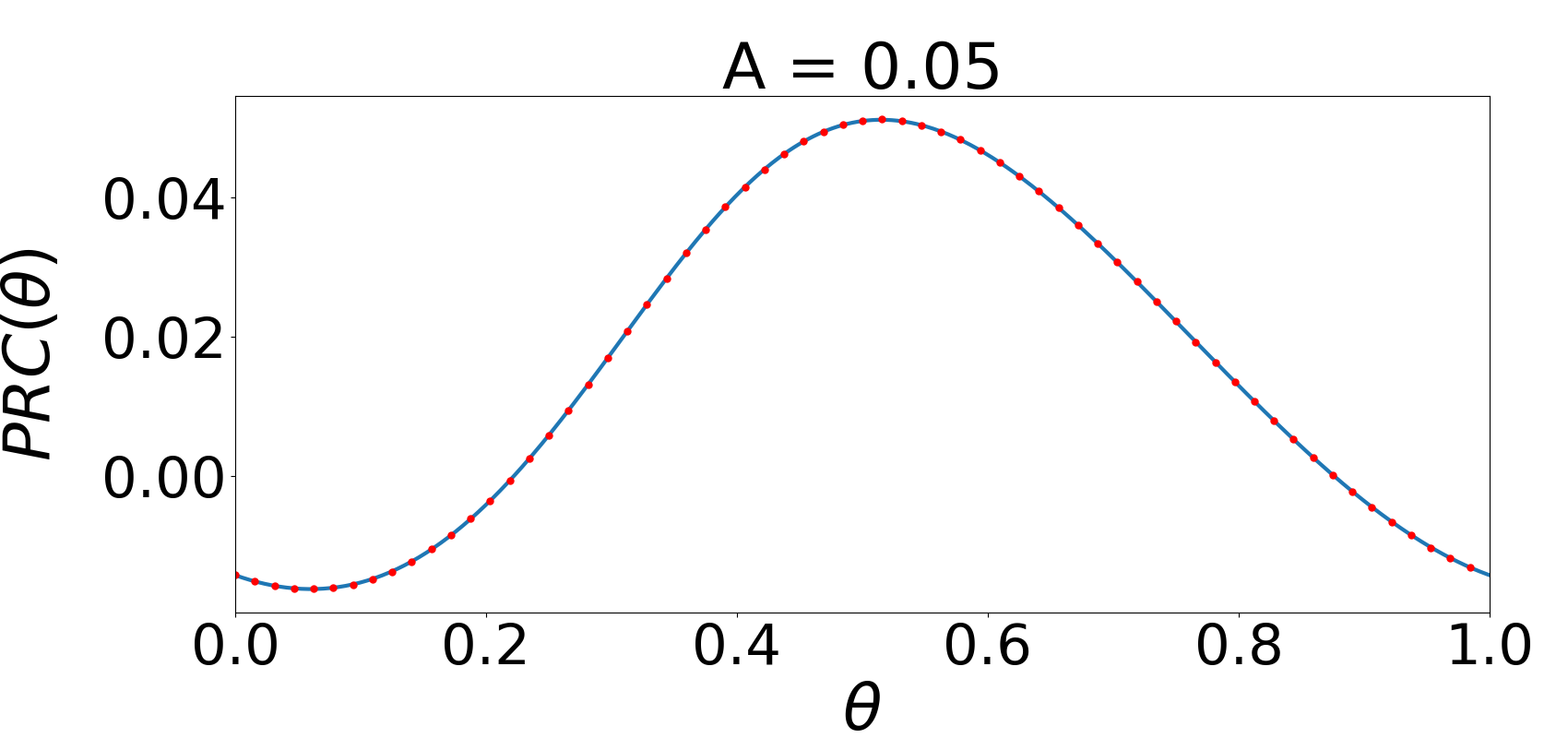}}
{\includegraphics[width=80mm]{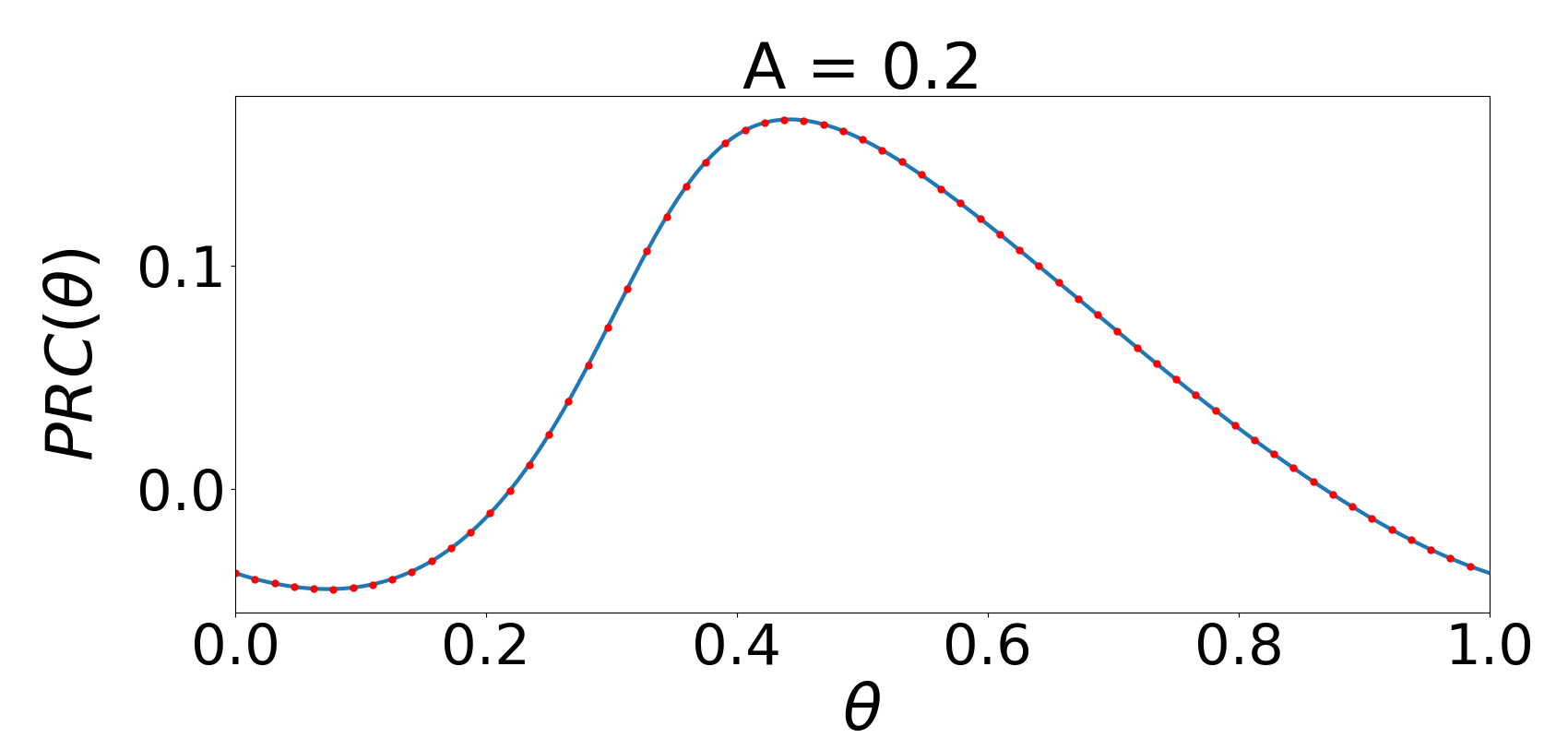}}
{\includegraphics[width=80mm]{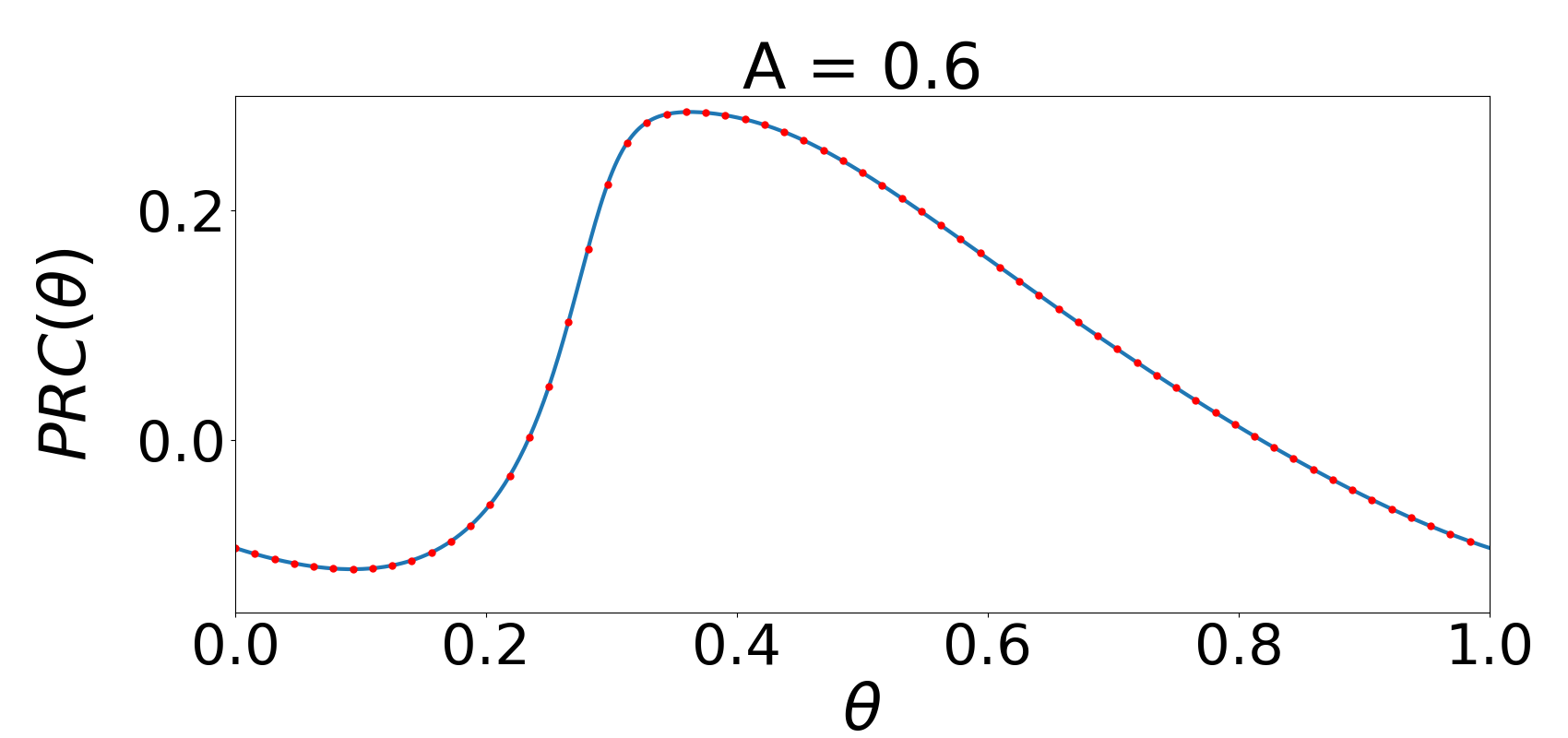}}
{\includegraphics[width=80mm]{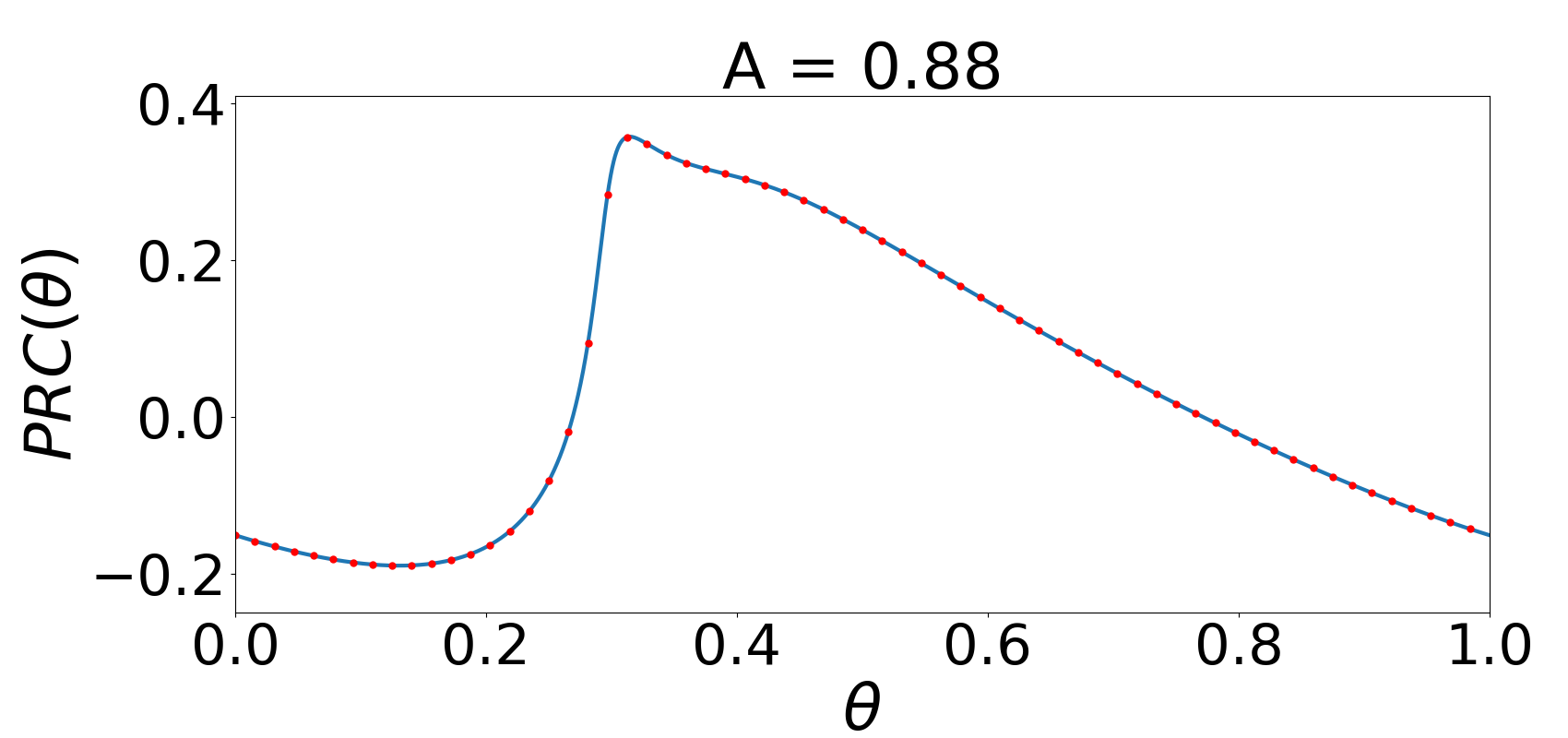}}
\caption{PRCs for the Wilson-Cowan model near a Hopf bifurcation (WC-Hopf) for different values of the amplitude
$A$ (as indicated in each panel) showing the comparison between the parameterization method (solid blue line) and the \textit{standard method}
(red dots).} \label{fig:prcsHopfCase0}
\end{figure}

\begin{figure}[H]
\centering
{\includegraphics[width=80mm]{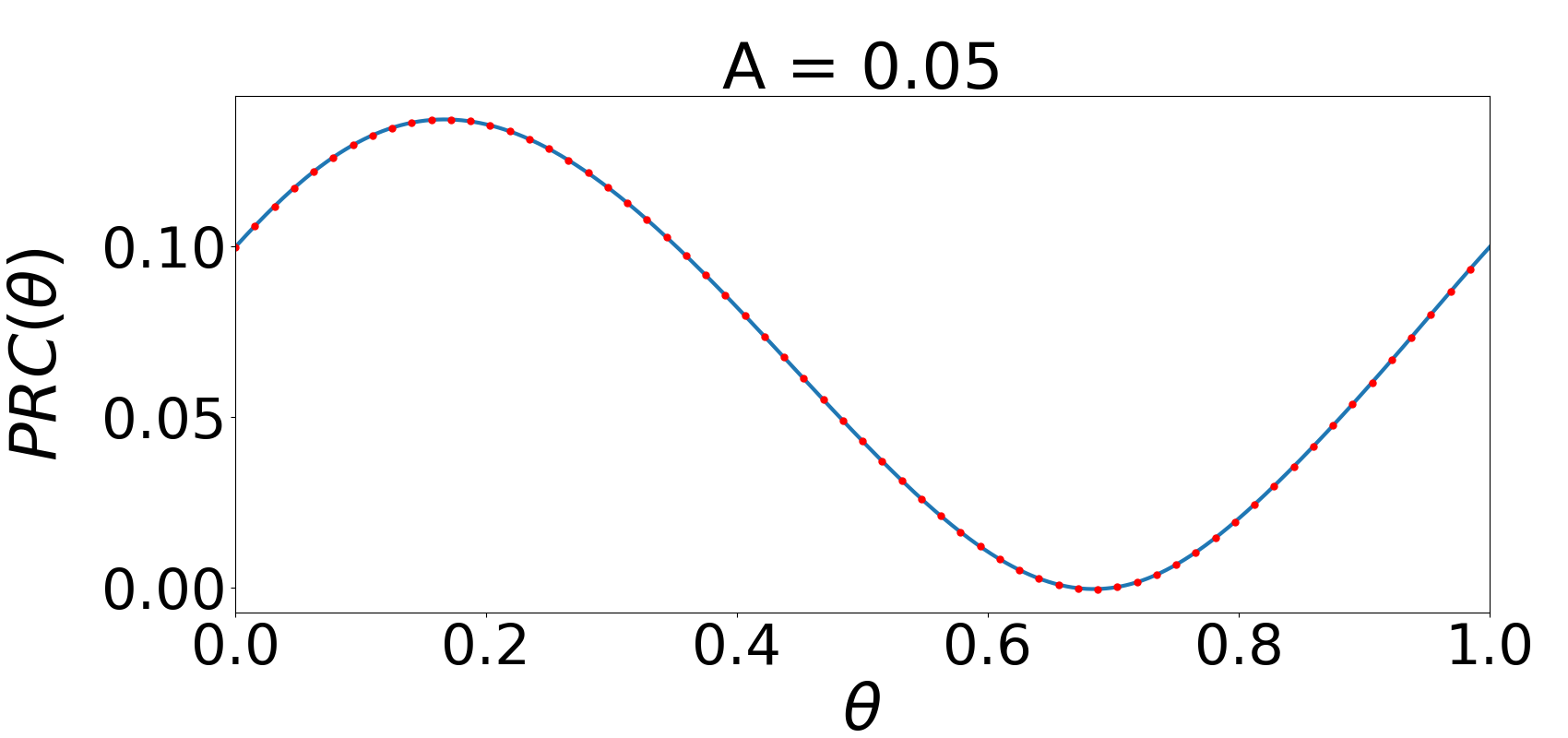}}
{\includegraphics[width=80mm]{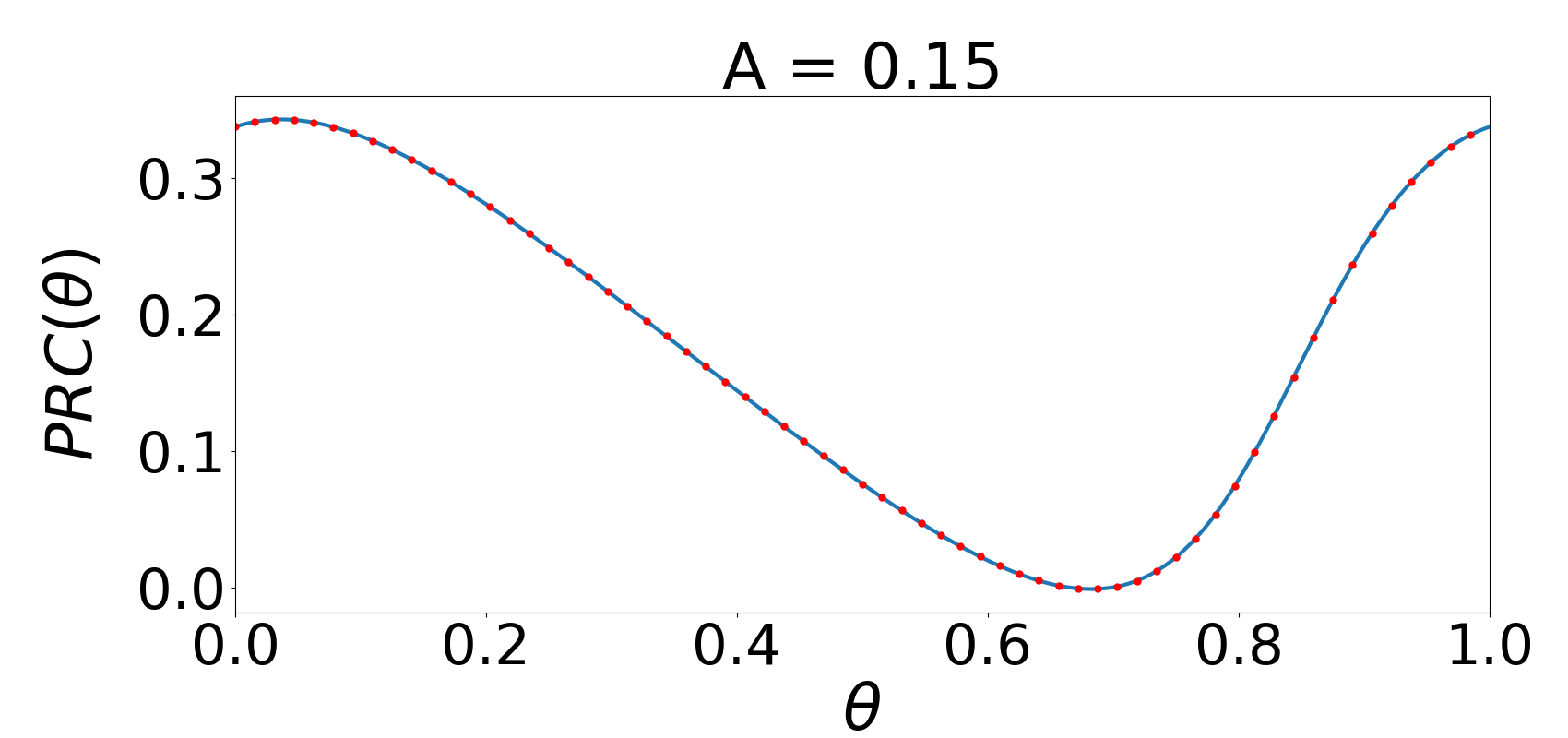}}
{\includegraphics[width=80mm]{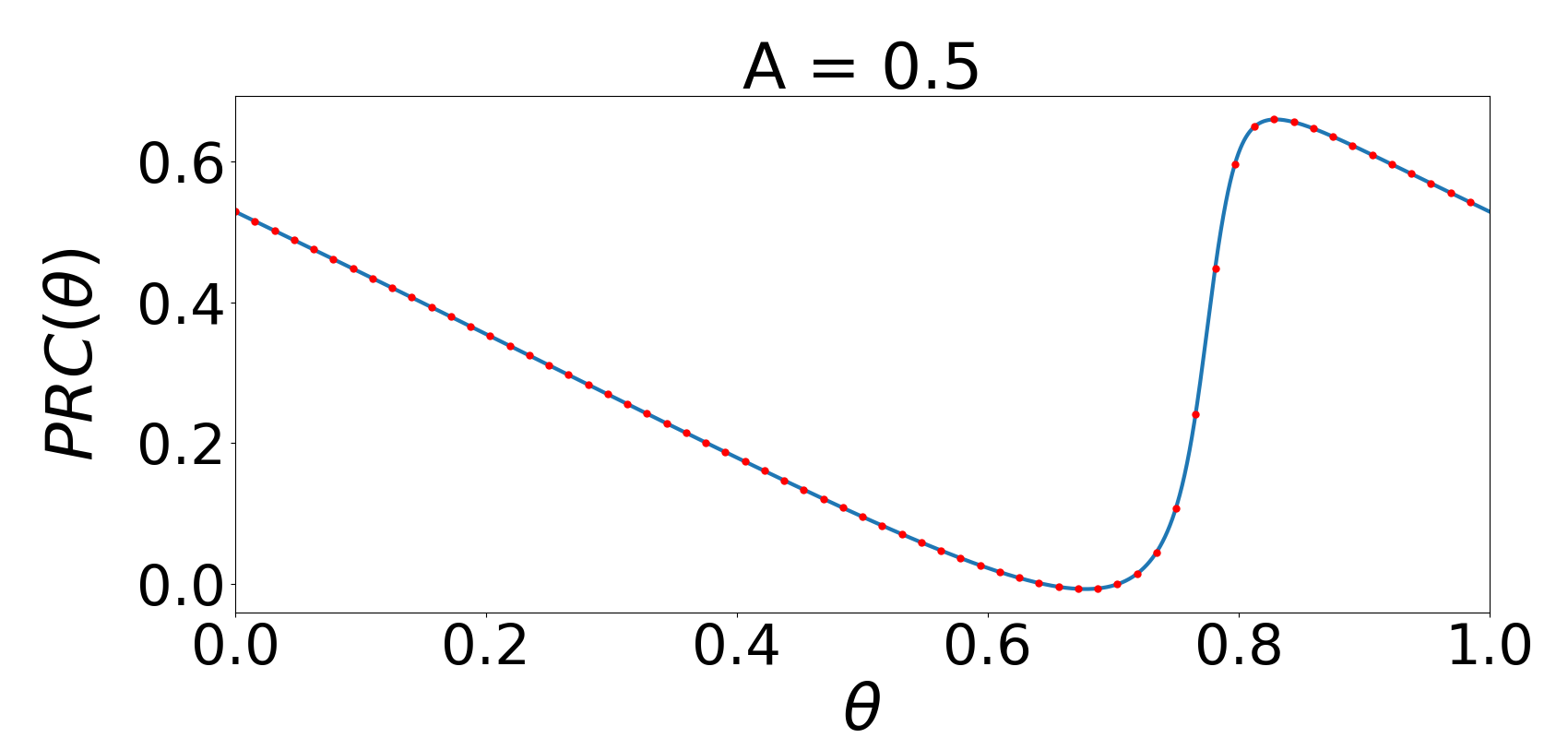}}
{\includegraphics[width=80mm]{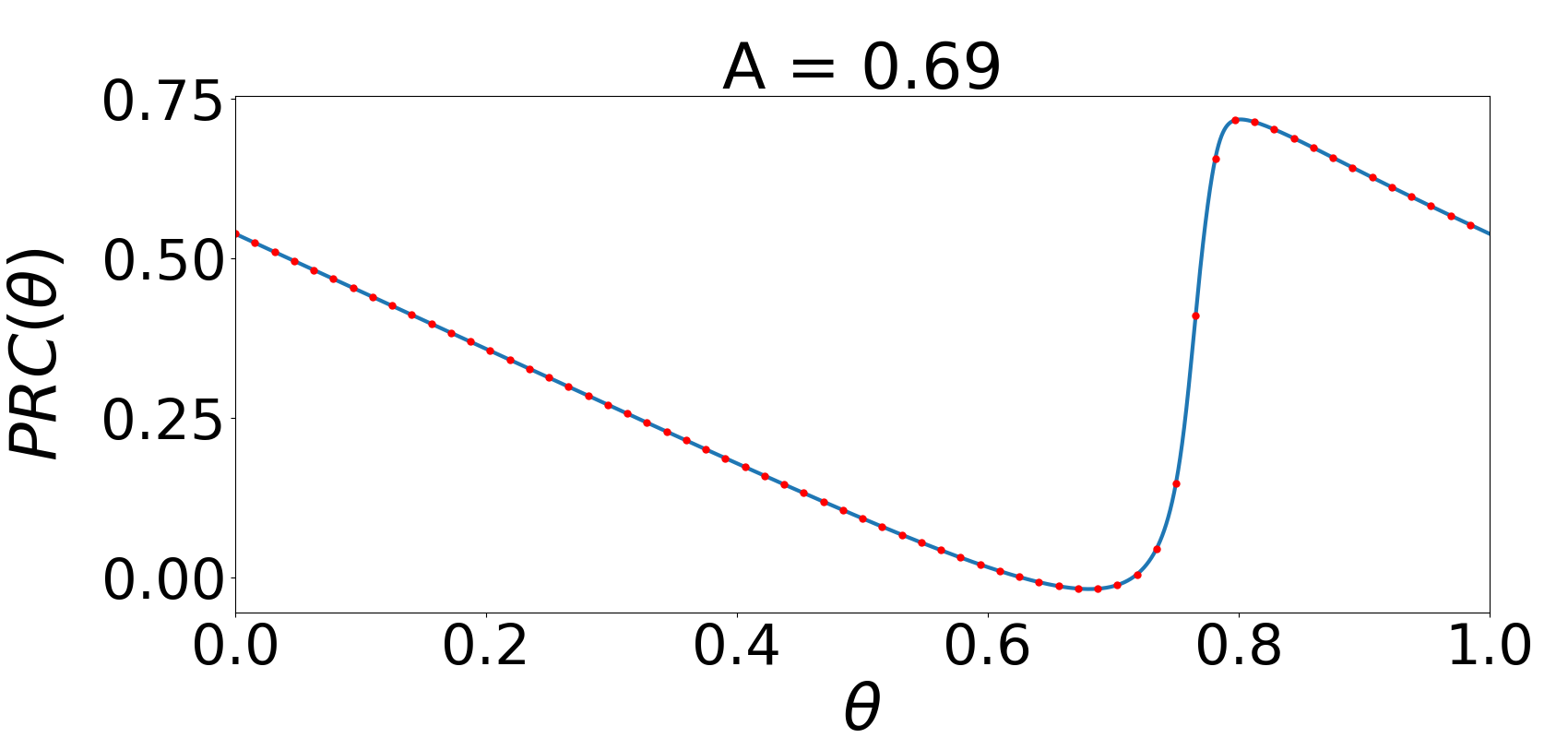}}
\caption{PRCs for the Wilson-Cowan near a SNIC bifurcation (WC-SNIC) for different values of the amplitude
$A$ (as indicated in each panel) showing the comparison between the parameterization method (solid blue line) and the \textit{standard method}
(red dots).} \label{fig:prcsSnicCase}
\end{figure}

\begin{figure}[H]
\centering
{\includegraphics[width=153mm]{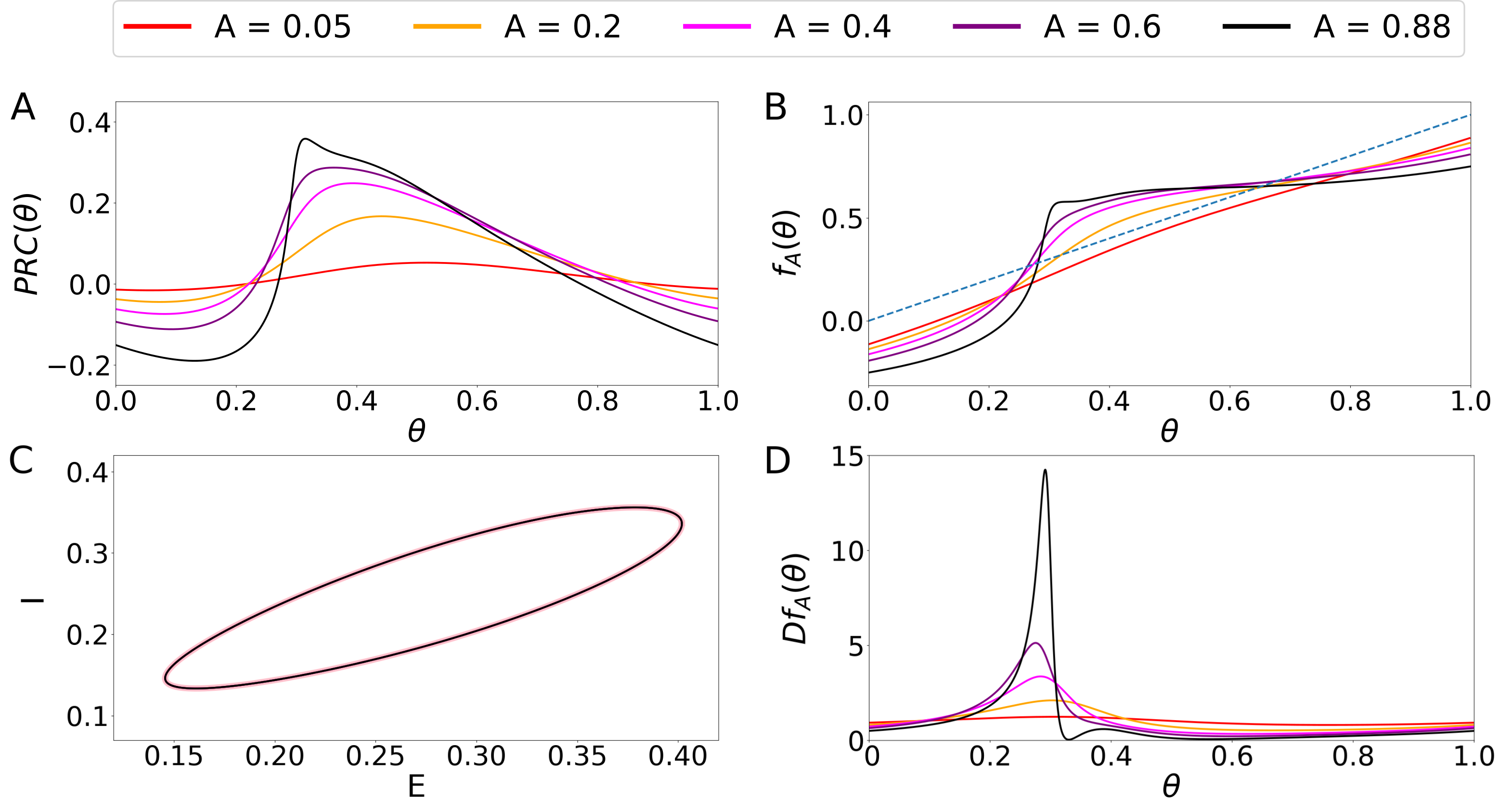}}
\caption{For the Wilson-Cowan model near a Hopf bifurcation (WC-Hopf)
and different amplitude values $A$  we show:
(A) the PRCs, (B) the dynamics $f_A(\theta)$ on the invariant curve $\Gamma_A$,
(C) the invariant curve $\Gamma_A$, (D) the derivative of $f_A(\theta)$.  
The dashed blue line in panel B corresponds to the identity function.}\label{fig:morePrcsHopfCase}
\end{figure}

\begin{figure}[H]
\centering
{\includegraphics[width=153mm]{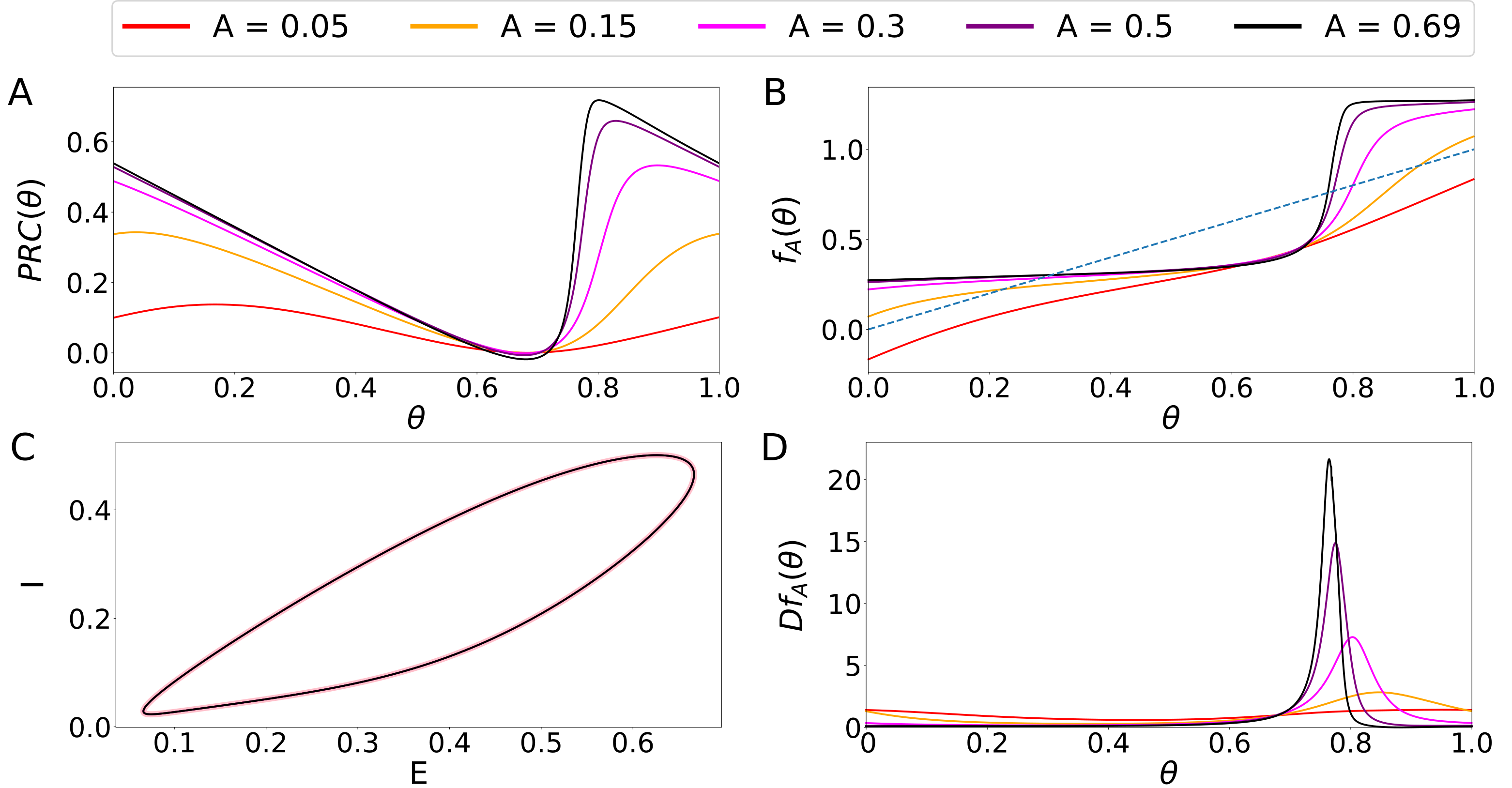}}
\caption{For the Wilson-Cowan model near a SNIC bifurcation (WC-SNIC) and different amplitude values we show:
(A) the PRCs, (B) the dynamics $f_A(\theta)$ on the invariant curve $\Gamma_A$,
(C) the invariant curve $\Gamma_A$,  (D) the derivative of $f_A(\theta)$. 
The dashed blue line in panel B corresponds to the identity function.} \label{fig:morePrcsSnicCase}
\end{figure}

\textbf{The Morris-Lecar model}.
It was originally developed to study the excitability properties for the
muscle fiber of the giant barnacle, and it has been established as a paradigm for
the study of different neuronal excitability types \cite{rinzel1989analysis, RinzelH13}.
The perturbed model has the form (see \cite{morris1981}):
\begin{equation}\label{eq:MLsys}
\begin{split}
C \dot{V} &= I_{app} - g_L(V - V_L) - g_Kw(V - V_K) - g_{Ca}m_{\infty}(V)(V - V_{Ca}) + Ap(t),\\
\dot{w} &= \phi \frac{w_{\infty}(V) - w}{\tau_w(V)},
\end{split}
\end{equation}
where
\begin{equation}
\begin{split}
m_{\infty}(V) = \frac{1}{2}(1 + \tanh((V-V_1)/V_2)),\\
w_{\infty}(V) = \frac{1}{2}(1 + \tanh((V-V_3)/V_4)),\\
\tau_{w}(V) = (\cosh((V-V_3)/(2V_4)))^{-1}.
\end{split}
\end{equation}

As in the previous example, we consider two sets of parameters for which the system
displays a limit cycle across a Hopf and a SNIC bifurcation \cite{ErmentroutTerman2010,RinzelH13}.
We will refer to them as MC-Hopf and MC-SNIC, respectively.
Some parameters of the model will be common to both cases, namely
$C=20, V_L=-60, V_K=-84, V_{Ca}=120, V_1=-1.2, V_2=18, g_L=2, g_K=8$.
The other parameter values are listed in Table~\ref{tab:MLprms}.
\begin{table}
	\begin{adjustbox}{minipage=.35\textwidth,valign=t}
\begin{center}
	\begin{tabular}{ccc}
		\hline
		Parameter         & Hopf & SNIC \\
		\hline
		$\phi$    & 0.04 & 0.067 \\
		$g_{Ca}$  & 4.4 & 4 \\
		$V_3$     & 2 & 12 \\
		$V_4$     & 30 & 17.4 \\
		$I_{app}$ & 91 & 45 \\
		\hline
		$T$       & 99.27 & 99.192 \\
		$T_{rel}$ & 6T & 5T \\
		$-\lambda$ & -0.0919 & -0.1198 \\
		\hline
	\end{tabular}
\end{center}
		\end{adjustbox}\hfill
		\begin{adjustbox}{minipage=[t]{.5\textwidth},valign=t}
			\setlength{\abovecaptionskip}{0pt}%
			{\caption{Parameter values for the Morris-Lecar model close to the
					corresponding type of bifurcation. For the indicated parameter values and $A=0$, 
					system~\eqref{eq:MLsys} has a stable periodic orbit of period $T$ and
					characteristic exponent $-\lambda$.}
				\label{tab:MLprms}}%
		\end{adjustbox}	
	\end{table}

We compute the PRC for different values of the amplitude $A$. In Figs. \ref{fig:prcsMLHopfCase} and \ref{fig:prcsMLSnicCase} we show
the comparison between the standard method and the parameterization method for ML-Hopf and ML-SNIC, respectively.
Again, we remark the good agreement between both methods. 
Other elements of the computation of the PRCs using the parameterization method are shown in Figures~\ref{fig:morePrcsMLHopfCase} and
\ref{fig:morePrcsMLSnicCase}. Again, both cases show a sharp rise in the PRC for certain phase values as the amplitude increases. 
We refer the reader to Section~\ref{sec:largeA} for a more detailed discussion.

\begin{figure}[H]
\centering
{\includegraphics[width=80mm]{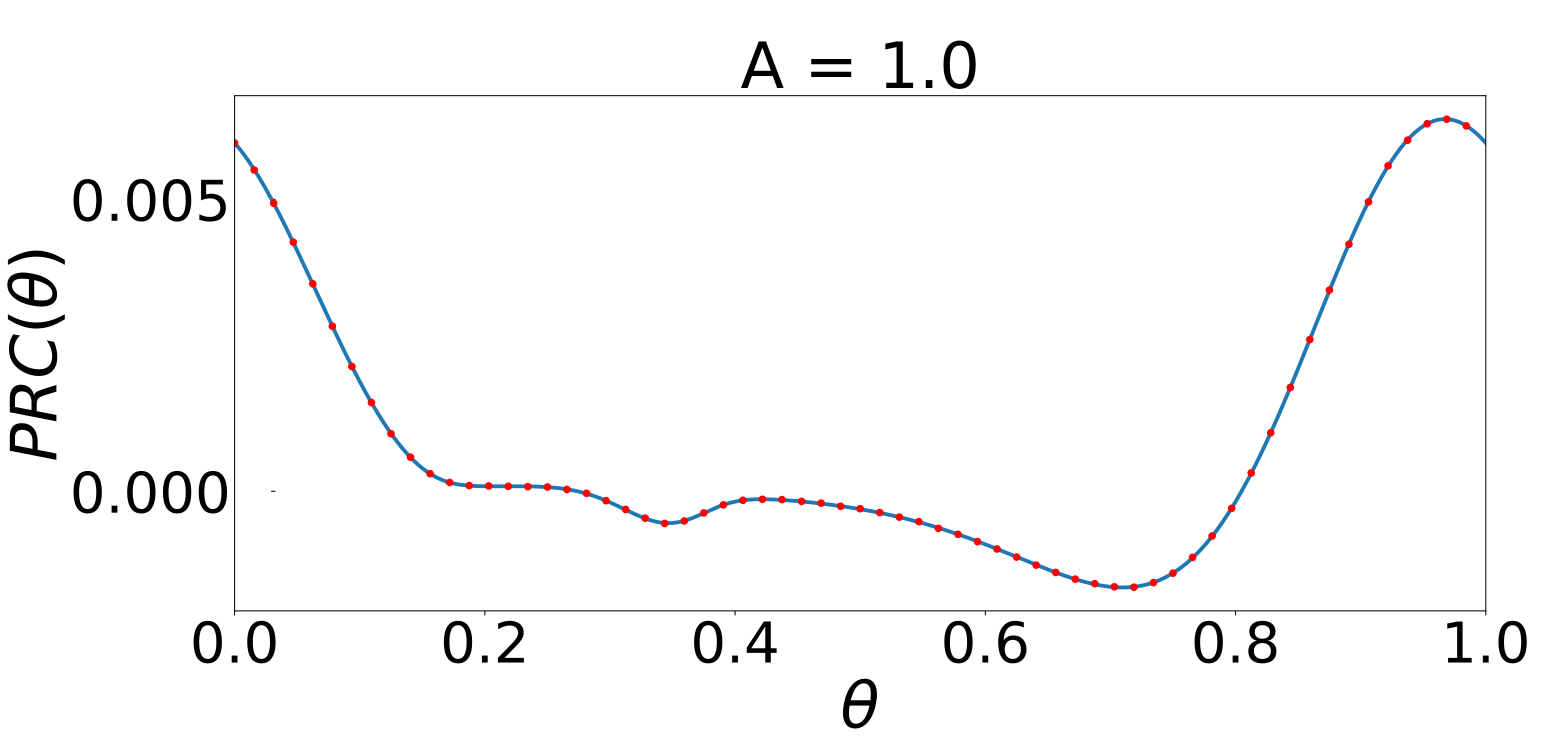}}
{\includegraphics[width=80mm]{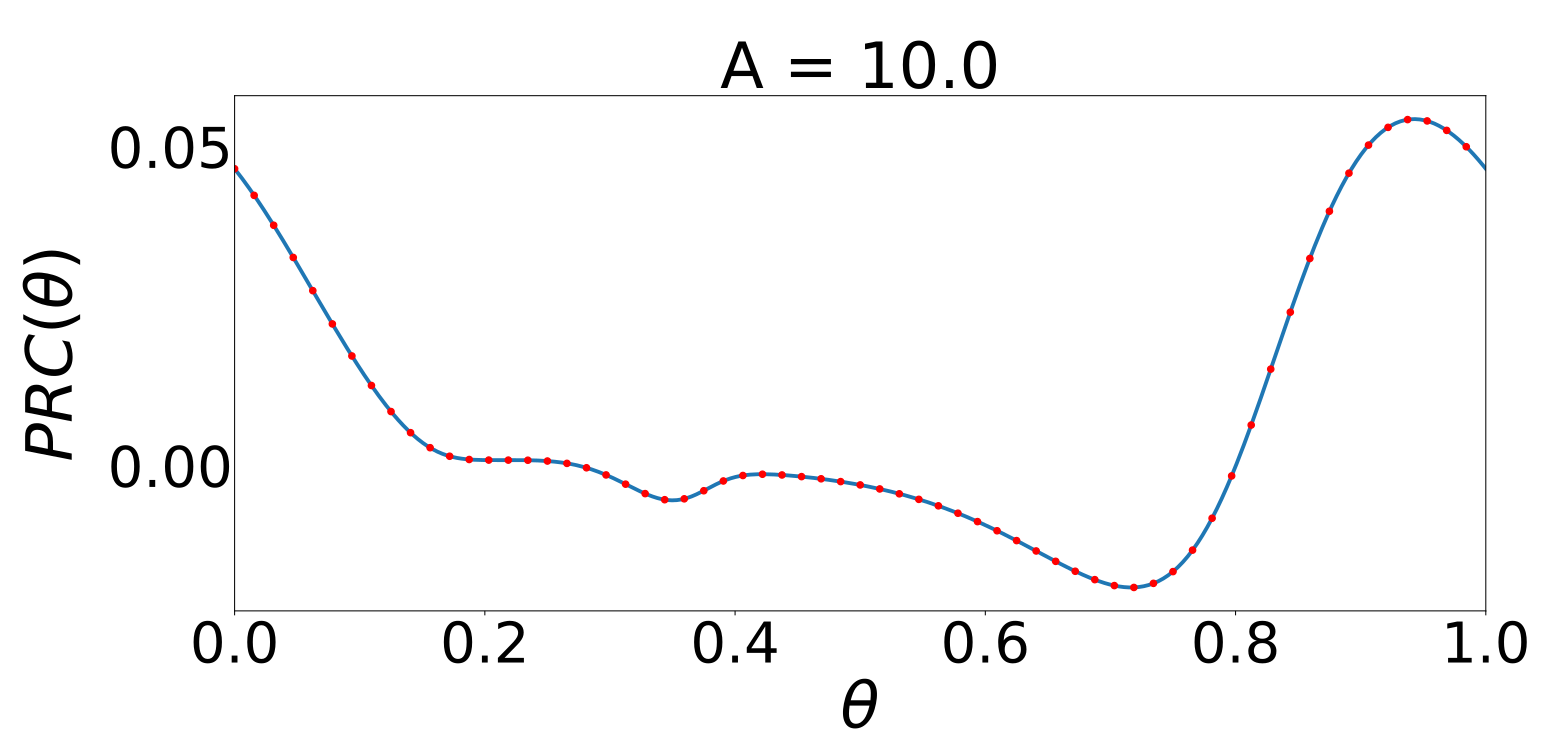}}
{\includegraphics[width=80mm]{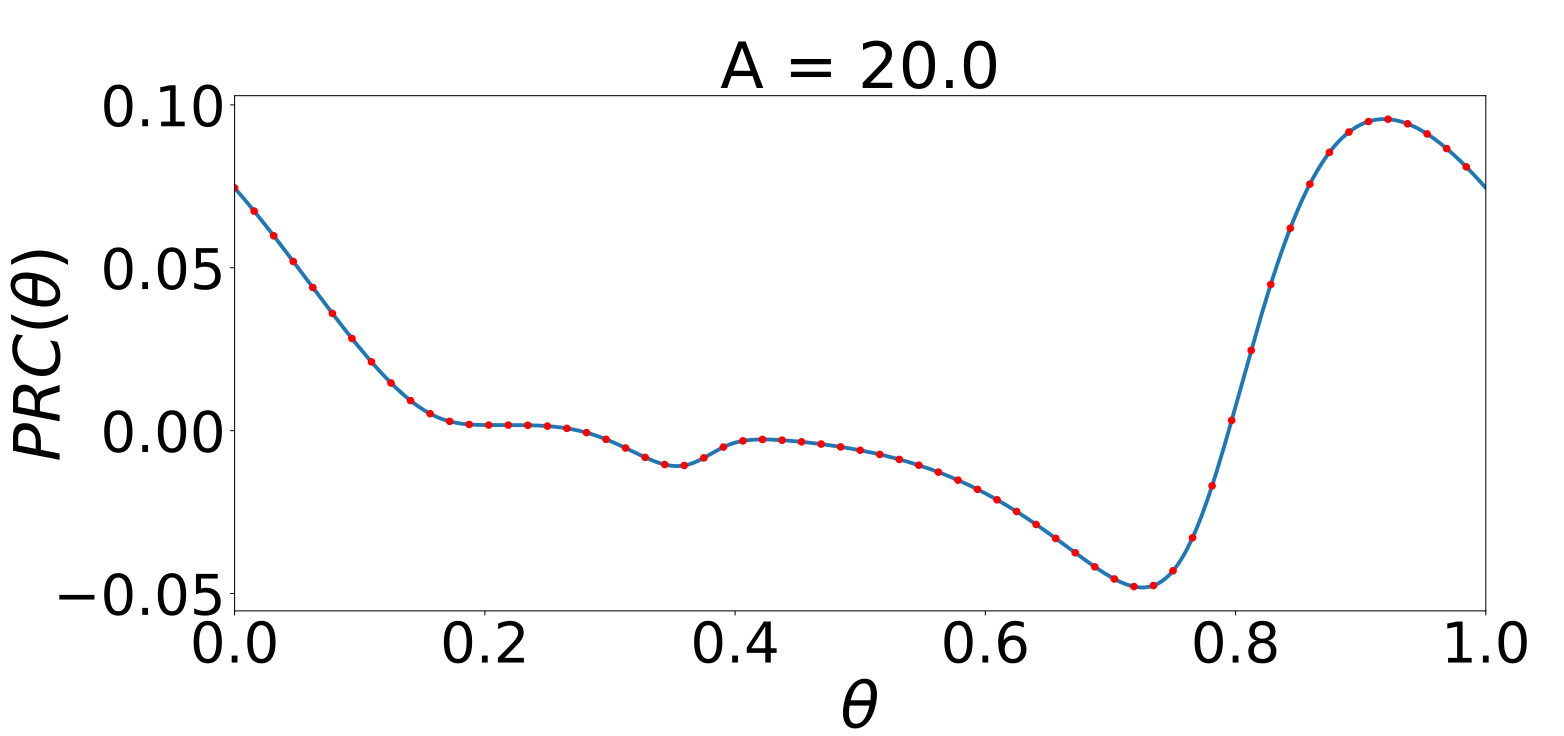}}
{\includegraphics[width=80mm]{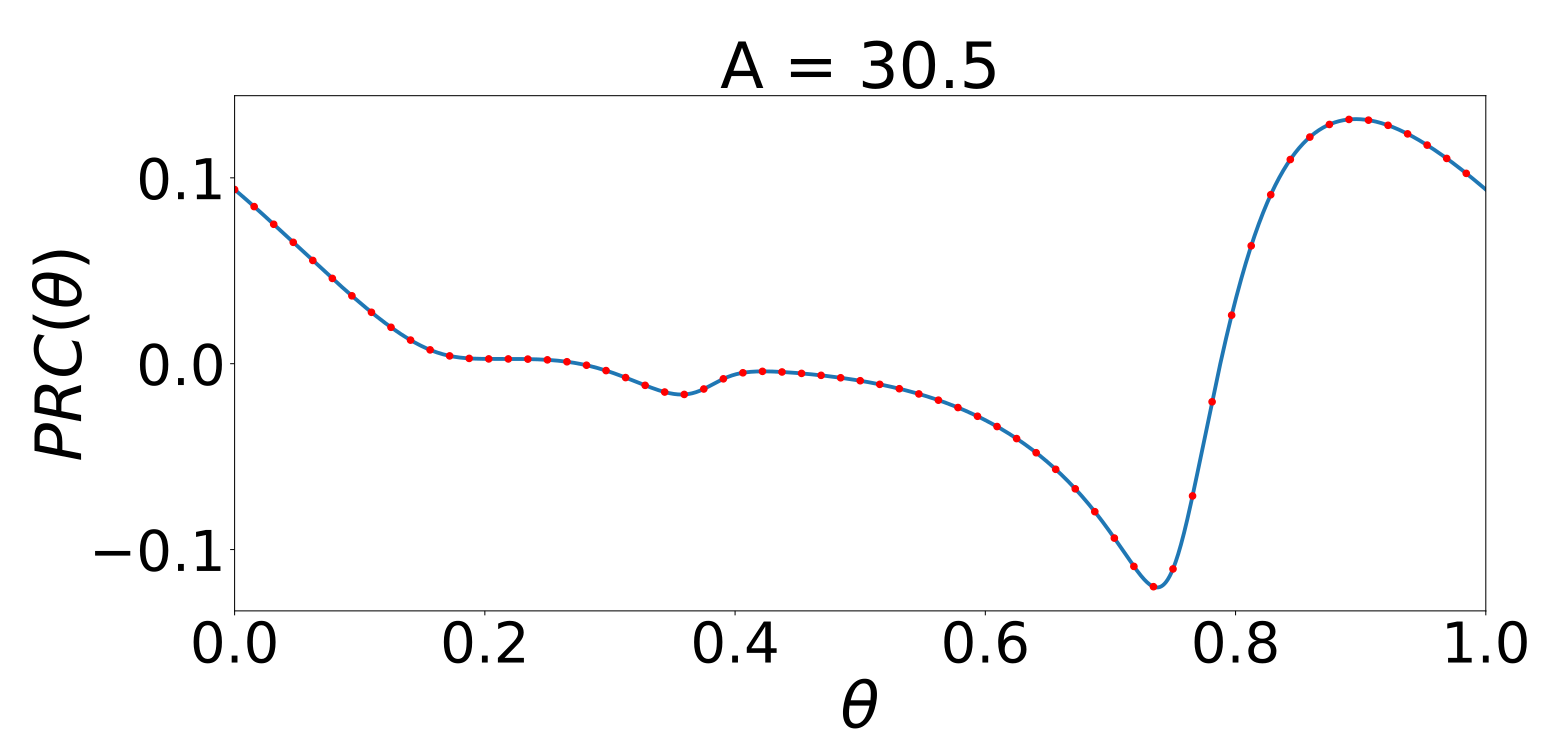}}
\caption{PRCs for the Morris-Lecar model near a Hopf bifurcation (ML-Hopf) for different values of the amplitude (indicated in each panel) 
showing the comparison between the parameterization method (solid blue line) and the \textit{standard method}
(red dots).} \label{fig:prcsMLHopfCase}
\end{figure}

\begin{figure}[H]
\centering
{\includegraphics[width=80mm]{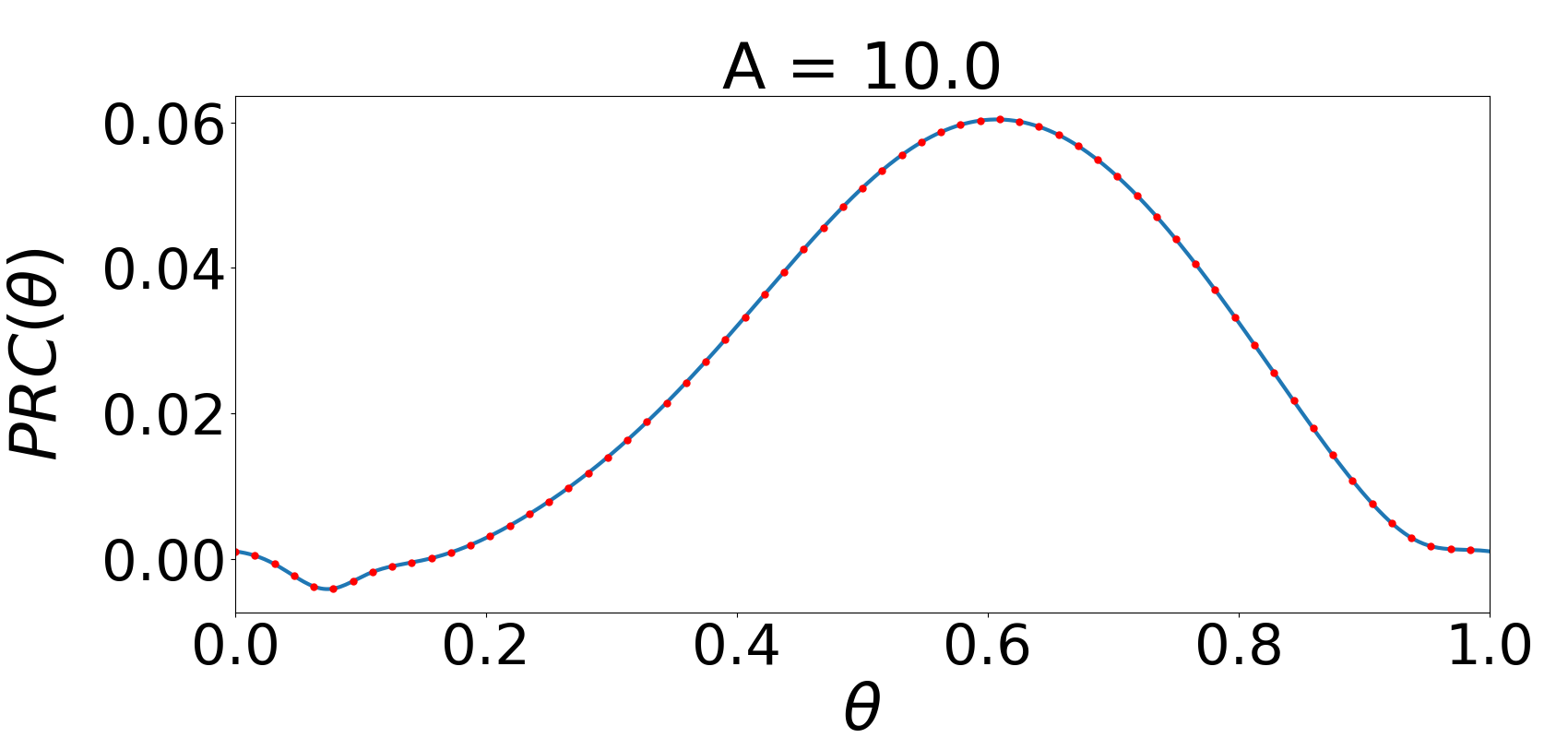}}
{\includegraphics[width=80mm]{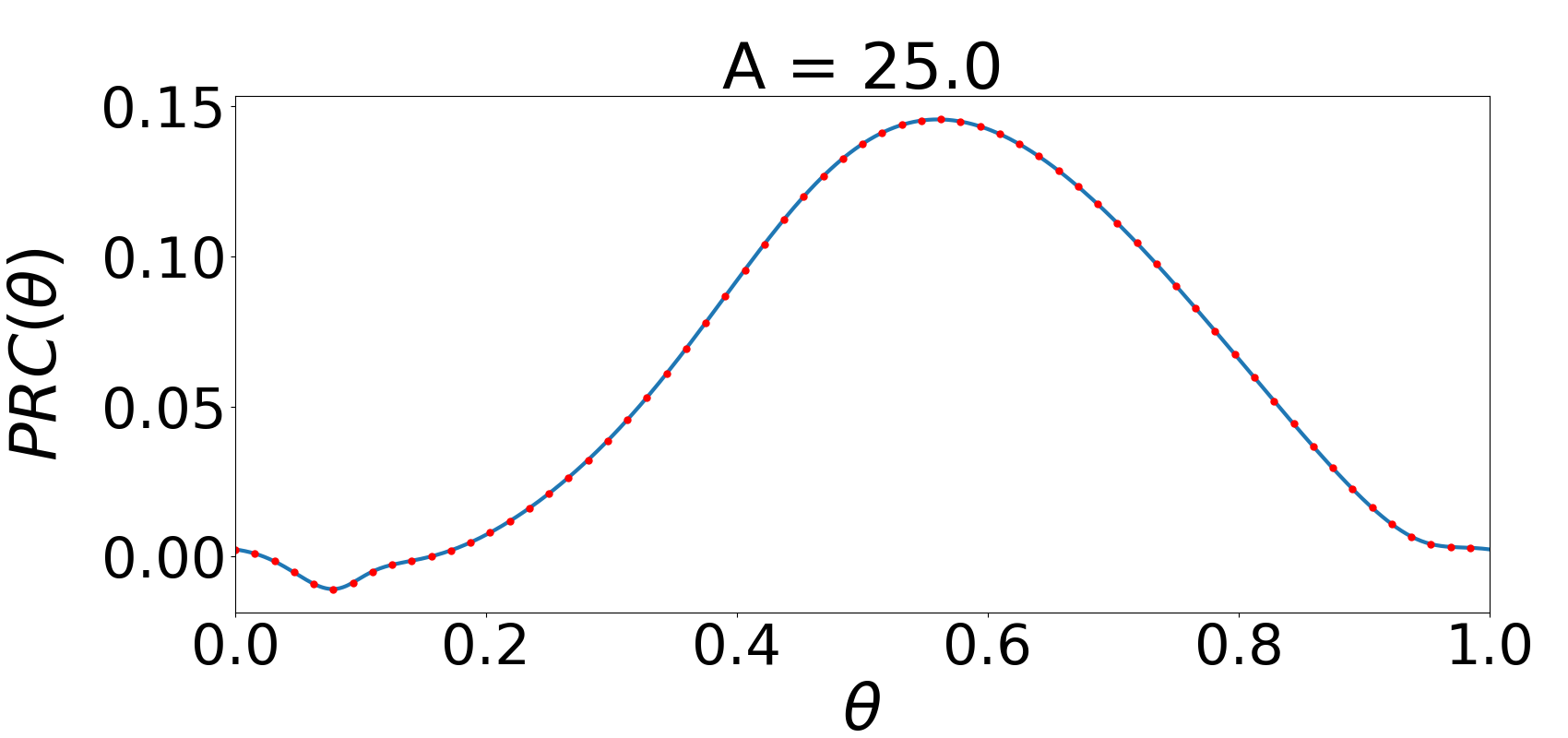}}
{\includegraphics[width=80mm]{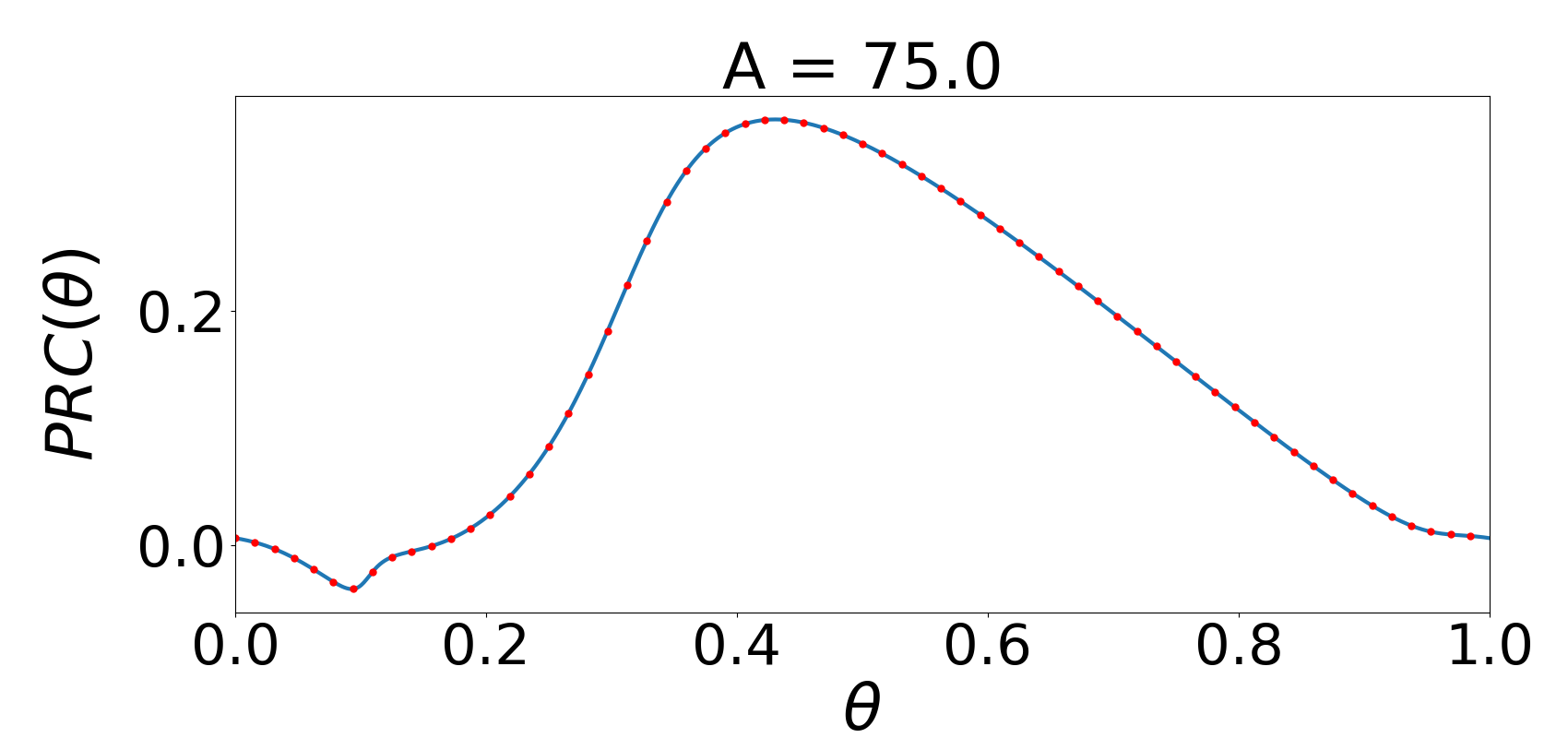}}
{\includegraphics[width=80mm]{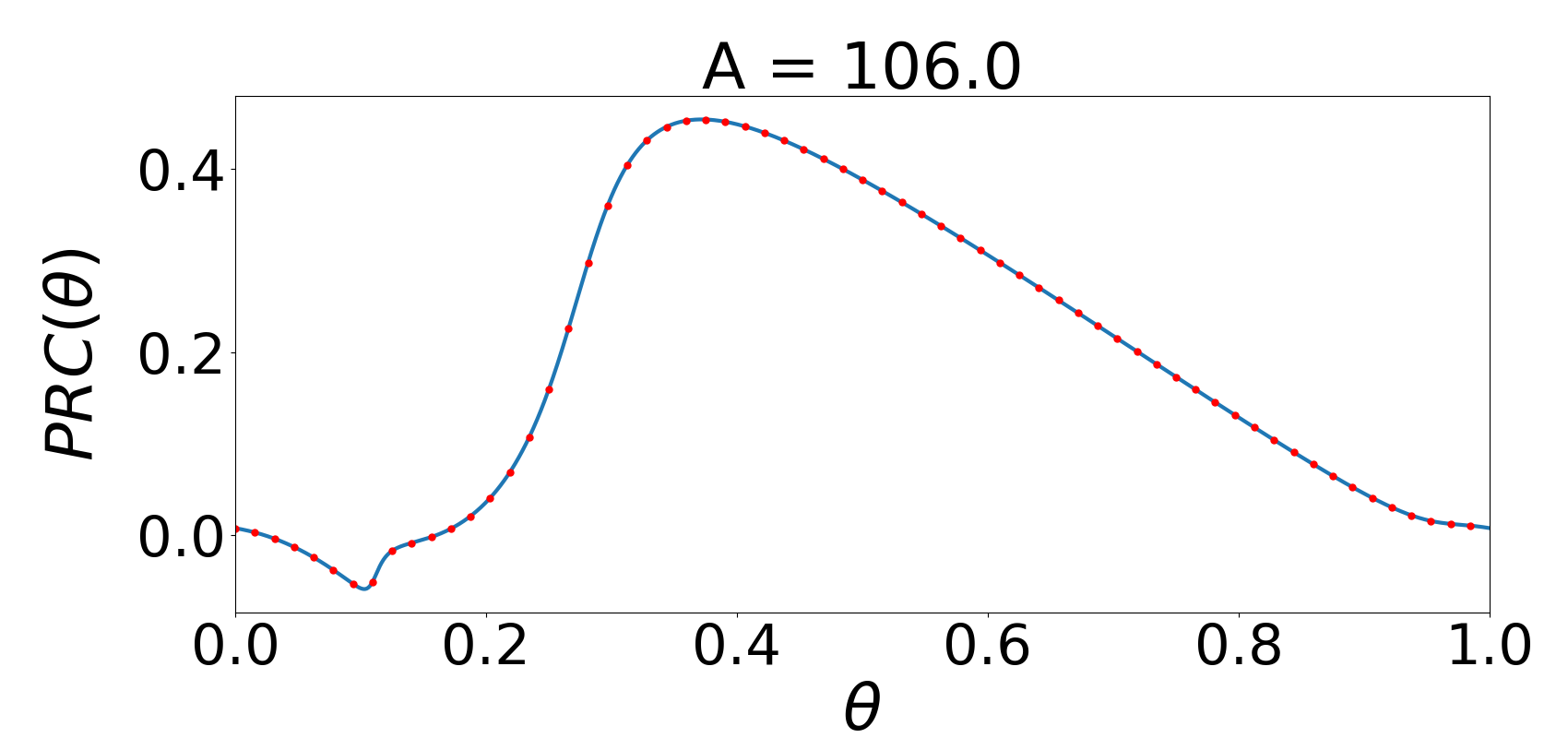}}
\caption{PRCs for the Morris-Lecar model near a SNIC bifurcation (ML-SNIC) for different values of the amplitude (indicated in each panel) 
showing the comparison between the parameterization method (solid blue line) and the \textit{standard method}
(red dots).} \label{fig:prcsMLSnicCase}
\end{figure}

\begin{figure}[H]
\centering
{\includegraphics[width=153mm]{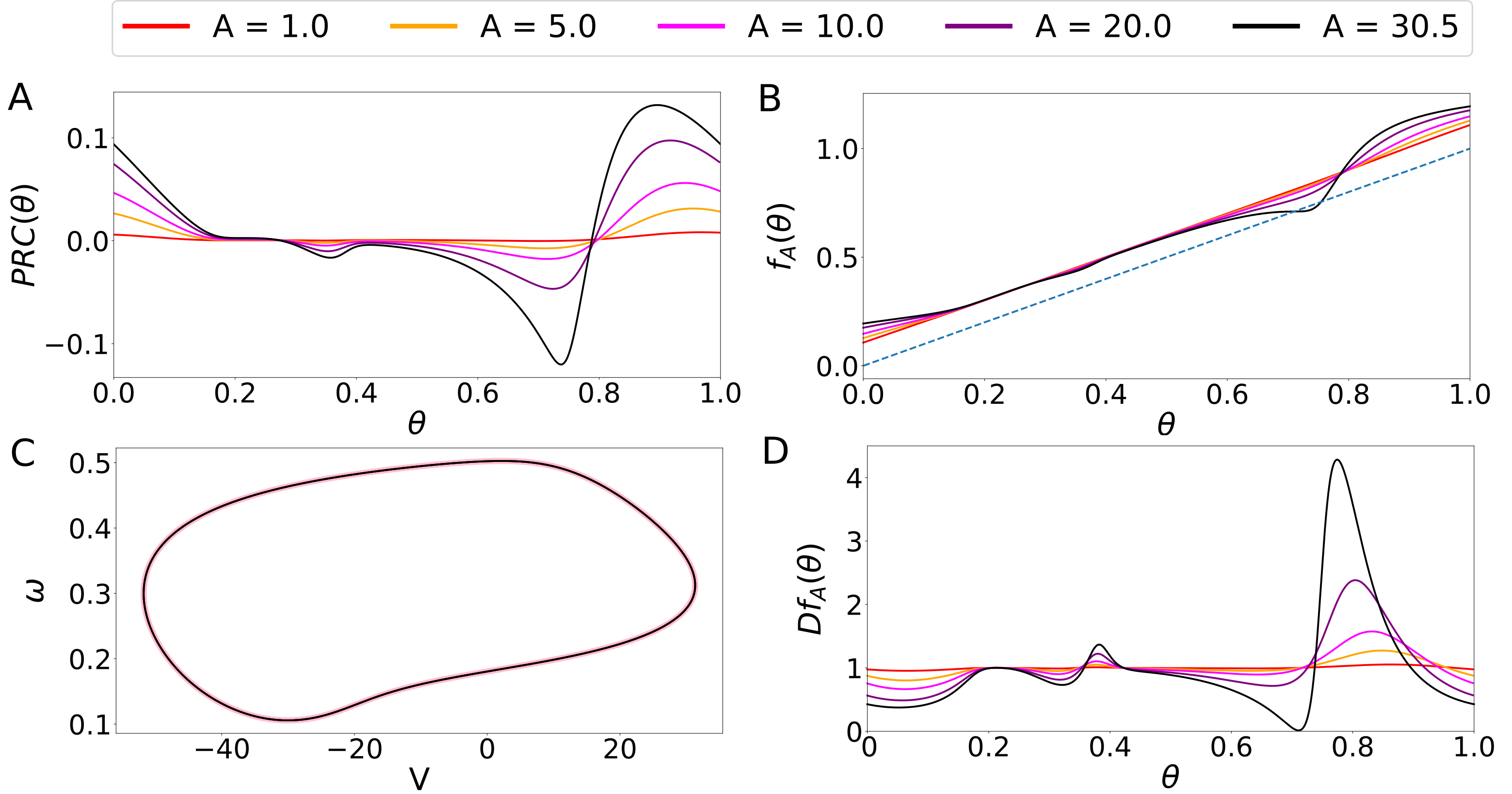}}
\caption{For the Morris-Lecar near a Hopf bifurcation (ML-Hopf) and different amplitude values we show: 
(A) the PRCs, (B) the dynamics $f_A(\theta)$ on the invariant curve $\Gamma_A$,
(C) the invariant curve $\Gamma_A$, (D) the derivative of $f_A(\theta)$. 
The dashed blue line in panel B corresponds to the identity function.} \label{fig:morePrcsMLHopfCase}
\end{figure}

\begin{figure}[H]
\centering
{\includegraphics[width=153mm]{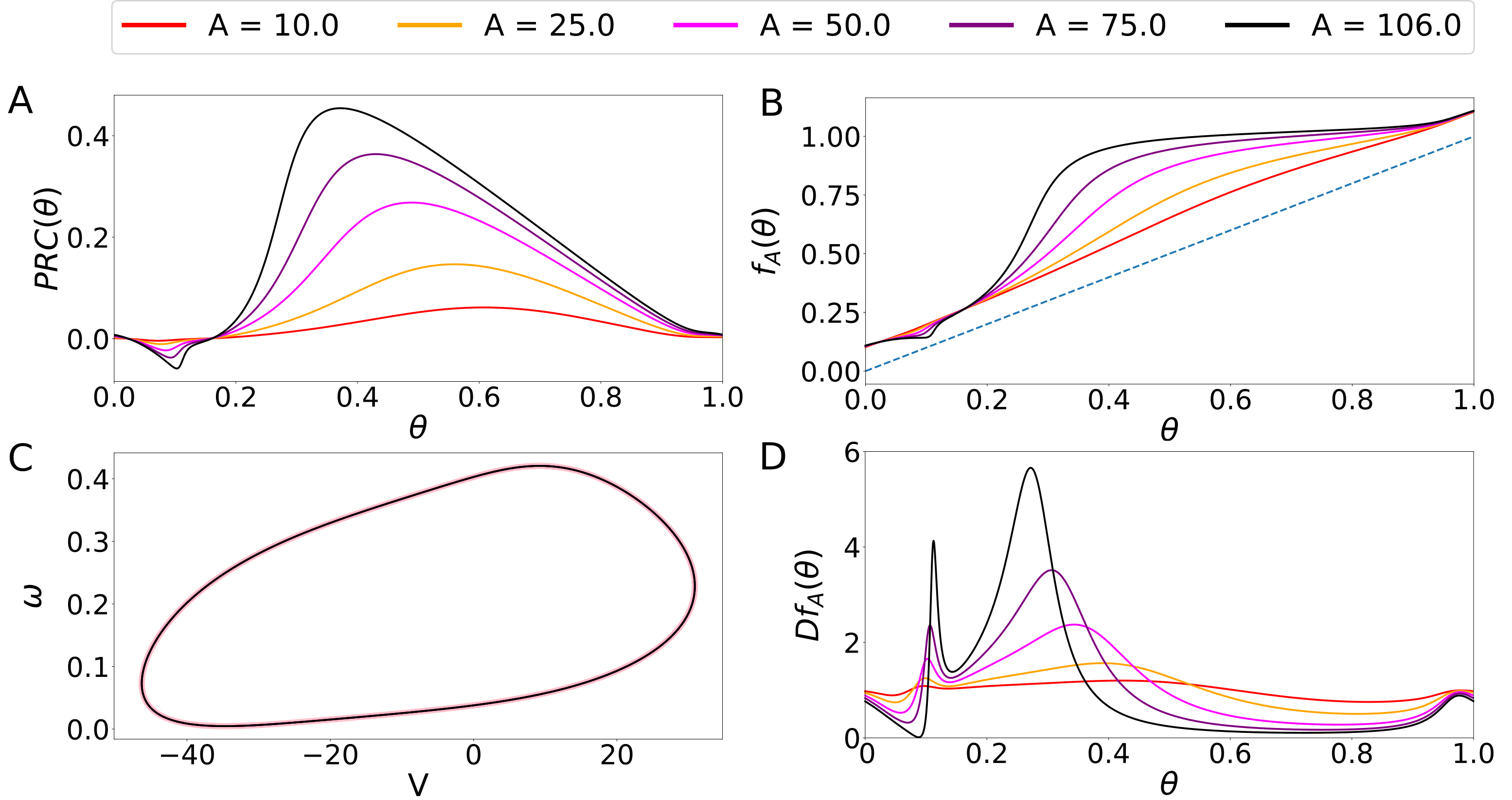}}
\caption{For the Morris-Lecar near a SNIC bifurcation (ML-SNIC) and different amplitude values we show: 
(A) the PRCs, (B) the dynamics $f_A(\theta)$ on the invariant curve $\Gamma_A$,
(C) the invariant curve $\Gamma_A$, (D) the derivative of $f_A(\theta)$. 
The dashed blue line in panel B corresponds to the identity function.} \label{fig:morePrcsMLSnicCase}
\end{figure}

\subsection{Large Amplitude Perturbations}\label{sec:largeA}
The application of the parameterization method using the algorithms in Appendix \ref{sec:num_alg} (see Section~\ref{sec:num_method}) strongly relies on the
existence of an invariant curve $\Gamma_A$ for the stroboscopic map of a system with an `artificially'
constructed periodic perturbation (see Theorem~\ref{thm:theorem1}). In the numerical examples shown in Figs.~\ref{fig:morePrcsHopfCase},  
\ref{fig:morePrcsSnicCase}, \ref{fig:morePrcsMLHopfCase} and \ref{fig:morePrcsMLSnicCase},
the computation of the invariant curve fails when the amplitude becomes large and approaches a certain value $A^*$ (which is different for each example), 
and so does the computation of the PRC using this method. In this Section we will discuss how changes in the waveform of the PRC might be related to normally
hyperbolic properties of the invariant curve. We will first focus our discussion on the WC-Hopf model.

First notice that, as the amplitude $A$ increases, the PRC becomes steeper (see Fig~\ref{fig:morePrcsHopfCase}A). 
By looking at the internal dynamics $f_A$ on the invariant curve $\Gamma_A$
(see Fig.~\ref{fig:morePrcsHopfCase}B), we observe that as the amplitude increases, the curve $f_A$
shows a sharp rise followed by a flat region. Moreover, there appear stable and unstable fixed points on the invariant curve $\Gamma_A$ 
(intersection of $f_A$ with the identity line on Fig.~\ref{fig:morePrcsHopfCase}B). 
Thus, the curve $\Gamma_A$ preserves  its normal hyperbolicity as 
long as the contraction/expansion rates on the invariant curve are weaker than the contraction rates on the normal directions 
(see Eq. \eqref{eq:equation43}). Indeed,
since the contraction rate in the normal direction is $\ocal(e^{-\lambda T_{rel}})$ (see \eqref{eq:equation43}) this means that 
$Df_A(\theta)=d\theta_{pert}/d\theta$ must remain bounded away from 0 (see Eq. \eqref{eq:nhcond}). 
However, we observe that as $A$ increases the value of $Df_A$ approaches 0 for a certain $\theta$ (see Fig.~\ref{fig:morePrcsHopfCase}D), 
thus causing the loss of the normal hyperbolicity property and the breakdown of the curve. 
For values of $A$ slightly smaller than the one for which $Df_A(\theta)$ vanishes, the numerical method in Appendix~\ref{sec:num_alg} fails to converge. 
However, we can apply the modified parameterization method provided by the algorithms in Section~\ref{sec:num_method} and compute the function $g_A$ 
(see equation~\eqref{eq:invariance_mod0}) and the PRC beyond the existence 
of an invariant curve (see Figs.~\ref{fig:prcsHopfCase} B, C, D right). 

Notice that the method in algorithm \ref{alg:algorithm_new0} also works if the invariant curve exists (see Fig.~\ref{fig:prcsHopfCase}A right). 
In this case, $g_A$ is $\ocal(e^{-\lambda T_{rel}})$-close to $f_A$ (see equations~\eqref{eq:fbara}, \eqref{eq:ffbar}, \eqref{eq:expandingK},  \eqref{eq:invariance_mod}). 
Therefore, for practical purposes, the modified method is faster and accurate enough to compute the PRC.

It is possible to describe the phenomenon of the breakdown of the curve in a geometric way using the concept of isochrons (curves of constant phase) 
and phaseless sets of the original limit cycle. 
For the model considered, the isochrons for the unperturbed limit cycle are shown in Fig.~\ref{fig:prcsHopfCase} (left). 
Notice that the unperturbed system has an unstable focus $P$ for which the isochrons are not defined (the phaseless set). 
Now, we consider the image of the curve $\Gamma_0$ under the map $F_{pert}$ introduced in \eqref{eq:Fpert}, for different values of $A$. 
Of course, the intersection of the curve $\Gamma_{pert}=F_{pert}(\Gamma_0)$ with the isochrons provides the new phases. 
For small values of $A$, $\Gamma_{pert}$ will intersect all the isochrons transversally leading to every possible new 
phase in a one-to-one correspondence (see Fig. \ref{fig:prcsHopfCase}A). 
Accordingly, the functions $f_A(\theta)$ and $g_A(\theta)$ are diffeomorphisms.
\newpage
However, as $A$ increases, there exists a value $A^*$ for which the curve $\Gamma_{pert}$ becomes tangent
to some isochrons, and therefore the curve $\Gamma_{pert}$ intersects some isochrons more than once (see Fig.~\ref{fig:prcsHopfCase}B). 
Thus, the map $g_A$  is no longer one-to-one, which means that $Dg_A$ vanishes for certain phases, 
causing the loss of normal hyperbolicity and the breakdown of the curve.

Clearly, the function $g_A$ for $A=0.95$  shows a local maximum and minimum (see Fig.~\ref{fig:prcsHopfCase}B), corresponding 
to an isochron tangency. 
When $A$ is increased further, the function $g_A$ splits in two. 
Indeed, the curve $\Gamma_{pert}$ will first intersect the phaseless point $P$ for a certain value  
$A \approx 1.035$ (see Fig.~\ref{fig:prcsHopfCase} C) and after that it will no longer enclose the point $P$, 
so it will not cross all the isochrons of the limit cycle (see Figure~\ref{fig:prcsHopfCase}D). The map $g_A$ will then be discontinuous at the point where the curve $\Gamma_{pert}$ intersects the phaseless set. 
After that, the function  $g_A$ will be continuous again when we take modulus 1, but, of course, the images will not span the whole interval $[0,1)$. Regarding the winding number classification for PRCs, defined as the number of times the curve $\Gamma_{pert}$ traverses a complete cycle as defined by the isochrons of the unperturbed limit cycle \cite{ glass1988clocks, glass1984discontinuities}, for $A \approx 1.035$ there is a transition from a type $1$ PRC to a type $0$ PRC.

The modified parameterization method introduced in Section~\ref{sec:sec42} allows for the computation of the Amplitude Response Curve (ARC) 
(see Algorithm~\ref{alg:algorithm_new}).
The ARCs for the amplitude values considered in Fig.~\ref{fig:prcsHopfCase} are shown in Fig.~\ref{fig:arcsHopfCase}.
The ARC provides information about how ``far'' in time the perturbation displaces the trajectory away from the limit cycle. 
That is, the larger the value of the ARC, the longer it will take for the displaced trajectory to relax back to the limit cycle (and therefore one should consider a larger $T_{rel}$). 

Notice that when the curve $\Gamma_{pert}$ intersects the phaseless set (point $P$), which occurs for a critical amplitude 
$A_c \approx 1.035$, 
there exists a perturbed trajectory which never returns to the limit cycle, and the ARC would show an essential discontinuity. 
Thus, for smaller values of the amplitude $A<A_c$, the ARC shows a peak whose size increases as the amplitude 
$A$ is increased towards $A_c$ (see Figs~\ref{fig:arcsHopfCase}A,B,C). 
However, for amplitude values larger than $A_c$ (Fig~\ref{fig:arcsHopfCase}D), the ARC peak decreases again with $A$, because 
$\Gamma_{pert}$ moves away from the neighbourhood of the phaseless point $P$.

\begin{figure}[H]
\centering
{\includegraphics[width=160mm]{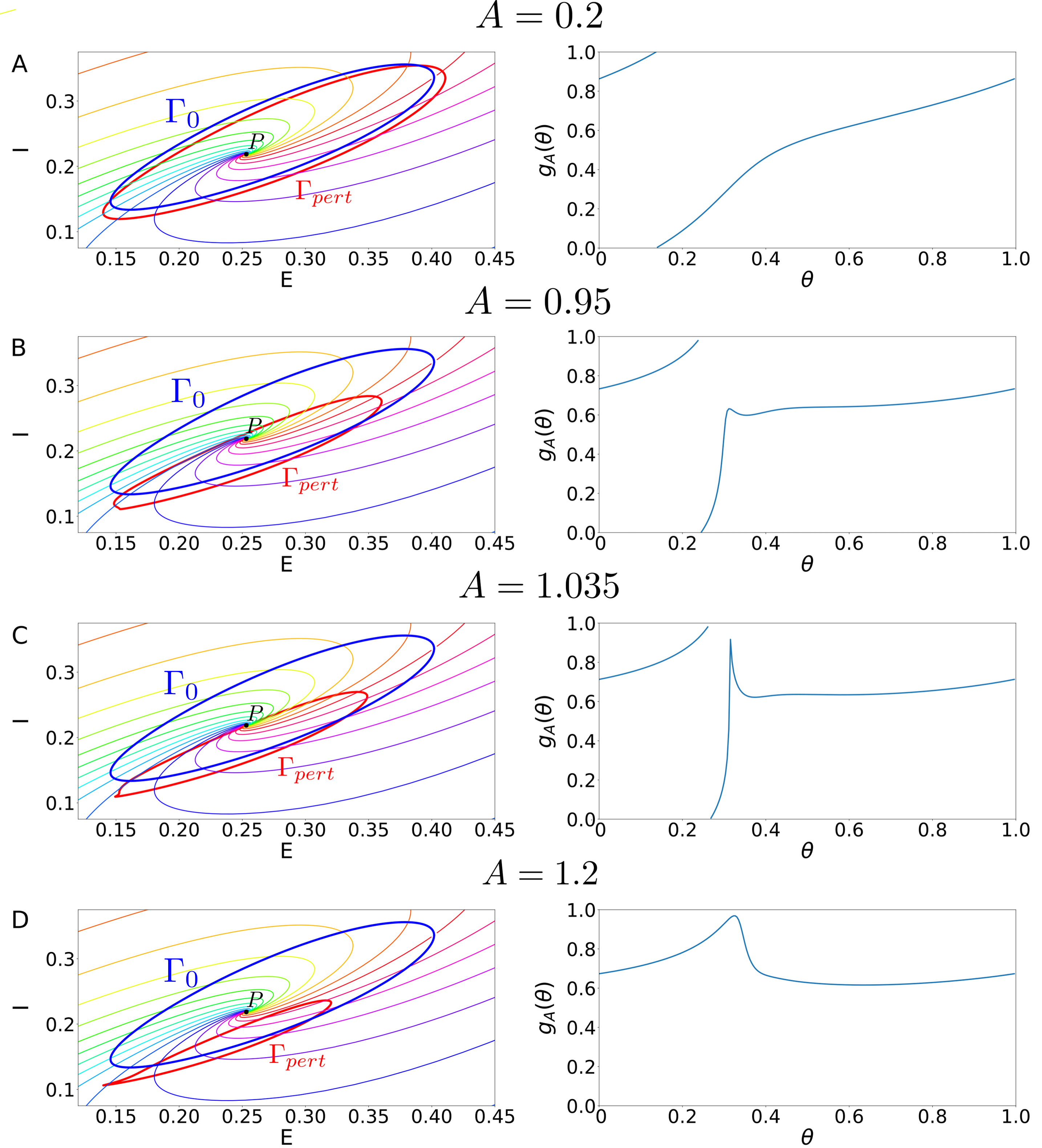}}
\caption{For different values of the amplitude (indicated in each panel) we show the 
isochrons of the unperturbed limit cycle $\Gamma_0$ for the Wilson-Cowan model near a Hopf bifurcation (WC-Hopf), together with the curves $\Gamma_{pert}$ (left) and 
the functions $g_A$ obtained with the modified parameterization method (right). 
The amplitudes selected cover the breakdown of the curve $\Gamma_A$ and a transition from type 1 to type 0 PRCs (see text).} \label{fig:prcsHopfCase}
\end{figure}

\begin{figure}[H]
\centering
{\includegraphics[width=80mm]{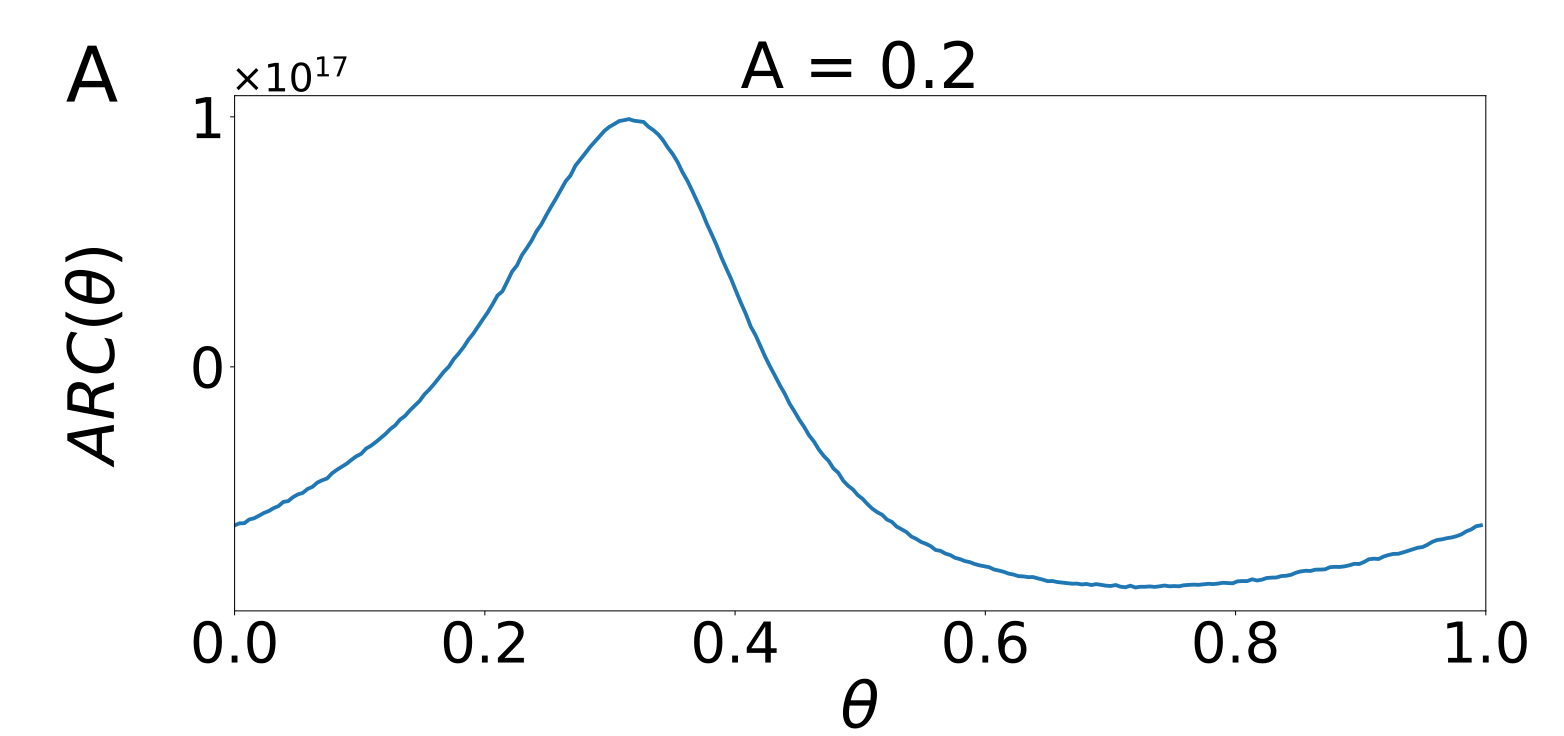}}
{\includegraphics[width=80mm]{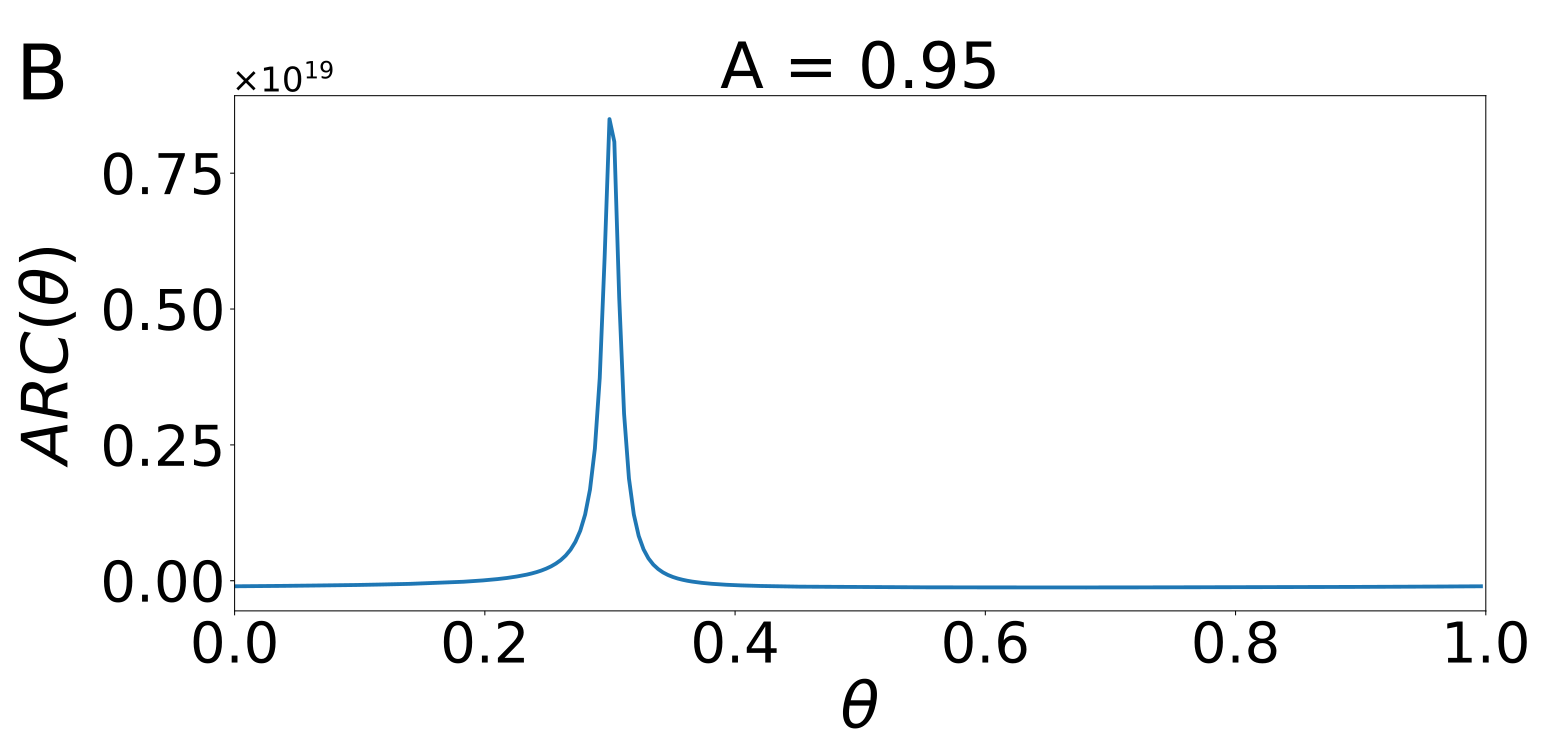}}
{\includegraphics[width=80mm]{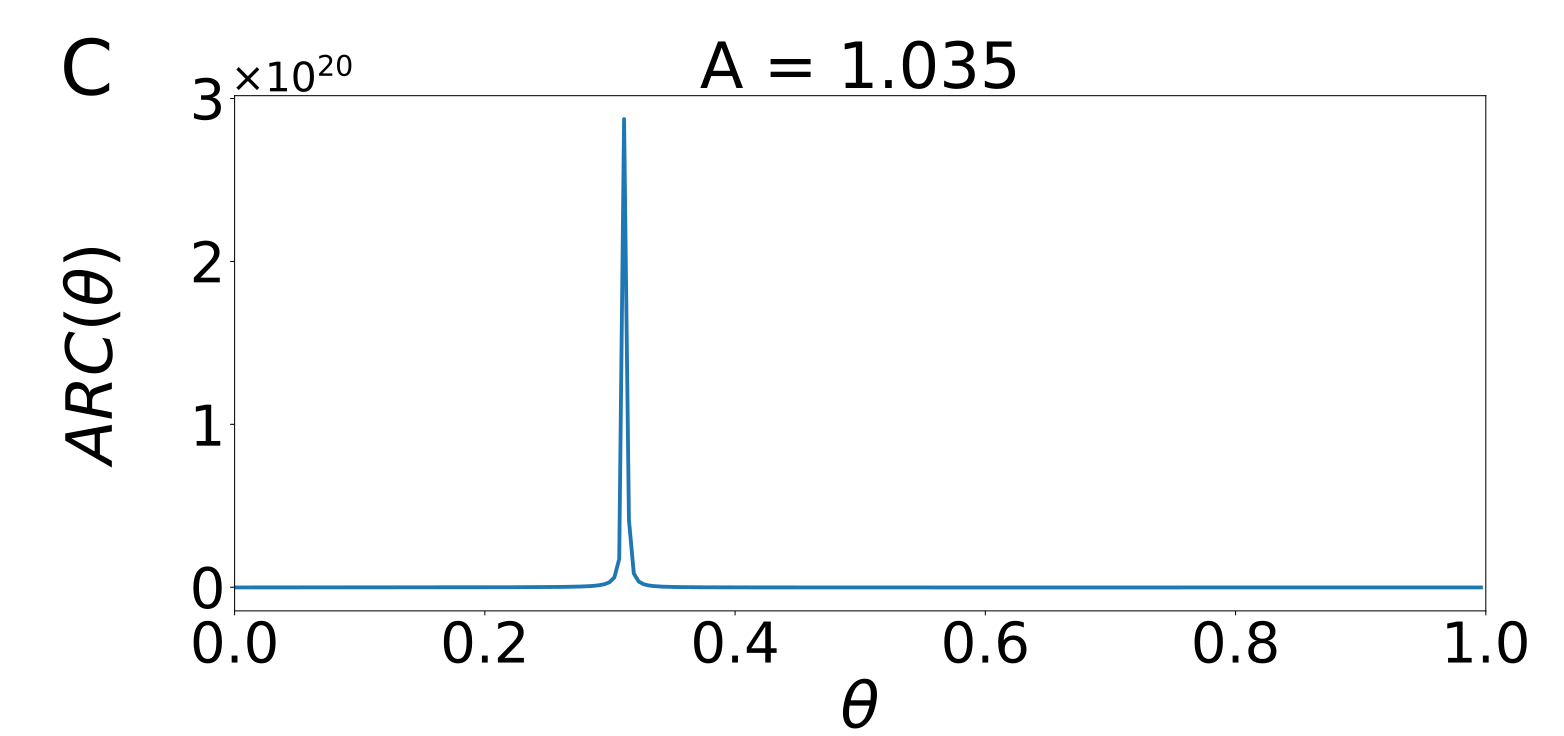}}
{\includegraphics[width=80mm]{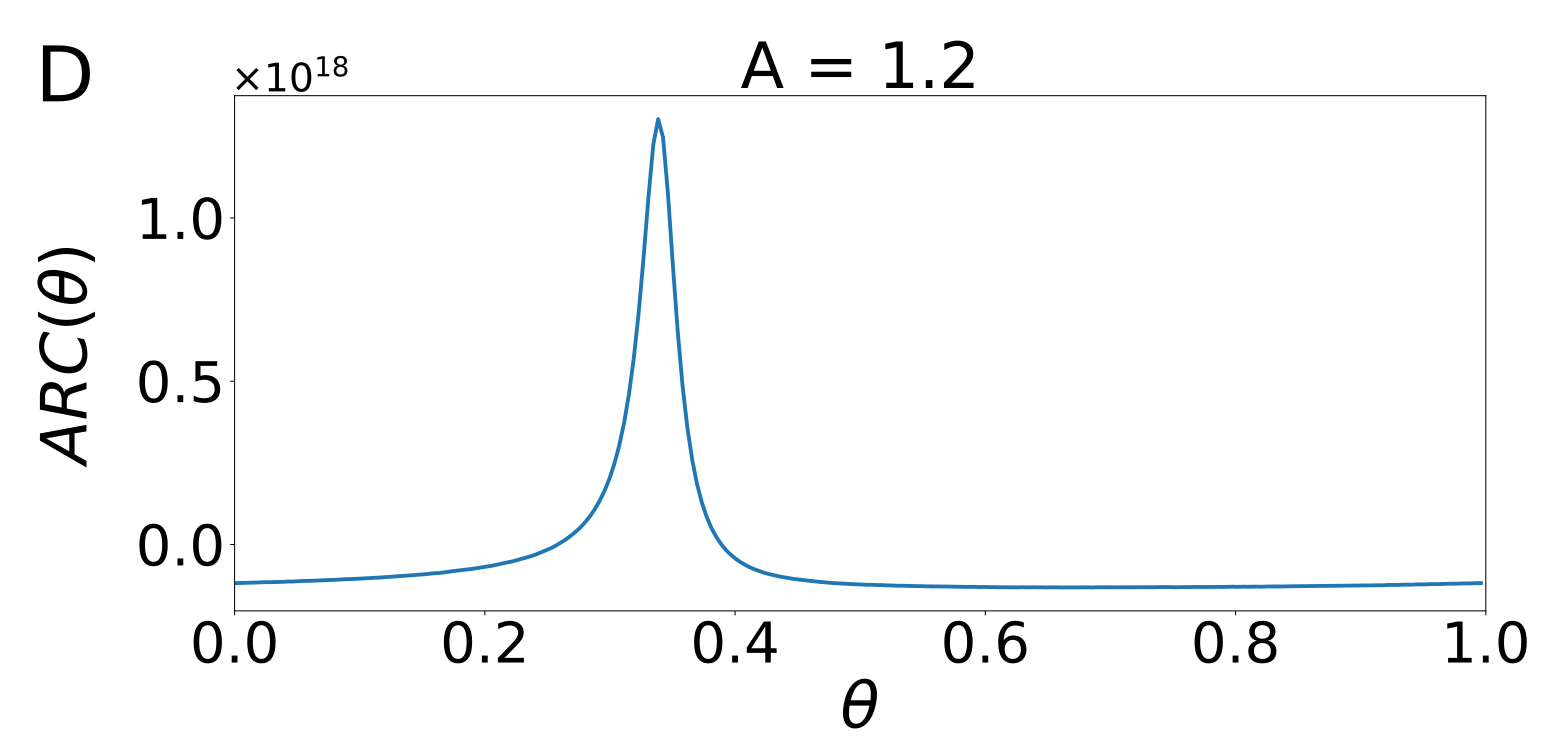}}
\caption{ARCs for the Wilson-Cowan model near a Hopf bifurcation (WC-Hopf) and different values of the
amplitude (same as in Fig.~\ref{fig:prcsHopfCase}).} 
\label{fig:arcsHopfCase}
\end{figure}

A similar phenomenon as discussed for the WC-Hopf occurs for the WC-SNIC and ML-SNIC examples. The Morris-Lecar model near a Hopf bifurcation (ML-Hopf) is slightly different, since 
in this case the phaseless set is larger compared to the 
WC-Hopf case: it is a positive measure set bounded by an unstable limit cycle $\Gamma_u$ which determines the basin of attraction of the 
equilibrium point lying in its interior (see Fig.\ref{fig:breakdownHopfML}A).
As in the WC-Hopf case, the invariant curve $\Gamma_A$ for the ML-Hopf disappears due to an isochron tangency of $\Gamma_{pert}$. Consistently with this tangency, $\Gamma_{pert}$ for $A=33$ crosses some isochrons more than once (see Fig.\ref{fig:breakdownHopfML}A and zoom in B) and as Figs.\ref{fig:breakdownHopfML} C, D show,
the function $g_A$ looses its monotonicity. 
Nevertheless, one can still compute the PRC by means of the Algorithm~\ref{alg:algorithm_new0}. 
By contrast, for larger amplitude values, several points of $\Gamma_{pert}$ leave the basin of attraction of the stable limit cycle 
(see Fig.\ref{fig:breakdownHopfML}A and zoom in B for $A=40$). 
Thus, the PRC can no longer be computed even with the modified method.

\begin{figure}[H]
\centering
{\includegraphics[width=160mm]{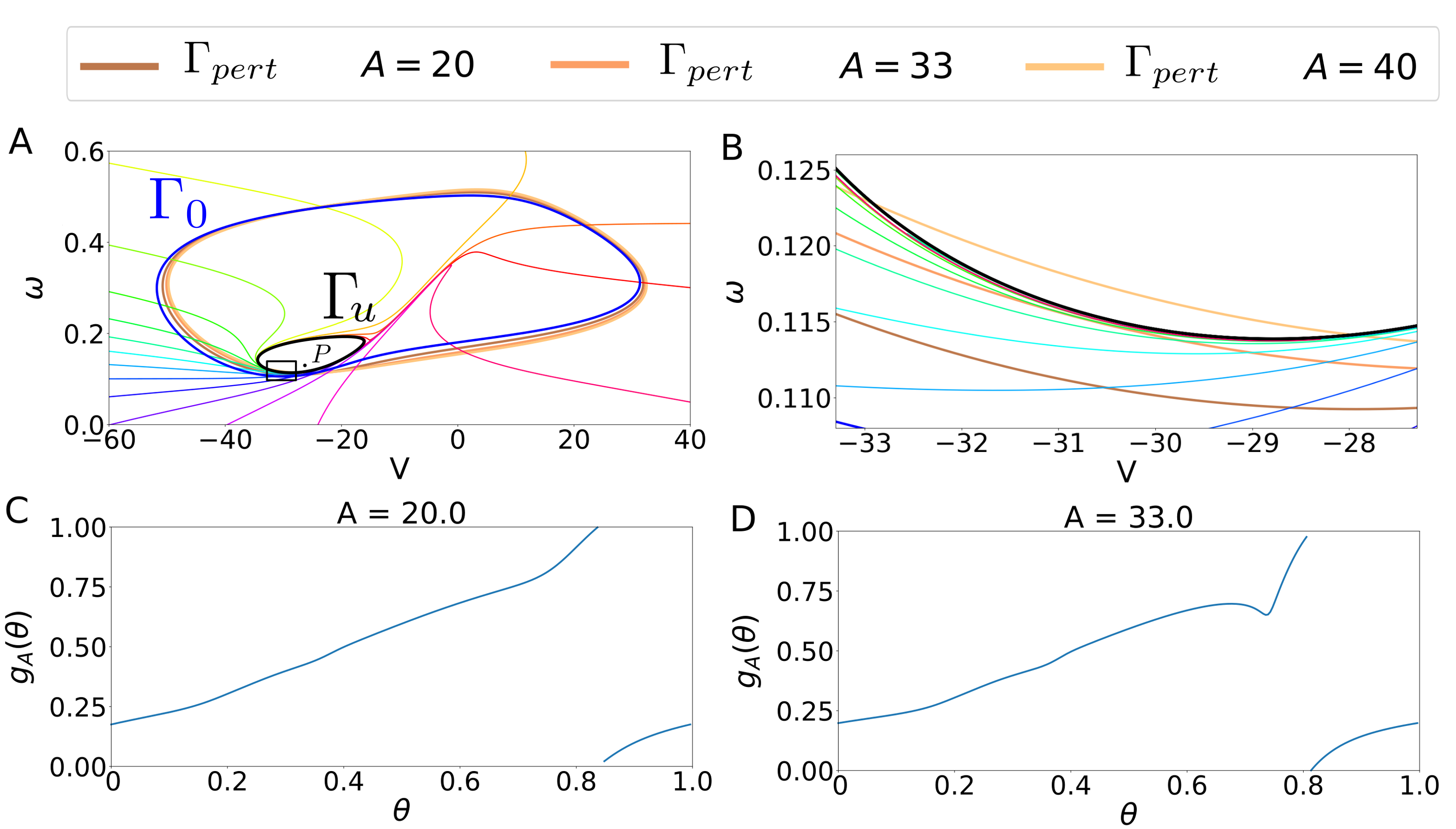}}
\caption{Phase space for ML-Hopf showing the stable limit cycle $\Gamma_0$, its isochrons, and the unstable limit cycle $\Gamma_u$ 
determining the basin of attraction of a stable focus $P$. 
For different values of the amplitude we show (A) the curves $\Gamma_{pert}$ and (B) a zoom close to the isochron tangency. 
Panels C and D show the functions $g_A$ obtained with the modified parameterization method for $A=20$ (C) and $A=33$ (D).
For $A=40$ some points on $\Gamma_{pert}$ intersect the basin of attraction of the stable focus $P$, and the function $g_A$ can not be computed.} 
\label{fig:breakdownHopfML}
\end{figure}

\section{Discussion}\label{sec:discussion}

In this paper we have introduced a new approach to PRCs based on the parameterization method. The main idea of the method is to introduce a periodic perturbation consisting of 
the actual perturbation followed by a relaxation time $T_{rel}$ that repeats periodically. This periodic perturbation allows us to define the corresponding stroboscopic map $F_A$ of the periodically perturbed system. 
The main result of this paper is Theorem~\ref{thm:theorem1}, where we prove the existence of an invariant curve $\Gamma_A$ for the map $F_A$ and we link its internal dynamics $f_A$ with the PRC. 
The proof relies on the parameterization method, which defines an invariance equation for the invariant curve and its internal dynamics \cite{Cabre2005, canadell2014parameterization, haro2006parameterization, haro2007parameterization}, and provides
a numerical algorithm to compute both.

Moreover, Theorem~\ref{thm:theorem1} establishes conditions for the existence of the invariant curve $\Gamma_A$. More precisely, although the range of amplitude values $A$ for which the invariant curve exists can be increased by 
considering a sufficiently large $T_{rel}$, there is a limitation established by the geometry of the isochrons. Indeed, whenever the curve 
$\Gamma_{pert}$ (the displacement of the limit cycle due to the active part of the perturbation) becomes tangent to some isochron of $\Gamma_0$, 
the invariant curve $\Gamma_A$ loses its normally hyperbolic properties and breaks down. Moreover, we explain how the isochrons of the unperturbed system and the perturbation interact to shape the waveform of the PRC as the amplitude of the perturbation increases.

Besides theoretical results, we present some strategies to compute the PRC. In \cite{Haro2016} one can find algorithms that implement a quasi-Newton method to solve the above mentioned invariance equation. 
The method though relies on the existence of an invariant curve. Nevertheless, as the PRC can be computed beyond the breakdown of the curve $\Gamma_A$, 
we have developed a modified numerical method to compute PRCs inspired by the parameterization method. This method solves a modified invariance equation, which avoids the computation of the invariant curve (thus making the 
computations faster), and is able to compute PRCs beyond the breakdown of the curve $\Gamma_A$ and the transition from type 1 to type 0 PRCs.
In addition, this algorithm computes not only PRCs but also ARCs, which provide information about the effects of the perturbation onto the amplitude variables (or alternatively, the displacement away from the limit cycle).
We show examples of the computed ARCs and the relationship between its shape and the transitions experienced by the PRC. 

In order to assess the validity of the method, we have applied it to two models in neuroscience: 
a neural population model (Wilson-Cowan) and a single neuron model (Morris-Lecar). 
Moreover, we have studied both models for values of the parameters near two different bifurcations: Hopf and Saddle Node on Invariant Circle (SNIC) bifurcations. 
Recall that PRCs are classified as type 1 or type 2 according to their shape, and this property is associated to a particular bifurcation:
type 1 PRCs mainly advance phase, and they are related to a SNIC bifurcation, 
while type 2 PRCs can either advance or delay the phase, and they are related to a Hopf bifurcation \cite{ermentrout1996type, oprisan2002influence,smeal2010phase}. 
The numerical examples presented show the evolution of both PRC types for large amplitude values. 
In all the examples considered, the PRCs preserve its type as the amplitude increases but the phase shifts tend to augment.

In this work, we have mainly developed the theory and numerical examples for two-dimensional systems. However, the underlying theorems and numerical algorithms have an straightforward extension for the case $n>2$ that we plan to explore in detail in future work.

\section{Acknowledgments}

This work has been partially funded by the grants MINECO-FEDER MTM2015-65715-P, MDM-2014-0445, PGC2018-098676-B-100 AEI/FEDER/UE, the Catalan grant 2017SGR1049, (GH, AP, TS), the MINECO-FEDER-UE MTM-2015-71509-C2-2-R (GH), 
and the Russian Scientific Foundation Grant 14-41-00044 (TS).
GH acknowledges the RyC project RYC-2014-15866. TS is supported by the Catalan Institution for research and advanced studies via an ICREA academia price 2018.
AP acknowledges the FPI Grant from project MINECO-FEDER-UE MTM2012-31714. We thank C. Bonet for providing us valuable references to prove Theorem~\ref{thm:theorem1} and A. Granados for numeric support.
We also acknowledge the use of the UPC Dynamical Systems group's cluster for research computing
(https://dynamicalsystems.upc.edu/en/computing/).

\bibliographystyle{abbrv}
\bibliography{bibPrcs}

\begin{thebibliography}{10}

\bibitem{BatesLuZeng2008}
P.~W. Bates, K.~Lu, and C.~Zeng.
\newblock Approximately invariant manifolds and global dynamics of spike
  states.
\newblock {\em Inventiones mathematicae}, 174(2):355--433, 2008.

\bibitem{borisyuk1992}
R.~M. Borisyuk and A.~B. Kirillov.
\newblock Bifurcation analysis of a neural network model.
\newblock {\em Biological Cybernetics}, 66(4):319--325, 1992.

\bibitem{buzsaki2006rhythms}
G.~Buzsaki.
\newblock {\em Rhythms of the Brain}.
\newblock Oxford University Press, 2006.

\bibitem{Cabre2005}
X.~Cabr{\'e}, E.~Fontich, and R.~De~La~Llave.
\newblock The parameterization method for invariant manifolds \textsc{III}:
  overview and applications.
\newblock {\em Journal of Differential Equations}, 218(2):444--515, 2005.

\bibitem{canadell2014parameterization}
M.~Canadell and A.~Haro.
\newblock Parameterization method for computing quasi-periodic reducible
  normally hyperbolic invariant tori.
\newblock In {\em Advances in differential equations and applications}, pages
  85--94. Springer, 2014.

\bibitem{canadell2016newton}
M.~Canadell and {\`A}.~Haro.
\newblock A newton-like method for computing normally hyperbolic invariant
  tori.
\newblock In {\em The Parameterization Method for Invariant Manifolds}, pages
  187--238. Springer, 2016.

\bibitem{canavier2010pulse}
C.~C. Canavier and S.~Achuthan.
\newblock Pulse coupled oscillators and the phase resetting curve.
\newblock {\em Mathematical biosciences}, 226(2):77--96, 2010.

\bibitem{castejon2013phase}
O.~Castej{\'o}n, A.~Guillamon, and G.~Huguet.
\newblock Phase-amplitude response functions for transient-state stimuli.
\newblock {\em The Journal of Mathematical Neuroscience}, 3(1):13, 2013.

\bibitem{castelli2015}
R.~Castelli, J.-P. Lessard, and J.~D. Mireles~James.
\newblock Parameterization of invariant manifolds for periodic orbits
  \textsc{I}: Efficient numerics via the floquet normal form.
\newblock {\em SIAM Journal on Applied Dynamical Systems}, 14(1):132--167,
  2015.

\bibitem{ermentrout1996type}
B.~Ermentrout.
\newblock Type \textsc{I} membranes, phase resetting curves, and synchrony.
\newblock {\em Neural computation}, 8(5):979--1001, 1996.

\bibitem{ErmentroutTerman2010}
B.~Ermentrout and D.~Terman.
\newblock {\em {Mathematical foundations of neuroscience.}}
\newblock New York : Springer, 2010.

\bibitem{ermentrout1991}
G.~B. Ermentrout and N.~Kopell.
\newblock Multiple pulse interactions and averaging in systems of coupled
  neural oscillators.
\newblock {\em Journal of Mathematical Biology}, 29(3):195--217, 1991.

\bibitem{Fenichel71}
N.~Fenichel.
\newblock Persistence and smoothness of invariant manifolds for flows.
\newblock {\em Indiana Univ. Math. J.}, 21:193--226, 1971/1972.

\bibitem{Fenichel74}
N.~Fenichel.
\newblock Asymptotic stability with rate conditions.
\newblock {\em Indiana Univ. Math. J.}, 23:1109--1137, 1973/74.

\bibitem{glass1988clocks}
L.~Glass and M.~C. Mackey.
\newblock {\em From clocks to chaos: the rhythms of life}.
\newblock Princeton University Press, 1988.

\bibitem{glass1984discontinuities}
L.~Glass and A.~.~T. Winfree.
\newblock Discontinuities in phase-resetting experiments.
\newblock {\em American Journal of Physiology-Regulatory, Integrative and
  Comparative Physiology}, 246(2):R251--R258, 1984.

\bibitem{Guckenheimer74}
J.~Guckenheimer.
\newblock Isochrons and phaseless sets.
\newblock {\em Journal of Mathematical Biology}, 1(3):259--273, 1975.

\bibitem{Guillamon2009}
A.~Guillamon and G.~Huguet.
\newblock A computational and geometric approach to phase resetting curves and
  surfaces.
\newblock {\em SIAM Journal on Applied Dynamical Systems}, 8(3):1005--1042,
  2009.

\bibitem{Haro2016}
{\`A}.~Haro, M.~Canadell, J.-L. Figueras, A.~Luque, and J.-M. Mondelo.
\newblock {\em The Parameterization Method for Invariant Manifolds}.
\newblock Springer, 2016.

\bibitem{haroDeLaLlave}
A.~Haro and R.~de~la Llave.
\newblock Persistence of normally hyperbolic invariant manifolds, internal
  communication.

\bibitem{haro2006parameterization}
A.~Haro and R.~de~la Llave.
\newblock A parameterization method for the computation of invariant tori and
  their whiskers in quasi-periodic maps: numerical algorithms.
\newblock {\em Discrete and Continuous Dynamical Systems Series B}, 6(6):1261,
  2006.

\bibitem{haro2007parameterization}
A.~Haro and R.~de~La~Llave.
\newblock A parameterization method for the computation of invariant tori and
  their whiskers in quasi-periodic maps: explorations and mechanisms for the
  breakdown of hyperbolicity.
\newblock {\em SIAM Journal on Applied Dynamical Systems}, 6(1):142, 2007.

\bibitem{HirschPS77}
M.~Hirsch, C.~Pugh, and M.~Shub.
\newblock {\em Invariant manifolds}, volume 583 of {\em Lecture Notes in Math.}
\newblock Springer-Verlag, Berlin, 1977.

\bibitem{hoppensteadt2012weakly}
F.~C. Hoppensteadt and E.~M. Izhikevich.
\newblock {\em Weakly connected neural networks}, volume 126.
\newblock Springer Science \& Business Media, 2012.

\bibitem{huguet2013}
G.~Huguet and R.~de~la Llave.
\newblock Computation of limit cycles and their isochrons: fast algorithms and
  their convergence.
\newblock {\em SIAM Journal on Applied Dynamical Systems}, 12(4):1763--1802,
  2013.

\bibitem{mauroy2012}
A.~Mauroy and I.~Mezi{\'c}.
\newblock On the use of fourier averages to compute the global isochrons of
  (quasi) periodic dynamics.
\newblock {\em Chaos: An Interdisciplinary Journal of Nonlinear Science},
  22(3):033112, 2012.

\bibitem{morris1981}
C.~Morris and H.~Lecar.
\newblock Voltage oscillations in the barnacle giant muscle fiber.
\newblock {\em Biophysical journal}, 35(1):193--213, 1981.

\bibitem{NippS92}
K.~Nipp and D.~Stoffer.
\newblock Attractive invariant mainfolds for maps: existence, smoothness and
  continuous dependence on the map.
\newblock In {\em Research report/Seminar f{\"u}r Angewandte Mathematik},
  volume 1992. Eidgen{\"o}ssische Technische Hochschule, Seminar f{\"u}r
  Angewandte Mathematik, 1992.

\bibitem{NippS2013}
K.~Nipp and D.~Stoffer.
\newblock Invariant manifolds in discrete and continuous dynamical systems.
\newblock {\em EMS tracts in mathematics}, 21, 2013.

\bibitem{oprisan2002influence}
S.~A. Oprisan and C.~C. Canavier.
\newblock The influence of limit cycle topology on the phase resetting curve.
\newblock {\em Neural computation}, 14(5):1027--1057, 2002.

\bibitem{osinga2010continuation}
H.~M. Osinga and J.~Moehlis.
\newblock Continuation-based computation of global isochrons.
\newblock {\em SIAM Journal on Applied Dynamical Systems}, 9(4):1201--1228,
  2010.

\bibitem{perez2018computation}
A.~P{\'e}rez-Cervera, G.~Huguet, and T.~M-Seara.
\newblock {\em Computation of Invariant Curves in the Analysis of Periodically
  Forced Neural Oscillators}.
\newblock Springer, 2018.

\bibitem{rinzel1989analysis}
J.~Rinzel and G.~B. Ermentrout.
\newblock {\em Analysis of neural excitability and oscillations}.
\newblock Cambridge, MA: MIT Press, 1989.

\bibitem{RinzelH13}
J.~Rinzel and G.~Huguet.
\newblock Nonlinear dynamics of neuronal excitability, oscillations, and
  coincidence detection.
\newblock {\em Comm. Pure Appl. Math.}, 66(9):1464--1494, 2013.

\bibitem{schultheiss2011phase}
N.~W. Schultheiss, A.~A. Prinz, and R.~J. Butera.
\newblock {\em Phase response curves in neuroscience: theory, experiment, and
  analysis}.
\newblock Springer Science \& Business Media, 2011.

\bibitem{smeal2010phase}
R.~M. Smeal, G.~B. Ermentrout, and J.~A. White.
\newblock Phase-response curves and synchronized neural networks.
\newblock {\em Philosophical Transactions of the Royal Society of London B:
  Biological Sciences}, 365(1551):2407--2422, 2010.

\bibitem{wedgwood2013phase}
K.~C. Wedgwood, K.~K. Lin, R.~Thul, and S.~Coombes.
\newblock Phase-amplitude descriptions of neural oscillator models.
\newblock {\em The Journal of Mathematical Neuroscience}, 3(1):2, 2013.

\bibitem{wilsonermentrout18}
D.~Wilson and B.~Ermentrout.
\newblock Greater accuracy and broadened applicability of phase reduction using
  isostable coordinates.
\newblock {\em Journal of mathematical biology}, 76(1-2):37--66, 2018.

\bibitem{wilson2015extending}
D.~Wilson and J.~Moehlis.
\newblock Extending phase reduction to excitable media: theory and
  applications.
\newblock {\em SIAM Review}, 57(2):201--222, 2015.

\bibitem{moehliswilsonpre2016}
D.~Wilson and J.~Moehlis.
\newblock Isostable reduction of periodic orbits.
\newblock {\em Physical Review E}, 94(5):052213, 2016.

\bibitem{wilson1972}
H.~R. Wilson and J.~D. Cowan.
\newblock Excitatory and inhibitory interactions in localized populations of
  model neurons.
\newblock {\em Biophysical journal}, 12(1):1--24, 1972.

\bibitem{winfree1974patterns}
A.~Winfree.
\newblock Patterns of phase compromise in biological cycles.
\newblock {\em Journal of Mathematical Biology}, 1(1):73--93, 1974.

\end{thebibliography}

\appendix

\section{Appendix. Numerical algorithms}\label{sec:num_alg}

In this Section we review the numerical algorithms introduced in \cite{Haro2016} to compute the
parameterization of an invariant curve $\Gamma_A$ of a given map $F:=F_A$ as well as the dynamics on
the curve, i.e. $f:=F|_{\Gamma_A}$. We present the algorithms in a format that is ready for numerical
implementation and we refer the reader to \cite{Haro2016, perez2018computation} for more details.

The method to compute the invariant curve consists in looking for a map
$K:\mathbb{T} \rightarrow \mathbb{R}^2$ and a scalar function $f:\mathbb{T} \rightarrow \mathbb{T}$
satisfying the invariance equation
\begin{equation}\label{eq:invarianceEq}
F(K(\theta)) = K(f(\theta)).
\end{equation}

In order to solve equation \eqref{eq:invarianceEq} by means of a Newton-like method, one needs to
compute alongside the invariant normal bundle of $K$, denoted by $N$, and its
linearised dynamics $\Lambda_N$, which satisfy the invariance equation
\begin{equation}\label{eq:bundlesCompact}
DF(K(\theta))N(\theta) = DK(f(\theta))\Lambda_N(\theta).
\end{equation}

Thus, the main algorithm provides a Newton method to solve equations \eqref{eq:invarianceEq} and
\eqref{eq:bundlesCompact} altogether. More precisely, at step $i$ of the Newton method one computes successive
approximations $K^i$,$f^i$, $N^i$ and $\Lambda_N^i$ of $K$, $f$, $N$ and $\Lambda_N$, respectively. The
algorithm is stated as follows:

\begin{algorithm}\label{alg:algorithm1}
\textbf{Main algorithm to solve equations \eqref{eq:invarianceEq}-\eqref{eq:bundlesCompact}}.
Given $K(\theta)$, $f(\theta)$, $f^{-1}(\theta)$, $N(\theta)$ and $\Lambda_N(\theta)$,
approximate solutions of equations \eqref{eq:invarianceEq} and \eqref{eq:bundlesCompact},
perform the following operations:
\begin{enumerate}
\item 
Compute the corrections $\Delta K(\theta)$ and $\Delta f(\theta)$ by using Algorithm~\ref{alg:algorithm2}.
\item 
Update $K(\theta) \leftarrow K(\theta) + \Delta K(\theta) \quad \quad  f(\theta) \leftarrow f(\theta) + \Delta f(\theta)$.
\item 
Compute the inverse function $f^{-1}(\theta)$ using Algorithm \ref{alg:algorithm3}.
\item 
Compute $DK(\theta)$ and $Df(\theta)$.
\item 
Compute the corrections $\Delta N(\theta)$ and $\Delta_N (\theta)$ by using Algorithm \ref{alg:algorithm5}.
\item 
Update $N(\theta) \leftarrow N(\theta) + \Delta N(\theta) \quad \quad  \Lambda_N(\theta) \leftarrow \Lambda_N(\theta) + \Delta_N (\theta)$.
\item 
Compute $E = F \circ K - K \circ f$ and repeat steps 1-6 until $E$ is smaller than the established tolerance.
\end{enumerate}
\end{algorithm}

Next, we provide the sub-algorithms for algorithm \ref{alg:algorithm1}.

\begin{algorithm}\label{alg:algorithm2}
\textbf{Correction of the approximate invariant curve}.
Given $K(\theta)$, $f(\theta)$, $f^{-1}(\theta)$ $N(\theta)$ and $\Lambda_N(\theta)$, approximate solutions of equations 
\eqref{eq:invarianceEq} and \eqref{eq:bundlesCompact}, perform the following operations:
\begin{enumerate}
\item 
Compute $E(\theta) = F(K(\theta)) - K(f(\theta))$.
\item 
Compute $P(f(\theta))= \Big(DK(f(\theta)) | N (f(\theta))\Big)$.
\item 
Compute $\eta(\theta) = \begin{pmatrix} \eta_T(\theta) \\ \eta_N(\theta) \end{pmatrix}  = -(P(f(\theta)))^{-1}E(\theta)$.
\item 
Compute $f^{-1}(\theta)$ using algorithm \ref{alg:algorithm3}.
\item 
Solve for $\xi$ the equation $\eta_N(f^{-1}(\theta))= \Lambda_N(f^{-1}(\theta))\xi(f^{-1}(\theta)) - \xi(\theta)$ by using algorithm \ref{alg:algorithm4}.
\item 
Set $\Delta f(\theta) \leftarrow - \eta_T(\theta)$.
\item 
Set $\Delta K(\theta) \leftarrow N(\theta)\xi(\theta)$.
\end{enumerate}
\end{algorithm}

\begin{algorithm}\label{alg:algorithm3}
\textbf{Refine $f^{-1}(\theta)$}. Given a function $f(\theta)$, its derivative $Df(\theta)$ and an approximate inverse function $f^{-1}(\theta)$, 
perform the following operations:
\begin{enumerate}
\item Compute $e(\theta) = f(f^{-1}(\theta)) - \theta$.
\item Compute $\Delta f^{-1}(\theta) = - \frac{e(\theta)}{Df(f^{-1}(\theta))}$.
\item Set $f^{-1}(\theta) \leftarrow f^{-1}(\theta) + \Delta f^{-1}(\theta)$.
\item Repeat steps 1-3 until $e(\theta)$ is smaller than a fixed tolerance.
\end{enumerate}
\end{algorithm}

\begin{algorithm}\label{alg:algorithm4}
\textbf{Solution of a fixed point equation}.
Given an equation of the form $B(\theta) = A(\theta)\eta(g(\theta)) - \eta(\theta)$ with $A$,
$B$, $g$ known and $\Vert A \Vert < 1$,
perform the following operations:
\begin{enumerate}
\item Set $\eta(\theta) \leftarrow  B(\theta)$.
\item Compute $\eta(g(\theta))$.
\item Set $\eta(\theta) \leftarrow A(\theta)\eta(g(\theta)) + \eta(\theta)$.
\item Repeat steps 2 and 3 until $|A(\theta)\eta(g(\theta))|$ is smaller than the established tolerance.
\end{enumerate}
\end{algorithm}

\begin{algorithm}\label{alg:algorithm5}
\textbf{Correction of the stable normal bundle}. Given $K(\theta)$, $f(\theta)$, $N(\theta)$ and $\Lambda_N(\theta)$,  approximate solutions of equations \eqref{eq:invarianceEq} and \eqref{eq:bundlesCompact}, perform the following operations:
\begin{enumerate}
\item Compute $E_N(\theta) = DF(K(\theta))N(\theta) - \Lambda_N(\theta)N(f(\theta))$.
\item Compute $P(f(\theta)) = (DK(f(\theta)) \quad N(f(\theta)))$.
\item Compute $\zeta(\theta) = \begin{pmatrix} \zeta_T(\theta) \\ \zeta_N(\theta) \end{pmatrix}  = -(P(f(\theta)))^{-1}E_N(\theta)$.
\item Solve for $Q$ the equation $Df^{-1}(\theta) \zeta_T(\theta)=Df^{-1}(\theta) \Lambda_N(\theta)Q(f(\theta)) - Q(\theta)$ by using algorithm \ref{alg:algorithm4}.
\item Set $\Delta_N (\theta) \leftarrow \zeta_N(\theta)$.
\item Set $\Delta N(\theta) \leftarrow DK(\theta)Q(\theta)$.
\end{enumerate}

\end{algorithm}

\begin{remark}
	Since our functions are defined on $\mathbb{T}$ we will use Fourier
	series to compute the derivatives and composition of functions.
\end{remark}

Of course, the main algorithm requires the knowledge of an approximate solution for
equations~\eqref{eq:invarianceEq}-\eqref{eq:bundlesCompact}. In our case, we can always use the
limit cycle of the unperturbed system as an approximate solution for the invariant curve. However, for the normal bundle we cannot use the one obtained from the unperturbed
limit cycle $\Gamma_0$. The following algorithm provides an initial seed for algorithm \ref{alg:algorithm1}:

\begin{algorithm}\label{alg:algorithm0}
\textbf{Computation of initial seeds}. Given a planar vector field $\dot{x} = X(x)$, having an attracting limit cycle $\gamma(t)$ of period $T$, perform the following operations:
\begin{enumerate}
\item Compute the fundamental matrix $\Phi(t)$ of the first variational equation along the periodic orbit.
\item Obtain the characteristic multiplier $\lambda \neq 1$ and its associated eigenvector $v_{\lambda}$ from $\Phi(T)$.
\item Set $N(\theta) \leftarrow \Phi(T\theta)v_{\lambda}e^{\lambda \theta}$.
\item Set $\Lambda_N(\theta) \leftarrow e^{\frac{\lambda T'}{T}}$.
\item Set $K(\theta) \leftarrow \gamma(T\theta), \quad \quad DK(\theta) = TX(\gamma(T\theta))$.
\item Set $f(\theta) \leftarrow \theta + \frac{T'}{T}, \quad \quad f^{-1}(\theta) \leftarrow \theta - \frac{T'}{T}, \quad \quad Df(\theta) = 1$.
\end{enumerate}
\end{algorithm}

Given a family of maps $F_A$ such that the solution for $F_0$ is known (see algorithm~\ref{alg:algorithm0}), it is standard to set a continuation scheme to compute the solutions for the
other values of $A$. Thus, assuming that the solution for $A=A^{*}$ is known, one can take this solution as an initial seed to solve the equations
for $F_{A^*+h}$, with $h$ small, using the Newton-like method described in algorithm~\ref{alg:algorithm1}.
However, one can perform an extra step to refine the initial seed values $K_{A^ {*}}$ and $f_{A^*}$ for $F_{A^{*}+h}$, described in
algorithm~\ref{alg:algorithm6}.

\begin{algorithm}\label{alg:algorithm6}
\textbf{Refine an initial seed}. Given $K_A(\theta)$, $f_A(\theta)$, $N_A(\theta)$ and $\Lambda_{N,A}(\theta)$, solutions of
equations \eqref{eq:invarianceEq}-\eqref{eq:bundlesCompact} for $F=F_A$, perform the following operations:
\begin{enumerate}
\item Compute $E(\theta) = \frac{\partial F_A}{\partial A}(K_A(\theta))$.
\item Compute $\eta(\theta) = \begin{pmatrix} \eta_T(\theta) \\ \eta_N(\theta) \end{pmatrix}  = -(P(f_A(\theta)))^{-1}E(\theta)$.
\item Solve for $\xi$ the equation $\xi(\theta) = \Lambda_N(f^{-1}_A(\theta))\xi(f^{-1}_A(\theta)) - \eta_N(f^{-1}_A(\theta))$ by using algorithm \ref{alg:algorithm4}.
\item Set $K_{A+h}(\theta) \leftarrow K_A(\theta) + N_A(\theta)\xi(\theta) \cdot h, \quad \quad f_{A+h}(\theta) \leftarrow f_A(\theta) - \eta_T(\theta) \cdot h$.
\end{enumerate}
\end{algorithm}

\begin{remark}
The term $\frac{\partial F}{\partial A}(K_A(\theta))$ is computed using variational equations with respect to the amplitude.
\end{remark}

The numerical continuation scheme for our problem is described in the following algorithm:

\begin{algorithm}\label{alg:continuation}
\textbf{Numerical continuation.} Consider a family of maps $F_A$, such that $F_0$ is the time-$T$ map of a planar autonomous system having a hyperbolic attracting limit cycle of
period $T$. Perform the following operations:
\begin{enumerate}
\item Compute solutions $K_0(\theta)$, $f_0(\theta)$, $N_0(\theta)$ and $\Lambda_{N,0}(\theta)$ of equations \eqref{eq:invarianceEq}-\eqref{eq:bundlesCompact} for $F=F_0$ using algorithm~\ref{alg:algorithm0}.
\item Set $A=0$.
\item Using $K_A$, $f_A$, $N_A$, $\Lambda_{N,A}$ compute an initial seed for $F_{A+h}$ using algorithm~\ref{alg:algorithm6}
\item Find solutions of equations \eqref{eq:invarianceEq}-\eqref{eq:bundlesCompact} for $F=F_{A+h}$ using algorithm~\ref{alg:algorithm1}.
\item Set $A \leftarrow A+h$.
\item Repeat steps 3-5 until $A$ reaches the desired value.
\end{enumerate}
\end{algorithm}

\end{document}